%% file: template-jems.tex
\documentclass{article}
\usepackage[journal=APDE,lang=british]{ems-journal-jems} 

\usepackage{textcomp}
\usepackage{mathrsfs}

\usepackage{chngcntr}
\usepackage{apptools}
\usepackage{verbatim}
\usepackage{tikz-cd}

\newtheorem{theorem}{Theorem}
\newtheorem{proposition}{Proposition}
\newtheorem{lemma}{Lemma}

\theoremstyle{definition}

\newtheorem{conjecture}{Conjecture}

\theoremstyle{remark}
\newtheorem{remark}{Remark}
\newtheorem{corollary}{Corollary}

\numberwithin{equation}{section}

\newcommand\restr[2]{{
		\left.\kern-\nulldelimiterspace 
		#1 
		\vphantom{\big|} 
		\right|_{#2} 
}}

\DeclareRobustCommand{\rchi}{{\mathpalette\irchi\relax}}
\newcommand{\irchi}[2]{\raisebox{\depth}{$#1\chi$}} 
\numberwithin{equation}{section}

\numberwithin{theorem}{section}
\numberwithin{proposition}{section}
\numberwithin{lemma}{section}
\numberwithin{remark}{section}
\numberwithin{exercise}{section}
\numberwithin{corollary}{section}
\numberwithin{conjecture}{section}
\numberwithin{problem}{section}
\newcounter{desccount}

\begin{document}

\title{Frequency conditions for the global stability of nonlinear delay equations with several equilibria}
\titlemark{Global stability of nonlinear delay equations with several equilibria}


\emsauthor{1}{
	\givenname{Mikhail}
	\surname{Anikushin}
	\orcid{0000-0002-7781-8551}}{M.M.~Anikushin}
\emsauthor{2}{
	\givenname{Andrey}
	\surname{Romanov}
	\orcid{0009-0008-4886-2351}}{A.O.~Romanov}

\Emsaffil{1}{
	\pretext{}
	\department{Department of Applied Cybernetics}
	\organisation{Faculty of Mathematics and Mechanics, St Petersburg University}
	\rorid{023znxa73}
	\address{Universitetskiy prospekt 28}
	\zip{198504}
	\city{Peterhof}
	\country{Russia}
	\posttext{}
	\affemail{demolishka@gmail.com}
}
\Emsaffil{2}{
	\pretext{}
	\department{Department of Applied Cybernetics}
	\organisation{Faculty of Mathematics and Mechanics, St Petersburg University}
	\rorid{023znxa73}
	\address{Universitetskiy prospekt 28}
	\zip{198504}
	\city{Peterhof}
	\country{Russia}
	\posttext{}
	\affemail{romanov.andrey.twai@gmail.com}
}
	

	


\classification[34K20, 37L30, 37L45]{37L15}

\begin{abstract}
	In our adjacent work, we developed a spectral comparison principle for compound cocycles generated by delay equations. It allows to derive frequency inequalities for the uniform exponential stability of such cocycles by means of their comparison with stationary problems. Such inequalities are hard to verify purely analytically, and in this work we develop approximation schemes to verify some of the arising frequency inequalities. Besides some general theoretical results, in applications we stick to the case of scalar equations. By means of the Suarez--Schopf delayed oscillator and the Mackey--Glass equations, we demonstrate applications of the theory to reveal regions in the space of parameters where the absence of closed invariant contours can be guaranteed. Since the frequency inequalities are robust, so close systems also satisfy them, we expect the method to actually imply the global stability, as in known finite-dimensional results utilizing variants of the closing lemma, which is still awaiting developments in infinite dimensions.
\end{abstract}

\maketitle
\tableofcontents

\input{Introduction}
\input{Preliminaries}
\input{GeneralTheory}
\input{VerificationOfFreqIneq}
\section{Examples}
\input{SuarezSchopfOscill}
\input{MackeyGlassEqs}




\section*{Acknowledgments}
The authors are grateful to the Associate Editor and the anonymous reviewers for their valuable comments and suggestions, which led to significant improvements in the readability and exposition.

\begin{funding}
The reported study was funded by the Russian Science Foundation (Project 25-11-00147).
\end{funding}

\section*{Data availability}
The data that support the findings of this study can be generated using the scripts in the repository:\\
\centerline{\url{https://gitlab.com/romanov.andrey/freq-criterion-delay}}

\section*{Conflict of interest}
The authors declare that they have no conflict of interest.


\end{document}

%% file: Introduction.tex
\section{Introduction}

\subsection{Historical perspective: generalized Bendixson criterion in global stability problems}
In the seminal paper \cite{Smith1986HD}, Smith presented a generalization of the Bendixson criterion for ordinary differential equations (ODEs) in $\mathbb{R}^{n}$. His abstract conditions were concerned with a continuous mapping $F$ of a bounded simply connected domain $\mathcal{D}$ in $\mathbb{R}^{n}$ such that $F(\mathcal{D})$ lies in a compact subset of $\mathcal{D}$. Then he proved that there are no closed invariant contours on which $F$ is bijective\footnote{This (bijectivity) was implicitly used in \cite{Smith1986HD}, and the explicit clarification was given in \cite{LiMuldowney1995LowBounds}.} provided that the Hausdorff dimension of the maximal compact invariant subset $\mathcal{A}$ (attractor) is strictly less than $2$.

In applications to ODEs, $F$ is given by the time-$t$ mapping $\varphi^{t}$ of the semiflow $\varphi$ generated by an ODE (vector field), and $\bar{\mathcal{D}}$ (the closure of $\mathcal{D}$ in $\mathbb{R}^{n}$) is a positively invariant\footnote{In the original work \cite{Smith1986HD}, it is required that $\bar{\mathcal{D}}$ is mapped into $\mathcal{D}$ under the semiflow. However, one can weaken the condition to just the positive invariance, i.e., $\varphi^{t}(\bar{\mathcal{D}}) \subset \bar{\mathcal{D}}$ for all $t \geq 0$, and the existence of an attractor in the interior, if the local closing lemma is used (see \cite{LiMuldowney1996SIAMGlobStab}).} with respect to $\varphi$ closed bounded region such that $\mathcal{A} \coloneq \bigcap_{t \geq 0} \varphi^{t}(\bar{\mathcal{D}})$ is a compact subset of $\mathcal{D}$. Then, using the Liouville trace formula, Smith derived the condition 
\begin{equation}
	\label{EQ: SmithContractionCondition}
	\alpha_{1}(q) + \alpha_{2}(q) < 0 \qquad \text{for any} \ q \in \bar{\mathcal{D}},
\end{equation}
where $\alpha_{1}(q)$ and $\alpha_{2}(q)$ are the first and the second largest eigenvalues of the additively symmetrized Jacobian matrix of the vector field at $q$. This condition guarantees the contraction of areas under the action of the differential $d_{q}\varphi^{t}$ of $\varphi^{t}$ uniformly in $q \in \mathcal{A}$ as $t \to +\infty$. This implies that the Hausdorff dimension of $\mathcal{A}$ is strictly less than $2$ (see \cite{ChepyzhovIlyin2004, KuzReit2020, Temam1997}) and, consequently, the abstract Bendixson criterion can be applied. In \cite{Smith1986HD}, it was also noted that \eqref{EQ: SmithContractionCondition} is a robust condition, i.e., $C^{1}$-close systems also satisfy it. This allowed Smith to utilize Pugh's closing lemma and deduce from \eqref{EQ: SmithContractionCondition} that any nonwandering point in $\bar{\mathcal{D}}$ is an equilibrium, and, moreover, any trajectory in $\bar{\mathcal{D}}$ necessarily converges to an equilibrium\footnote{If the stationary set is finite, then this last conclusion is obvious, since the $\omega$-limit set of any point is connected and consists of nonwandering points.}.

Later, condition \eqref{EQ: SmithContractionCondition} was sharpened by Li and Muldowney \cite{LiMuldowney1996SIAMGlobStab, LiMuldowney1995Conv} with the aid of compound matrices and Lozinskii (logarithmic) norms. They also widely extended geometric ideas of Smith to semiflows in Banach spaces \cite{LiMuldowney1995LowBounds} and, in particular, established a generalized Bendixson criterion for such semiflows.

Moreover, Leonov and Boichenko \cite{LeonovBoi1992} gave another sharper conditions via Lyapunov-like functions. Now this approach is known as the Leonov method \cite{Kuznetsov2016}. In the monograph of Kuznetsov and Reitmann \cite{KuzReit2020}, the method is combined with logarithmic norms.

In fact, all the mentioned results are implicitly concerned with showing that
\begin{equation}
	\label{EQ: LyapExpContrCond}
	\lambda_{1}(\Xi) + \lambda_{2}(\Xi) < 0,
\end{equation}
where $\lambda_{1}(\Xi)$ and $\lambda_{2}(\Xi)$ are the first and the second uniform Lyapunov exponents of the derivative cocycle $\Xi$ of $\varphi$ in $\bar{\mathcal{D}}$, i.e., we have $\Xi^{t}(q,\cdot) \coloneq d_{q}\varphi^{t}$ for $q \in \bar{\mathcal{D}}$ and $t \geq 0$ in terms of Section \ref{SUBSEC: CocyclesAndSemiflowsBrief}. In the terminology of \cite{Anikushin2023LyapExp}, \eqref{EQ: LyapExpContrCond} is obtained by computing the infinitesimal growth exponents for the twofold antisymmetric multiplicative compound $\Xi_{2}$ of $\Xi$ in an adapted metric. Then the so-called maximization procedure (or the averaging procedure in the case of \cite{LiMuldowney1996SIAMGlobStab}) is applied to estimate the quantity $\lambda_{1}(\Xi) + \lambda_{2}(\Xi)$ from above.

Such a sum as $\lambda_{1}(\Xi) + \lambda_{2}(\Xi)$, being the largest uniform Lyapunov exponent of $\Xi_{2}$, is upper semicontinuous with respect to $\Xi$ under natural conditions. This is the robustness that is required for the application of the closing lemma. Moreover, the Lyapunov exponents depend only on the attractor $\mathcal{A}$, and not on the domain $\mathcal{D}$ that encloses it, see \cite{Anikushin2023LyapExp}.

In \cite{Anikushin2023LyapExp}, it is shown (under some natural conditions) that one can always adapt the metric (not necessarily coercive) on the twofold exterior power such that the maximization procedure will produce quantities arbitrarily close to $\lambda_{1}(\Xi) + \lambda_{2}(\Xi)$. Thus, on the geometric level, there are no ``autonomous convergence theorems'' (in plural, as it is used in \cite{LiMuldowney1996SIAMGlobStab}), but rather only one abstract statement concerned with \eqref{EQ: LyapExpContrCond}. Diversity arises in applications due to the use of particular metrics for specific problems in order to verify \eqref{EQ: LyapExpContrCond}. Of course, for efficient applications, this approach demands constructing adapted metrics.

For invertible finite-dimensional systems, there is a more delicate result related to the existence of adapted metrics due to Kawan, Matveev, and Pogromsky \cite{KawanPogromsky2021}. Motivated by such existence theorems, recent advances concern the numerical search for adapted metrics using subgradient optimization (see Kawan, Hafstein, and Giesl \cite{KawanHafsteinGiesl2021}) or nonlinear constrained optimization (see our work \cite{AnikushinRomanov2025RobustEstimates}). We refer to \cite{Anikushin2023LyapExp,AnikushinRomanov2025RobustEstimates} for more discussions.

From the perspective of the analysis of systems depending on parameters, it is convenient to call a semiflow \textit{globally stable} if any of its trajectories tends to the stationary set \cite{Kuzetal2020Lorenz}. This term covers multistable systems (which are more common) and emphasizes the global character of the problem. In the space of parameters, the boundary of global stability distinguishes the regions with simple and complex behavior.

From \eqref{EQ: LyapExpContrCond} we immediately see limitations of the method. Namely, the first $\lambda_{1}$ and the second $\lambda_{2}$ (as the real part decreases) eigenvalues at any equilibrium from $\mathcal{D}$ must satisfy $\operatorname{Re}\lambda_{1} + \operatorname{Re}\lambda_{2} < 0$. There are systems where the boundary of global stability is determined by local bifurcations of equilibria (usually, the Andronov--Hopf bifurcation). In this case the boundary is called \textit{trivial} in the terminology of \cite{Kuzetal2020Lorenz}. For such systems, the criterion based on \eqref{EQ: LyapExpContrCond} has a prospect (with the use of adapted metrics) to reveal the entire region of global stability, provided that there are no saddles with $\operatorname{Re}\lambda_{1} + \operatorname{Re}\lambda_{2} \geq 0$ until the bifurcation occurs. However, there are systems where the boundary of global stability is determined by nonlocal bifurcations (such as the Lorenz system), and in this case the boundary is called \textit{hidden}. For such systems, applications of analytical methods may be complicated. Looking ahead, we note that the Suarez--Schopf model (see Section \ref{SEC: SuarezSchopfCompoundStab}) has a hidden boundary of global stability, while the Mackey--Glass model (see Section \ref{SEC: MackeyGlassCompoundStab}) is conjectured to have a trivial boundary.

The above considerations can be illustrated by means of the Lorenz system, for which the conditions given by Smith \cite{Smith1986HD} were improved in \cite{LeonovBoi1992}. Moreover, by developing the Leonov method, Leonov et al. \cite{LeoKuzKorKusakin2016} and Kuznetsov et al. \cite{Kuzetal2020Lorenz} derived an exact analytical formula for the Lyapunov dimension of the global attractor in the Lorenz system for all standard parameters\footnote{More precisely, \cite{LeoKuzKorKusakin2016} establishes an exact formula for some parameters, and \cite{Kuzetal2020Lorenz} proves that for other parameters the system is globally stable (in this case the Lyapunov dimension is also given by the local Lyapunov dimension at the zero equilibrium).}. As a consequence, there is an analytical description of the region where \eqref{EQ: LyapExpContrCond} is satisfied, and this is the maximum that can be achieved via the generalized Bendixson criterion.

We refer to the survey of Zelik \cite{ZelikAttractors2022} and our paper \cite{Anikushin2023LyapExp} for more discussions on dimension estimates.

\subsection{Contribution of the present work}
This paper is concerned with applications of the generalized Bendixson criterion developed by Li and Muldowney \cite{LiMuldowney1995LowBounds} to delay equations in $\mathbb{R}^{n}$ by verifying \eqref{EQ: LyapExpContrCond} for the corresponding derivative cocycles. This is related to the problem of obtaining effective dimension estimates for such equations that is rarely addressed in the known literature (see \cite{Anikushin2023LyapExp, Anikushin2022Semigroups, AnikushinRomanov2024EffEst} for discussions). To the best of our knowledge, the first satisfactory results in this direction were obtained in \cite{Anikushin2023LyapExp}. In particular, dimension estimates for the global attractor in the Mackey--Glass equations, which seem to be asymptotically sharp (i.e., up to a constant) as the delay value tends to infinity, are obtained therein. Although such estimates provide nontrivial regions where \eqref{EQ: LyapExpContrCond} holds, numerical analysis indicates much larger regions of global stability.

Here we follow the approach developed in the adjacent paper \cite{Anikushin2023Comp}, where a spectral comparison principle for compound cocycles in Hilbert spaces generated by delay equations is established. This principle treats the compound cocycle as a nonautonomous perturbation of a $C_{0}$-semigroup and provides frequency conditions (inequalities) to guarantee that certain properties concerned with spectral dichotomies of the semigroup will be preserved under such a perturbation. Here the perturbation is described through the so-called quadratic constraints, and the perturbation problem is posed in the context of an appropriate infinite-horizon quadratic regulator problem, which, in its turn, is resolved via the frequency theorem developed by the first author in \cite{Anikushin2020FreqDelay} (see also \cite{Anikushin2020FreqParab}). In particular, the principle provides frequency conditions for the uniform exponential stability of compound cocycles. This is clearly related to the initial problem, since, in terms of such cocycles, \eqref {EQ: LyapExpContrCond} means that the twofold compound cocycle is uniformly exponentially stable. On the geometric level, the frequency conditions guarantee the existence of an adapted metric given by a positive-definite quadratic functional on the exterior power, see Theorem \ref{TH: QuadraticFunctionalDelayCompoundTheorem}. Although it is not necessarily coercive, its relation with the dynamics allows us to obtain the required bound, see Corollary \ref{COR: DelayCompoundUniformExponentialStability}. We give a brief exposition of this theory in Section \ref{SEC: StabilityOfDelayCompounds}.

However, for the verification of arising frequency inequalities, we need to compute resolvents of additive compound operators. In the case of delay equations, this reduces to solving a first-order PDE with boundary conditions involving both partial derivatives and delays. This prevents dealing with the problem in a purely analytical way, see Section \ref{SUBSEC: ExampleResolventEquations} for the simplest example of such equations.

In this paper, we aim to develop approximation schemes to verify frequency inequalities and consider implementations of such schemes for conducting reliable numerical experiments, see Section \ref{SEC: ComputationOfFreqIneqDelayComp}. Besides some abstract results, we mainly stick to the case of scalar equations\footnote{See Remark \ref{REM: FrequencyCompoundGeneralCase} concerned with developing analogs of the approximation scheme for systems of equations and the end of Section \ref{SUBSEC: ExampleResolventEquations}.}. We give applications to the Suarez--Schopf delayed oscillator (see Section \ref{SEC: SuarezSchopfCompoundStab}), which is a system with a hidden boundary of global stability (see \cite{AnikushinRom2023SS}), and the Mackey--Glass equations (see Section \ref{SEC: MackeyGlassCompoundStab}), which is conjectured to be a system with a trivial boundary of global stability. For these models, the developed machinery indicates sharper regions of global stability than the purely analytical results from \cite{Anikushin2022Semigroups, Anikushin2023LyapExp}. Moreover, for Mackey--Glass equations with classical parameters, it even improves the delicate result of Liz, Tkachenko, and Trofimchuk \cite{Lizetal2003} relying on some specificity of scalar equations. In \cite{AnikushinRomanov2024EffEst}, we also gave applications of the frequency criterion to the Nicholson blowflies model and compare it with several existing stability results.

Note also that the frequency-domain approach to the uniform exponential stability of compound cocycles is potentially applicable to a range of problems, which include systems enjoying a kind of asymptotic compactness, such as parabolic, hyperbolic, or neutral delay equations. However, besides this paper and the adjacent paper \cite{Anikushin2023Comp}, we do not know such applications even in the case of ODEs. As to delay equations, here the general approach presented in \cite{Anikushin2020FreqDelay} reveals some specificity of such equations and leads to the discovery of their important functional-analytical properties, which we call structural Cauchy formulas. Such properties are related to the well-posedness of the infinite-horizon quadratic regulator problem.

Although the analytical side of our approach constituted by \cite{Anikushin2023Comp} and \cite{Anikushin2020FreqDelay} may seem complicated (especially for experimentalists), the approximation scheme \nameref{DESC: AS1DelayCompound}--\nameref{DESC: AS4DelayCompound} stated in Section \ref{SEC: ApproximationFreqIneqDelayCompound}, as well as the explicit analysis presented in Section \ref{SUBSEC: ExampleResolventEquations}, shall be accessible to a wide audience.

To the best of our knowledge, there is still no variant of the closing lemma that is appropriate for infinite-dimensional problems and delay equations in particular. Because of this, we are unable to prove that our conditions generally imply the global stability, but we believe in this because of their robustness. However, in the case of the Suarez--Schopf delayed oscillator, this problem can be avoided since the system belongs to the class of monotone feedback systems that satisfy the Poincar\'{e}--Bendixson trichotomy \cite{MalletParetSell1996}. Moreover, for some delay equations one may construct finite-dimensional inertial manifolds (see \cite{Anikushin2020FreqDelay, Anikushin2020Geom, AnikushinAADyn2021}) and apply the usual closing lemma. We hope that our research will stimulate developments of the closing lemma in infinite dimensions.

This paper is organized as follows. In Section \ref{SEC: Preliminaries} we introduce basic definitions. Namely, in Section \ref{SUBSEC: MultiplicativeAddCompoundsBrief} we briefly discuss tensor products of Hilbert spaces and compound operators on $m$-fold tensor products. In Section \ref{SUBSEC: CocyclesAndSemiflowsBrief} we give definitions of semiflows and cocycles. In Section \ref{SEC: StabilityOfDelayCompounds} we expound a part of the theory developed in \cite{Anikushin2023Comp}, which is necessary to introduce frequency conditions for the uniform exponential stability of compound cocycles generated by delay equations, see the final Theorem \ref{TH: QuadraticFunctionalDelayCompoundTheorem}. In Section \ref{SEC: ComputationOfFreqIneqDelayComp} we develop approximation schemes to verify frequency inequalities (see Section \ref{SEC: ApproxSchemeStatement} for the statement and Section \ref{SEC: ApproxSchemeConverg} for a discussion) and compare them with a direct approach to examine the resolvent equations for twofold additive compound operators (see Section \ref{SUBSEC: ExampleResolventEquations}). Then we give applications to the Suarez--Schopf delayed oscillator (see Section \ref{SEC: SuarezSchopfCompoundStab}) and the Mackey--Glass equations (see Section \ref{SEC: MackeyGlassCompoundStab}).

%% file: Preliminaries.tex
\section{Preliminaries}
\label{SEC: Preliminaries}
\subsection{Multiplicative and additive compound operators on tensor products of Hilbert spaces}
\label{SUBSEC: MultiplicativeAddCompoundsBrief}

Throughout the paper, $\mathcal{L}(\mathbb{E};\mathbb{F})$ denotes the space of all bounded linear operators between Banach spaces $\mathbb{E}$ and $\mathbb{F}$, and $\|\cdot\|_{\mathcal{L}(\mathbb{E};\mathbb{F})}$ denotes the operator norm. If $\mathbb{E}=\mathbb{F}$, we usually write just $\mathcal{L}(\mathbb{E})$. If  $\mathbb{E} = \mathbb{H}$ is a Hilbert space, we denote the inner product and the induced norm in $\mathbb{H}$ by $\langle \cdot, \cdot \rangle_{\mathbb{H}}$ and $|\cdot|_{\mathbb{H}}$, respectively. 

Let us briefly discuss basic concepts concerned with tensor products of Hilbert spaces, see, for example, \cite{Temam1997}. Let $\mathbb{H}_{1}$ and $\mathbb{H}_{2}$ be real or complex Hilbert spaces. By $\mathbb{H}_{1} \odot \mathbb{H}_{2}$ we denote the algebraic tensor product of $\mathbb{H}_{1}$ and $\mathbb{H}_{2}$. For $v_{1} \in \mathbb{H}_{1}$ and $v_{2} \in \mathbb{H}_{2}$, there is an element $v_{1} \otimes v_{2}$ of $\mathbb{H}_{1} \odot \mathbb{H}_{2}$ called the tensor product of $v_{1}$ and $v_{2}$. Recall that $v_{1} \otimes v_{2}$ is linear in both arguments, and such elements, which are called decomposable tensors, span $\mathbb{H}_{1} \odot \mathbb{H}_{2}$. Let $\langle \cdot, \cdot \rangle_{\mathbb{H}_{1}}$ and $\langle \cdot, \cdot \rangle_{\mathbb{H}_{2}}$ be the inner products in $\mathbb{H}_{1}$ and $\mathbb{H}_{2}$. We endow $\mathbb{H}_{1} \odot \mathbb{H}_{2}$ with the inner product defined by 
\begin{equation}
	\label{EQ: InnerProductTensorHilbertSpace}
	\langle v_{1} \otimes v_{2}, w_{1} \otimes w_{2} \rangle_{\mathbb{H}_{1} \otimes \mathbb{H}_{2}} \coloneq \langle v_{1},w_{1} \rangle_{\mathbb{H}_{1}} \langle v_{2}, w_{2} \rangle_{\mathbb{H}_{2}}
\end{equation}
for all $v_{1},w_{1} \in \mathbb{H}_{1}$ and $v_{2},w_{2} \in \mathbb{H}_{2}$. This formula indeed correctly defines an inner product in $\mathbb{H}_{1} \odot \mathbb{H}_{2}$ due to the universal property of algebraic tensor products. Now the \textit{tensor product} $\mathbb{H}_{1} \otimes \mathbb{H}_{2}$ of $\mathbb{H}_{1}$ and $\mathbb{H}_{2}$ is defined as the completion of $\mathbb{H}_{1} \odot \mathbb{H}_{2}$ by the norm induced by \eqref{EQ: InnerProductTensorHilbertSpace}. 

Given Hilbert spaces $\mathbb{H}_{1},\mathbb{H}_{2}, \mathbb{W}_{1}$, and $\mathbb{W}_{2}$ and bounded linear operators $L_{1} \in \mathcal{L}(\mathbb{H}_{1};\mathbb{W}_{1})$ and $L_{2} \in \mathcal{L}(\mathbb{H}_{2}; \mathbb{W}_{2})$, there is a unique operator $L_{1} \otimes L_{2} \in \mathcal{L}(\mathbb{H}_{1} \otimes \mathbb{H}_{2}; \mathbb{W}_{1} \otimes \mathbb{W}_{2})$ called the \textit{tensor product} of $L_{1}$ and $L_{2}$ such that
\begin{equation}
	(L_{1} \otimes L_{2})(v_{1} \otimes v_{2}) = L_{1} v_{1} \otimes L_{2} v_{2} \qquad \text{for all} \quad v_{1} \in \mathbb{H}_{1}, v_{2} \in \mathbb{H}_{2}.
\end{equation}
It can be shown that $\| L_{1} \otimes L_{2} \| = \| L_{1} \| \cdot \| L_{2}\|$, where $\|\cdot\|$ denotes appropriate operator norms associated with the above inner products. Moreover, by definition, the tensor product of operators behaves well with respect to compositions of operators in the sense that $(BA) \otimes (DC) = (B \otimes D)(A \otimes C)$ holds for any bounded linear operators $A,B,C$, and $D$ defined on appropriate spaces.

For any triple $\mathbb{H}_{1}, \mathbb{H}_{2}$, and $\mathbb{H}_{3}$ of Hilbert spaces, we have that the tensor products $(\mathbb{H}_{1} \otimes \mathbb{H}_{2}) \otimes \mathbb{H}_{3}$ and $\mathbb{H}_{1} \otimes (\mathbb{H}_{2} \otimes \mathbb{H}_{3})$ are naturally isometrically isomorphic and therefore denoted just by $\mathbb{H}_{1} \otimes \mathbb{H}_{2} \otimes \mathbb{H}_{3}$. This allows to carry the above constructions to any finite product $\mathbb{H}_{1} \otimes \cdots \otimes \mathbb{H}_{m}$ of Hilbert spaces.

Given a Hilbert space $\mathbb{H}$ and a positive integer $m$, we denote the $m$-fold tensor product $\mathbb{H}^{\otimes m}$ of $\mathbb{H}$ with itself by $\mathbb{H}^{\otimes m} \coloneq \mathbb{H} \otimes \cdots \otimes \mathbb{H}$ ($m$ times). Then for any operator $L \in \mathcal{L}(\mathbb{H})$ we denote its $m$-fold product $L \otimes \cdots \otimes L \in \mathcal{L}(\mathbb{H}^{\otimes m})$ by $L^{\otimes m}$ and call it the \textit{$m$-fold multiplicative compound} of $L$.

Let $\mathbb{S}_{m}$ be the symmetric group on $\{1,\ldots,m\}$. For each $\sigma \in \mathbb{S}_{m}$, let $S_{\sigma} \in \mathcal{L}(\mathbb{H}^{\otimes m})$ be the transposition operator with respect to $\sigma$, i.e., for all $v_{1},\ldots,v_{m} \in \mathbb{H}$, we have
\begin{equation}
	\label{EQ: HilbertTensorProdOperatorSsigma}
	S_{\sigma}(v_{1} \otimes \cdots \otimes v_{m}) \coloneq v_{\sigma(1)} \otimes \cdots \otimes v_{\sigma(m)}. 
\end{equation}
It is not hard to show that $S_{\sigma_{1} \sigma_{2}} = S_{\sigma_{2}} S_{\sigma_{1}}$ for all $\sigma_{1},\sigma_{2} \in \mathbb{S}_{m}$, see Remark \ref{REM: AntihomomorphismOfInducedByPermutations}. In particular, $S^{-1}_{\sigma} = S_{\sigma^{-1}}$. Moreover, $S^{*}_{\sigma} = S^{-1}_{\sigma}=S_{\sigma^{-1}}$, i.e., $S_{\sigma}$ is a unitary operator.

Let $\Pi^{\wedge}_{m} \in \mathcal{L}(\mathbb{H}^{\otimes m})$ be given by
\begin{equation}
	\label{EQ: WedgeProjectorHilbertSpace}
	\Pi^{\wedge}_{m} \coloneq \frac{1}{m!} \sum_{\sigma \in \mathbb{S}_{m}} (-1)^{\sigma}S_{\sigma}.
\end{equation}
From the above properties of $S_{\sigma}$, it can be shown that $\Pi^{\wedge}_{m}$ is an orthogonal projector in $\mathbb{H}^{\otimes m}$. Let $\mathbb{H}^{\wedge m}$ be its image, which is called the \textit{$m$-fold exterior product} of $\mathbb{H}$. For all $v_{1},\ldots,v_{m} \in \mathbb{H}$, we set $v_{1} \wedge \cdots \wedge v_{m} \coloneq \Pi^{\wedge}_{m}(v_{1} \otimes \cdots \otimes v_{m})$. 

It is clear that for any $L \in \mathcal{L}(\mathbb{H})$, the operator $L^{\otimes m}$ commutes with any $S_{\sigma}$ from \eqref{EQ: HilbertTensorProdOperatorSsigma} and, as a consequence, commutes with $\Pi^{\wedge}_{m}$ from \eqref{EQ: WedgeProjectorHilbertSpace}. Thus, $\mathbb{H}^{\wedge m}$ is invariant with respect to $L^{\otimes m}$. Let $L^{\wedge m}$ be the restriction of $L^{\otimes m}$ to $\mathbb{H}^{\wedge m}$ called the \textit{$m$-fold antisymmetric multiplicative compound} of $L$. It is sometimes convenient to say that $L^{\wedge m}$ is the \textit{$m$-fold multiplicative compound of $L$} in $\mathbb{H}^{\wedge m}$. It is not hard to see that
\begin{equation}
	L^{\wedge m} (v_{1} \wedge \cdots \wedge v_{m}) = Lv_{1} \wedge \cdots \wedge Lv_{m}
\end{equation}
holds for all $v_{1},\ldots,v_{m} \in \mathbb{H}$.

Now suppose that $G$ is a $C_{0}$-semigroup in $\mathbb{H}$, see \cite{EngelNagel2000}, and let $G(t) \in \mathcal{L}(\mathbb{H})$ denote its time-$t$ mapping for $t \geq 0$. By $G^{\otimes m}$ (resp. $G^{\wedge m}$) we denote the semigroup called the \textit{$m$-fold multiplicative compound} of $G$ in $\mathbb{H}^{\otimes m}$ (resp. $\mathbb{H}^{\wedge m}$) such that its time-$t$ mappings are given by $G^{\otimes m}(t) \coloneq (G(t))^{\otimes m}$ (resp. $G^{\wedge m}(t) \coloneq (G(t))^{\wedge m}$) for $t \geq 0$. It can be shown, see \cite[Section~2]{Anikushin2023Comp}, that $G^{\otimes m}$ (resp. $G^{\wedge m}$) is a $C_{0}$-semigroup in $\mathbb{H}^{\otimes m}$ (resp. $\mathbb{H}^{\wedge m}$).

Suppose that $A$ is the generator of a $C_{0}$-semigroup $G$ in $\mathbb{H}$. Let $A^{[\otimes m]}$ (resp. $A^{[\wedge m]}$) be the generator of $G^{\otimes m}$ (resp. $G^{\wedge m}$). Then $A^{[\otimes m]}$ (resp. $A^{[\wedge m]}$) is called the \textit{$m$-fold additive compound} (resp. the \textit{$m$-fold antisymmetric additive compound}) of $A$. We also say that $A^{[\wedge m]}$ is the \textit{$m$-fold additive compound of $A$} in $\mathbb{H}^{\wedge m}$.

If $G$ is eventually norm continuous (resp. eventually compact), the semigroups $G^{\otimes m}$ and $G^{\wedge m}$ are also eventually norm continuous (resp. eventually compact) by \cite[Propositions 2.2 and 2.3]{Anikushin2023Comp}. In the case of eventually compact semigroups, which arise in the study of delay equations, we can relate eigenvalues and the corresponding spectral subspaces of $A$ with those of $A^{[\otimes m]}$ or $A^{[\wedge m]}$, see \cite[Theorem 3.2]{Anikushin2023Comp}. In this paper, the following property concerned with the spectral bound of $A^{[\wedge m]}$ is important.
\begin{proposition}
	\label{PROP: SpectralBoundAwedgeViaA}
	Suppose that $G$ is eventually compact, and let $\lambda_{1}(A), \lambda_{2}(A), \ldots$ be the eigenvalues of $A$ arranged by nonincreasing their real parts and according to their multiplicities. Then the spectral bound $s(A^{[\wedge m]})$ of $A^{[\wedge m]}$ is given by
	\begin{equation}
		s(A^{[\wedge m]}) = \sum_{j=1}^{m}\operatorname{Re}\lambda_{j}(A),
	\end{equation}
    provided that $A$ has at least $m$ eigenvalues, and $s(A^{[\wedge m]}) = -\infty$ otherwise.
\end{proposition}
\begin{proof}
	Since $G$ is eventually compact, $G^{\wedge m}$ is also eventually compact due to \cite[Proposition 2.2]{Anikushin2023Comp}. In virtue of \cite[Theorem 3.1, Chapter V]{EngelNagel2000}, the spectrum of $A^{[\wedge m]}$ consists of eigenvalues having finite algebraic multiplicities, and, thanks to \cite[Corollary 2.11, Chapter IV]{EngelNagel2000}, for any $\nu \in \mathbb{R}$ the number of eigenvalues in the half-plane $\operatorname{Re}\lambda > \nu$ is finite. Consequently, the spectral bound of $A^{[\wedge m]}$ is given by the largest real part of its eigenvalues (or by $-\infty$ if the spectrum is empty).
	
	In view of \cite[Theorem 3.2]{Anikushin2023Comp}, the spectrum of $A^{[\otimes m]}$ consists of the sums 
	\begin{equation}
		\label{EQ: SpectralBoundCompoundProofDecompos}
		\lambda_{0} \coloneq \sum_{j=1}^{m}\lambda_{j},
	\end{equation}
	where $\lambda_{j}$ is an eigenvalue of $A$ for any $j \in \{1,\ldots,m\}$, and for the spectral subspace $\mathbb{L}_{A^{[\otimes m]}}(\lambda_{0})$ of $A^{[\otimes m]}$ corresponding to $\lambda_{0}$ we have
	\begin{equation}
		\label{EQ: LeadingSpectralSubspace}
		\mathbb{L}_{A^{[\otimes m]}}(\lambda_{0}) = \bigoplus_{k=1}^{N} \bigotimes_{j=1}^{m}\mathbb{L}_{A}(\lambda^{k}_{j}),
	\end{equation}
	where $N$ is the number of distinct $m$-tuples $(\lambda^{k}_{1},\ldots,\lambda^{k}_{m})$ enumerated by $k \in \{1,\ldots,N\}$ such that \eqref{EQ: SpectralBoundCompoundProofDecompos} holds with $\lambda_{j} = \lambda^{k}_{j}$ for any $j \in \{1,\ldots,m\}$, and $\mathbb{L}_{A}(\lambda^{k}_{j})$ is the spectral subspace of $A$ corresponding to $\lambda^{k}_{j}$.
	
	Moreover, the spectrum of $A^{[\wedge m]}$ exactly consists of such $\lambda_{0}$ for which the projector $\Pi^{\wedge}_{m}$ defined in \eqref{EQ: WedgeProjectorHilbertSpace} does not vanish on $\mathbb{L}_{A^{[\otimes m]}}(\lambda_{0})$. In this case, $\Pi^{\wedge}_{m}\mathbb{L}_{A^{[\otimes m]}}(\lambda_{0})$ is the spectral subspace of $A^{[\wedge m]}$ corresponding to $\lambda_{0}$.
	
	By \eqref{EQ: LeadingSpectralSubspace}, $\Pi^{\wedge}_{m} \mathbb{L}_{A^{[\otimes m]}}(\lambda_{0}) \not= 0$ if and only if $\Pi^{\wedge}_{m}\bigotimes_{j=1}^{m}\mathbb{L}_{A}(\lambda^{k}_{j}) \not=0$ for some $k \in \{1,\ldots,N\}$.  Moreover, $\Pi^{\wedge}_{m}\bigotimes_{j=1}^{m}\mathbb{L}_{A}(\lambda^{k}_{j}) \not=0$ if and only if each $\lambda^{k}_{j}$ does not occur in the $m$-tuple $(\lambda^{k}_{1},\ldots,\lambda^{k}_{m})$ more often than its algebraic multiplicity. Clearly, $\lambda_{0} = \sum_{j=1}^{m}\lambda_{j}(A)$ satisfies this condition and has the largest real part (or the spectrum is empty if there are less than $m$ eigenvalues of $A$).
\end{proof}

\subsection{Semiflows and cocycles}
\label{SUBSEC: CocyclesAndSemiflowsBrief}

Consider a \textit{time space} $\mathbb{T} \in \{ \mathbb{R}_{+}, \mathbb{R} \}$, where $\mathbb{R}_{+} = [0,+\infty)$. A family of mappings $\vartheta^{t} \colon \mathcal{Q} \to \mathcal{Q}$, where $t \in \mathbb{T}$ and $\mathcal{Q}$ is a complete metric space, is called a \textit{dynamical system} on $\mathcal{Q}$ if
\begin{description}[before=\let\makelabel\descriptionlabel]
	\item[(DS1)\refstepcounter{desccount}\label{DESC: DS1}] for any $q \in \mathcal{Q}$ and $t,s \in \mathbb{T}$, we have $\vartheta^{t+s}(q) = \vartheta^{t}( \vartheta^{s}(q))$ and $\vartheta^{0}(q) = q$;
	\item[(DS2)\refstepcounter{desccount}\label{DESC: DS2}] the mapping $\mathbb{T} \times \mathcal{Q} \ni (t,q) \mapsto \vartheta^{t}(q)$ is continuous.
\end{description}
For brevity, we often use the notation $(\mathcal{Q},\vartheta)$ or simply $\vartheta$ to denote the dynamical system. In the case $\mathbb{T} = \mathbb{R}_{+}$ (resp. $\mathbb{T} = \mathbb{R}$), we call $\vartheta$ a \textit{semiflow} (resp. a \textit{flow}) on $\mathcal{Q}$.

Let a dynamical system $(\mathcal{Q},\vartheta)$ be fixed. Given a Banach space $\mathbb{E}$, a family of mappings $\psi^{t}(q,\cdot) \colon \mathbb{E} \to \mathbb{E}$, where $t \in \mathbb{R}_{+}$ and $q \in \mathcal{Q}$, is called a \textit{cocycle} in $\mathbb{E}$ over $(\mathcal{Q},\vartheta)$ if 
\begin{description}[before=\let\makelabel\descriptionlabel]
	\item[\textbf{(CO1)}\refstepcounter{desccount}\label{DESC: CO1}] for any $v \in \mathbb{E}$, $q \in \mathcal{Q}$, and $t,s \in \mathbb{R}_{+}$, we have $\psi^{t+s}(q,v) = \psi^{t}(\vartheta^{s}(q),\psi^{s}(q,v))$ and $\psi^{0}(q,v) = v$;
	\item[\textbf{(CO2)}\refstepcounter{desccount}\label{DESC: CO2}] the mapping $\mathbb{R}_{+} \times \mathcal{Q} \times \mathbb{E} \ni (t,q,v) \mapsto \psi^{t}(q,v)$ is continuous.
\end{description}
For brevity, we often denote such a cocycle by $\psi$. In the context of cocycles, $(\mathcal{Q},\vartheta)$ is often called the \textit{base system} or the \textit{driving system}.

For each cocycle $\psi$ in $\mathbb{E}$ over $(\mathcal{Q},\vartheta)$, there is the associated semiflow $\pi$ on $\mathcal{Q} \times \mathbb{E}$, called a \textit{skew-product semiflow}, given by
\begin{equation}
	\pi^{t}(q,v) \coloneq ( \vartheta^{t}(q), \psi^{t}(q,v) ) \qquad \text{for all} \quad t \geq 0, q \in \mathcal{Q}, \ \text{and} \ v \in \mathbb{E}.
\end{equation}

Suppose that each cocycle mapping $\psi^{t}(q,\cdot)$ belongs to the space $\mathcal{L}(\mathbb{E})$ of bounded linear operators in $\mathbb{E}$. Then we say that $\psi$ is a \textit{linear cocycle} and denote it by $\Xi$. Moreover, if $\Xi$ additionally satisfies the following properties:
\begin{description}[before=\let\makelabel\descriptionlabel]
	\item[\textbf{(UC1)}\refstepcounter{desccount}\label{DESC: UC1}] for any $t \in \mathbb{R}_{+}$, the mapping $\mathcal{Q} \ni q \mapsto \Xi^{t}(q,\cdot) \in \mathcal{L}(\mathbb{E})$ is continuous in the operator norm; 
	\item[\textbf{(UC2)}\refstepcounter{desccount}\label{DESC: UC2}] the cocycle mappings are bounded uniformly in finite times,
	\begin{equation}
		\sup_{t \in [0,1]} \sup_{q \in \mathcal{Q}}\| \Xi^{t}(q,\cdot) \|_{\mathcal{L}(\mathbb{E})} < +\infty,
	\end{equation}
\end{description}
then $\Xi$ is called a \textit{uniformly continuous linear cocycle}. Note that for such cocycles, the condition \nameref{DESC: CO2} is equivalent to that the operator $\Xi^{t}(q,\cdot)$ depends continuously on $(t,q)$ in the strong operator topology.

In this paper, we deal with uniformly continuous linear cocycles in a Hilbert space $\mathbb{H}$, see Remark \ref{REM: CocyclesArisingFromDDEs}. Let $\Xi$ be such a cocycle. For each integer $m > 0$, we associate with $\Xi$ a cocycle $\Xi_{m}$ acting on the $m$-fold tensor product $\mathbb{H}^{\otimes m}$ of $\mathbb{H}$. For $\Xi_{m}$, each cocycle mapping $\Xi^{t}_{m}(q,\cdot) \in \mathcal{L}(\mathbb{H}^{\otimes m})$, where $t \geq 0$ and $q \in \mathcal{Q}$, is given by the $m$-fold multiplicative compound of $\Xi^{t}(q,\cdot)$ in $\mathbb{H}^{\otimes m}$. It can be shown that $\Xi_{m}$ is indeed a uniformly continuous cocycle, and we call it the \textit{$m$-fold multiplicative compound} of $\Xi$ in $\mathbb{H}^{\otimes m}$. Moreover, we use the same notation to denote the restriction of $\Xi_{m}$ to the $m$-fold exterior power $\mathbb{H}^{\wedge m}$ and call it the \textit{$m$-fold multiplicative compound} of $\Xi$ in $\mathbb{H}^{\wedge m}$ or the \textit{$m$-fold antisymmetric multiplicative compound} of $\Xi$. It should be clear from the context in which spaces the cocycle acts.

%% file: GeneralTheory.tex
\section{Exponential stability of compound cocycles generated by delay equations}
\label{SEC: StabilityOfDelayCompounds}

\subsection{Cocycles generated by nonautonomous delay equations}
We are going to describe the class of delay equations to which our theory applies. For this, let $\pi$ be a semiflow on a complete metric space $\mathcal{P}$. For some positive integers $r_{1}$ and $r_{2}$, we set $\mathbb{U} \coloneq \mathbb{R}^{r_{1}}$ and $\mathbb{M} \coloneq \mathbb{R}^{r_{2}}$, where the spaces are endowed with some (not necessarily standard) inner products. Consider the class of nonautonomous delay equations in $\mathbb{R}^{n}$ over $(\mathcal{P},\pi)$ that are described over each $\wp \in \mathcal{P}$ as follows:
\begin{equation}
	\label{EQ: DelayRnLinearized}
	\dot{x}(t) = \widetilde{A}x_{t} + \widetilde{B}F'(\pi^{t}(\wp))Cx_{t}.
\end{equation}
Here $\tau > 0$ is a fixed real number (delay); $x(\cdot) \colon [-\tau,T] \to \mathbb{R}^{n}$ for some $T>0$, and $x_{t}(\theta) = x(t+\theta)$ for $\theta \in [-\tau,0]$ denotes the $\tau$-history segment of $x(\cdot)$ at $t \in [0,T]$. Moreover, $\widetilde{A} \colon C([-\tau,0];\mathbb{R}^{n}) \to \mathbb{R}^{n}$ and $C \colon C([-\tau,0];\mathbb{R}^{n}) \to \mathbb{M}$ are bounded linear operators; $\widetilde{B} \colon \mathbb{U} \to \mathbb{R}^{n}$ is a linear operator, and $F' \colon \mathcal{P} \to \mathcal{L}(\mathbb{M};\mathbb{U})$ is a continuous mapping such that for some $\Lambda>0$ we have
\begin{equation}
	\label{EQ: LipschitzFprimeDelay}
	\|F'(\wp)\|_{\mathcal{L}(\mathbb{M};\mathbb{U})} \leq \Lambda \qquad \text{for all} \quad \wp \in \mathcal{P}.
\end{equation}
\begin{remark}
	\label{REM: CocyclesArisingFromDDEs}
	Equations as \eqref{EQ: DelayRnLinearized} arise after linearization of nonautonomous (over a dynamical system $(\mathcal{Q},\vartheta)$) nonlinear delay equations, which generate a nonlinear cocycle $\psi$ over $(\mathcal{Q},\vartheta)$. In this case, the derivative cocycle $\Xi$ generated by the linearized equations is a uniformly continuous linear cocycle over the skew-product semiflow $\pi$ associated with $\psi$. 
\end{remark}

To discuss the well-posedness of \eqref{EQ: DelayRnLinearized}, let us write it as an evolutionary equation in a proper Hilbert space. For this, consider the Hilbert space 
\begin{equation}
	\label{EQ: DelayEqsHilbertSpace}
	\mathbb{H} = L_{2}([-\tau,0];\mu;\mathbb{R}^{n}),
\end{equation}
where the measure $\mu$ is given by the sum of the Lebesgue measure on $[-\tau,0]$ and the $\delta$-measure concentrated at $0$. For $\phi \in \mathbb{H}$, we define $R^{(1)}_{0}\phi = \phi(0) \in \mathbb{R}$ and $R^{(1)}_{1}\phi = \restr{\phi}{(-\tau,0)} \in L_{2}(-\tau,0;\mathbb{R}^{n})$. Here the upper index in the notation will be explained below.

We embed the space $\mathbb{E} = C([-\tau,0];\mathbb{R}^{n})$ into $\mathbb{H}$ by sending each $\psi \in \mathbb{E}$ into $\phi \in \mathbb{H}$ such that $R^{(1)}_{0}\phi = \psi(0)$ and $R^{(1)}_{1}\phi = \psi$. It will be convenient to identify the elements of $\mathbb{E}$ and their images in $\mathbb{H}$ under the embedding. In particular, we use the same notation for the operator $C$ from \eqref{EQ: DelayRnLinearized} and its composition with the embedding. Namely, we set $C\phi \coloneq CR^{(1)}_{1}\phi = C\psi$ for $\phi \in \mathbb{H}$ and $\psi \in \mathbb{E}$ related by the just introduced embedding.

With $\widetilde{A}$ from \eqref{EQ: DelayRnLinearized}, we associate the operator $A$ in $\mathbb{H}$ defined for $\phi \in \mathcal{D}(A)$ by
\begin{equation}
	\label{EQ: OperatorAScalarDelayEquations}
	R^{(1)}_{0}(A\phi) = \widetilde{A} R^{(1)}_{1}\phi \quad \text{and} \quad R^{(1)}_{1}(A\phi) = \frac{d}{d\theta} R^{(1)}_{1}\phi,
\end{equation}
where the domain $\mathcal{D}(A)$ of $A$ is given by the embedding of $W^{1,2}(-\tau,0;\mathbb{R}^{n})$ into $\mathbb{H}$ similarly to the above. Since $W^{1,2}(-\tau,0;\mathbb{R}^{n})$ can be naturally continuously embedded into $C([-\tau,0];\mathbb{R}^{n})$, the definition is correct. Clearly, $A$ is a closed operator.

Now define a bounded linear operator $B \colon \mathbb{U} \to \mathbb{H}$ by $R^{(1)}_{0}B\eta = \widetilde{B} \eta$ and $R^{(1)}_{1}B\eta = 0$ for $\eta \in \mathbb{U}$. Then \eqref{EQ: DelayRnLinearized} can be treated as an abstract evolution equation in $\mathbb{H}$ given by
\begin{equation}
	\label{EQ: DelayLinearCocAbsract}
	\dot{\xi}(t) = A\xi(t) + BF'(\pi^{t}(\wp))C\xi(t).
\end{equation}

By an adaptation of \cite[Theorem 1]{Anikushin2022Semigroups} and the variation of constants formula derived therein, one can show that \eqref{EQ: DelayLinearCocAbsract} generates a uniformly continuous linear cocycle $\Xi$ in $\mathbb{H}$ over $(\mathcal{P},\pi)$ given by $\Xi^{t}(\wp,\xi_{0}) \coloneq \xi(t;\xi_{0})$, where $\xi(t;\xi_{0})$ for $t \geq 0$ is a solution (in a generalized sense) of \eqref{EQ: DelayLinearCocAbsract} with $\xi(0;\xi_{0}) = \xi_{0}$. We refer to \cite{Anikushin2023Comp} for precise formulations in which sense the solutions may be understood.

It can be shown that the operator $A$, as in \eqref{EQ: OperatorAScalarDelayEquations}, generates an eventually compact $C_{0}$-semigroup $G$ in $\mathbb{H}$, see \cite{Anikushin2022Semigroups}. For any integer $m \geq 1$, according to Subsection \ref{SUBSEC: MultiplicativeAddCompoundsBrief}, let $G^{\wedge m}$ be the $m$-fold multiplicative compound of $G$ in $\mathbb{H}^{\wedge m}$, and let $A^{[\wedge m]}$ be the $m$-fold antisymmetric additive compound of $A$, i.e., the generator of $G^{\wedge m}$.

Below, we aim to study the $m$-fold antisymmetric multiplicative compound $\Xi_{m}$ of $\Xi$ defined in Section \ref{SUBSEC: CocyclesAndSemiflowsBrief}. Namely, we will state conditions for its uniform exponential stability by considering $\Xi_{m}$ as a perturbation of $G^{\wedge m}$ (see Theorem \ref{TH: QuadraticFunctionalDelayCompoundTheorem}), expounding the theory from our adjacent work \cite{Anikushin2023Comp}. On this way, our basic aim is given by \eqref{EQ: InfinitesimaldDescriptionXim}, which gives a description of $\Xi_{m}$ on the infinitesimal level analogously to \eqref{EQ: DelayLinearCocAbsract}. This requires a description of the abstract spaces and operators along with the study of their intrinsic properties. Although in the subsequent applications we treat only the case of $n=1$ and $m=2$, we find it useful (to provide better understanding) to expound the theory in the general case.

In the forthcoming subsections, we present a compact exposition of some results from \cite{Anikushin2023Comp} that are necessary for applications. We therefore refer the interested reader to \cite{Anikushin2023Comp} for a more systematic treatment and detailed proofs of these results.

\subsection{Description of the abstract $m$-fold tensor and exterior products}
First, let us consider the abstract $m$-fold tensor product $\mathbb{H}^{\otimes m}$ of $\mathbb{H}$ from \eqref{EQ: DelayEqsHilbertSpace}. It is well known that $\mathbb{H}^{\otimes m}$ is naturally isometrically isomorphic to the space 
\begin{equation}
	\mathcal{L}^{\otimes}_{m} \coloneq L_{2}([-\tau,0]^{m};\mu^{\otimes m};(\mathbb{R}^{n})^{\otimes m}),
\end{equation}
where $\mu^{\otimes m}$ is the $m$-fold product of $\mu$. Recall that the isomorphism is defined on decomposable tensors $\phi_{1} \otimes \cdots \otimes \phi_{m}$, where $\phi_{1},\ldots,\phi_{m} \in \mathbb{H}$, by
\begin{equation}
	\label{EQ: IsomorphismHWedgeLm}
	\begin{split}
		\phi_{1} \otimes \cdots \otimes \phi_{m} \mapsto (\phi_{1} \otimes \cdots \otimes \phi_{m})(\cdot_{1},\ldots,\cdot_{m}) \in \mathcal{L}^{\otimes}_{m},\\
		(\phi_{1} \otimes \cdots \otimes \phi_{m})(\theta_{1},\ldots,\theta_{m}) \coloneq \phi_{1}(\theta_{1}) \otimes \cdots \otimes \phi_{m}(\theta_{m})
	\end{split}
\end{equation}
for $\mu^{\otimes m}$-almost all $(\theta_{1},\ldots,\theta_{m}) \in [-\tau,0]^{m}$.

In particular, the restriction of the above isomorphism to $\mathbb{H}^{\wedge m}$ provides an isometric isomorphism with the subspace $\mathcal{L}^{\wedge}_{m}$ of $\mu^{\otimes m}$-antisymmetric functions in $\mathcal{L}^{\otimes}_{m}$. Recall that $\mathcal{L}^{\wedge}_{m}$ consists of $\Phi \in \mathcal{L}^{\otimes}_{m}$ satisfying
\begin{equation}
	\label{EQ: AntisymmetricFunctionDefinition}
	\Theta_{\sigma}\Phi = (-1)^{\sigma} T_{\sigma} \Phi
\end{equation}
for any $\sigma \in \mathbb{S}_{m}$. Here $\Theta_{\sigma}$ permutes the arguments of $\Phi$ according to $\sigma$, i.e.,
\begin{equation}
	\label{EQ: ThetaSigmaDefinition}
	(\Theta_{\sigma}\Phi)(\theta_{1}, \ldots, \theta_{m}) \coloneq \Phi(\theta_{\sigma(1)},\ldots, \theta_{\sigma(m)} )
\end{equation}
for $\mu^{\otimes m}$-almost all $(\theta_{1}, \ldots, \theta_{m}) \in [-\tau,0]^{m}$, and $T_{\sigma}$ is the transposition operator (with respect to $\sigma$) in $(\mathbb{R}^{n})^{\otimes m}$ given by
\begin{equation}
	\label{EQ: TranspositionOperatorOnValues}
	T_{\sigma}(x_{1} \otimes \cdots \otimes x_{m}) \coloneq x_{\sigma(1)} \otimes \cdots \otimes x_{\sigma(m)}
\end{equation}
for all $x_{1},\ldots,x_{m} \in \mathbb{R}^{n}$. 
\begin{remark}
	For $n=1$, we have $(\mathbb{R}^{n})^{\otimes m} = \mathbb{R}$, and, consequently, $T_{\sigma}$ is just the identity operator. In this case, \eqref{EQ: AntisymmetricFunctionDefinition} coincides with the usual definition of an antisymmetric function, which changes its sign according to the permutation of arguments.
\end{remark}
\begin{remark}
	\label{REM: AntihomomorphismOfInducedByPermutations}
	For $\sigma_{1},\sigma_{2} \in \mathbb{S}_{m}$, let us emphasize that $T_{\sigma_{2}} T_{\sigma_{1}} = T_{\sigma_{1} \sigma_{2}}$, i.e., the correspondence $\sigma \to T_{\sigma}$ is an antihomomorphism\footnote{In a similar context, \cite{MalletParretNussbaum2013} asserts that $\sigma \mapsto S_{\sigma}$ is homomorphic, which is false.}. On the other hand, we have $\Theta_{\sigma_{2}} \Theta_{\sigma_{1}} = \Theta_{\sigma_{2} \sigma_{1}}$.
	
	To get a conceptual explanation for this, one should consider $x_{j}$ as a function of $j$. Then $T_{\sigma}$ is related\footnote{The space of functions $x(j) = x_{j}$ with domain $\{1,\ldots,m\}$ and values in $\mathbb{R}^{n}$ can be naturally identified with $(\mathbb{R}^{n})^{m}$. Then $\sigma^{*}$ is a multilinear mapping, and $T_{\sigma}$ can be considered as an extension of $\sigma^{*}$ to $(\mathbb{R}^{n})^{\otimes m}$.} to the action (a change of variables) $\sigma^{*}$ on such functions induced by the action of $\sigma$ on the space $\{1,\ldots,m\}$ of arguments. As is always the case for induced actions on functions, the inducing is contravariant, i.e., $(\sigma_{1}\sigma_{2})^{*} = \sigma^{*}_{2} \circ \sigma^{*}_{1}$.
	
	As to $\Theta_{\sigma}$, its action on functions $\Phi$ is induced by the permutation $h_{\sigma}$ of arguments $(\theta_{1},\ldots,\theta_{m})$, i.e., $\Phi_{\sigma} = h^{*}_{\sigma}$ in similar terms. By considering $\theta_{j}$ as a function of $j$ and the induced action $\sigma^{*}$ on such functions, we may write $h_{\sigma} = \sigma^{*}$. So, there are two contravariant operations resulting in $\Theta_{\sigma}$. \qed
\end{remark}

Up to the isomorphism \eqref{EQ: IsomorphismHWedgeLm}, one may express the operator $S_{\sigma}$ given by \eqref{EQ: HilbertTensorProdOperatorSsigma} as $S_{\sigma} = T_{\sigma}\Theta_{\sigma^{-1}}$. In particular, the projector $\Pi^{\wedge}_{m}$ defined in \eqref{EQ: WedgeProjectorHilbertSpace} is expressed by
\begin{equation}
	\label{EQ: AbstractPiWedgeToConcrete}
	\Pi^{\wedge}_{m} = \frac{1}{m!}\sum_{\sigma \in \mathbb{S}_{m}} (-1)^{\sigma} T_{\sigma^{-1}} \Theta_{\sigma}. 
\end{equation} 

From \eqref{EQ: AbstractPiWedgeToConcrete} it is clear that the restriction of the isomorphism \eqref{EQ: IsomorphismHWedgeLm} to the space $\mathbb{H}^{\wedge m}$ sends each $\phi_{1} \wedge \cdots \wedge \phi_{m}$, where $\phi_{1},\ldots,\phi_{m} \in \mathbb{H}$, into the function
\begin{equation}
	\label{EQ: DelayCompoundIsomorphismAntisymmetricDecomp}
	(\phi_{1} \wedge \cdots \wedge \phi_{m})(\theta_{1},\ldots,\theta_{m}) = \frac{1}{m!}\sum_{\sigma \in \mathbb{S}_{m}}(-1)^{\sigma}T_{\sigma^{-1}} \phi_{1}(\theta_{\sigma(1)}) \otimes \cdots \otimes \phi_{m}(\theta_{\sigma(m)})
\end{equation}
defined for $\mu^{\otimes m}$-almost all $(\theta_{1},\ldots,\theta_{m}) \in [-\tau,0]^{m}$.

It will be convenient to work in the spaces $\mathcal{L}^{\otimes}_{m}$ and $\mathcal{L}^{\wedge}_{m}$. For this, we need to introduce some related notations.

For all integers $k \in \{1,\ldots,m\}$ and $1 \leq j_{1} < \cdots < j_{k} \leq m$, we form a multi-index $j_{1}\ldots j_{k}$ and define the set $\mathcal{B}^{(m)}_{j_{1}\ldots j_{k}}$, which is called a $k$-\textit{face} of $[-\tau,0]^{m}$ with respect to $\mu^{\otimes m}$, by
\begin{equation}
	\label{EQ: DefinitionOfBoundaryFace}
	\mathcal{B}^{(m)}_{j_{1}\ldots j_{k}} \coloneq \left\{ (\theta_{1},\ldots,\theta_{m}) \in [-\tau,0]^{m} \ | \ \theta_{j} = 0 \ \text{for any} \ j \notin \{ j_{1},\ldots,j_{k}  \} \right\}.
\end{equation}
We also set $\mathcal{B}^{(m)}_{0} \coloneq \{0\}^{m}$, denoting the set corresponding to the unique $0$-face with respect to $\mu^{\otimes m}$, and consider it as $\mathcal{B}^{(m)}_{j_{1}\ldots j_{k}}$ with $k=0$. Then we define the \textit{restriction operator $R^{(m)}_{j_{1}\ldots j_{k}}$} (including $R^{(m)}_{0}$) by
\begin{equation}
	\label{EQ: RestrictionOperatorDelayTensor}
	\mathcal{L}^{\otimes}_{m} \ni \Phi \mapsto R^{(m)}_{j_{1}\ldots j_{k}}\Phi \coloneq \restr{\Phi}{\mathcal{B}^{(m)}_{j_{1}\ldots j_{k}}} \in L_{2}((-\tau,0)^{k};(\mathbb{R}^{n})^{\otimes m}),
\end{equation}
where in the last inclusion we naturally identified $\mathcal{B}^{(m)}_{j_{1}\ldots,j_{k}}$ with $[-\tau,0]^{k}$ by omitting the zeroed arguments. In other words, $R^{(m)}_{j_{1}\ldots j_{k}}$ takes a function of $m$ arguments $\theta_{1}, \ldots, \theta_{m}$ to the function of $k$ arguments $\theta_{j_{1}}, \ldots, \theta_{j_{k}}$, setting $\theta_{j}=0$ for $j \notin \{j_{1},\ldots,j_{k}\}$, considered as a function in the usual $L_{2}$-space over the $k$-cube $(-\tau,0)^{k}$.

Similarly to the operators $R^{(1)}_{1}$ and $R^{(1)}_{0}$ used in \eqref{EQ: OperatorAScalarDelayEquations}, any element $\Phi$ of $\mathcal{L}^{\otimes}_{m}$ is uniquely determined by its restrictions $R^{(m)}_{j_{1}\ldots j_{k}}\Phi$ taken over all multi-indices $j_{1}\ldots j_{k}$ as above. From this, we define $\partial_{j_{1}\ldots j_{k}}\mathcal{L}^{\otimes}_{m}$ as the subspace of $\mathcal{L}^{\otimes}_{m}$ where all the restriction operators, except possibly $R^{(m)}_{j_{1}\ldots j_{k}}$, vanish. We call $\partial_{j_{1}\ldots j_{k}}\mathcal{L}^{\otimes}_{m}$ the \textit{boundary subspace over the $k$-face $\mathcal{B}^{(m)}_{j_{1}\ldots j_{k}}$}. Note that $R^{(m)}_{j_{1}\ldots j_{k}}$ provides a natural isometric isomorphism between $\partial_{j_{1}\ldots j_{k}}\mathcal{L}^{\otimes}_{m}$ and $L_{2}((-\tau,0)^{k};(\mathbb{R}^{n})^{\otimes m})$. Clearly, the space $\mathcal{L}^{\otimes}_{m}$ decomposes into the inner orthogonal sum of boundary subspaces,
\begin{equation}
	\label{EQ: TensorSpaceDelayCompoundDecompositionBoundarySubspaces}
	\mathcal{L}^{\otimes}_{m} = \bigoplus_{k=0}^{m}\bigoplus_{j_{1}\ldots j_{k}} \partial_{j_{1}\ldots j_{k}}\mathcal{L}^{\otimes}_{m}.
\end{equation}
\begin{remark}
	\label{REM: SumMultiIndexNotation}
	In \eqref{EQ: TensorSpaceDelayCompoundDecompositionBoundarySubspaces} and for what follows, the sums over multi-indices $j_{1}\ldots j_{k}$ (with $k$ fixed) are always taken over all $1 \leq j_{1} < \cdots < j_{k} \leq m$. For $k=0$, by definition, we assume that there is a unique multi-index $j_{1}\ldots j_{k} = 0$, and we also set $\{j_{1},\ldots,j_{k}\} \coloneq \{0\}$. We always emphasize the limits of $k$ (if it is supposed to vary), which may be different. 
\end{remark}

If it is clear from the context, we often omit the upper index in $R^{(m)}_{j_{1}\ldots j_{k}}$ or $\mathcal{B}^{(m)}_{j_{1}\ldots j_{k}}$ and write simply $R_{j_{1}\ldots j_{k}}$ or $\mathcal{B}_{j_{1}\ldots j_{k}}$. Moreover, it will be convenient to use the notation $R_{j_{1}\ldots j_{k}}$ for a not necessarily increasing sequence $j_{1},\ldots,j_{k}$ to denote the same operator as for the properly rearranged sequence. Sometimes we will use the excluded index notation to denote restriction operators and $k$-faces. For example, for $j \in \{1, \ldots, m\}$ we will often use $R_{\hat{j}} \coloneq R_{1\ldots \hat{j} \ldots m}$ and $\mathcal{B}_{\hat{j}}\coloneq \mathcal{B}_{1 \ldots \hat{j} \ldots m}$, where the hat on the right means that the element is excluded from the considered set constituting the multi-index.
\begin{remark}
	\begin{figure}
		\centering
		\includegraphics[width=0.5\linewidth]{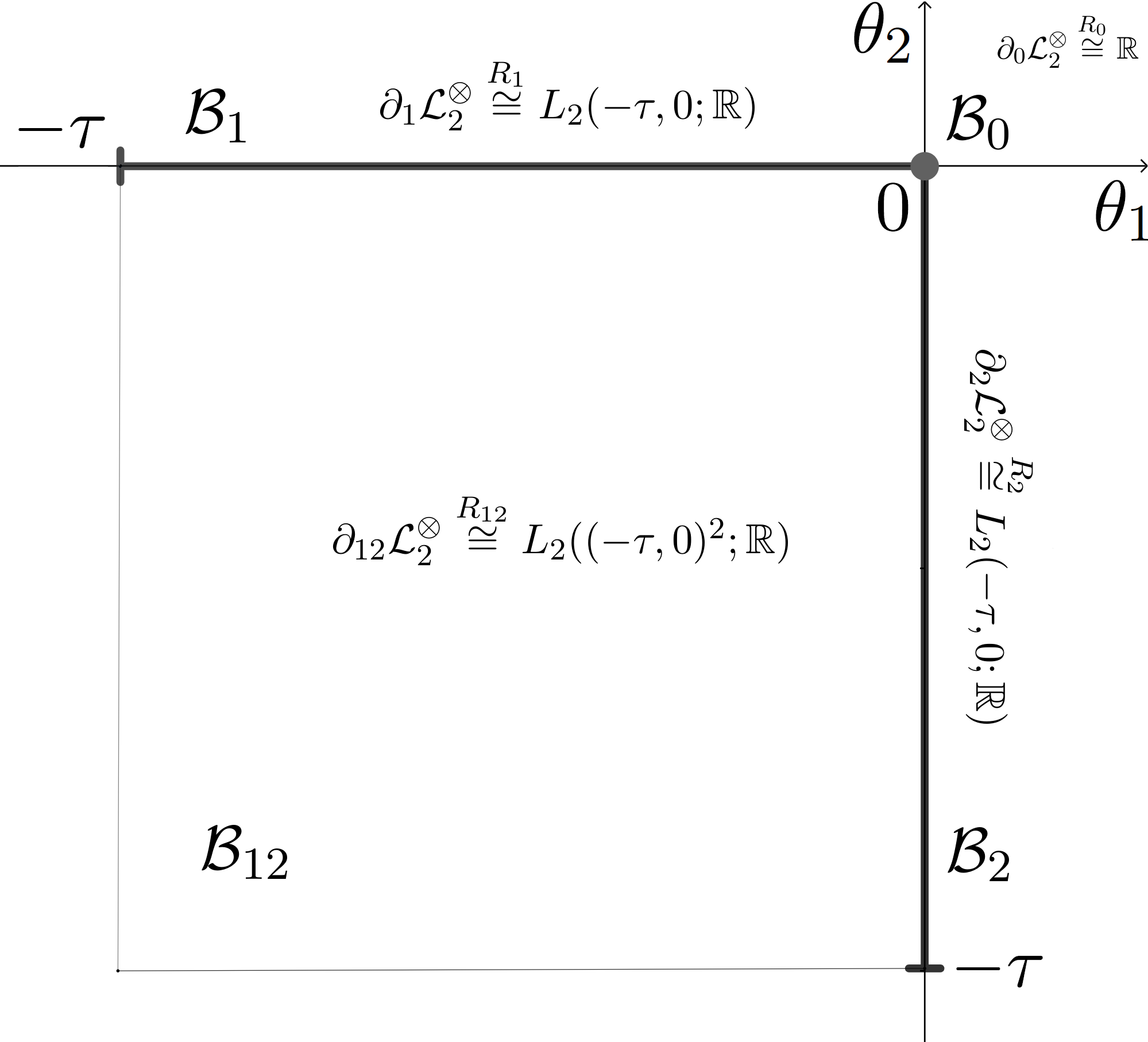}	
		\caption{The decomposition \eqref{EQ: TensorSpaceDelayCompoundDecompositionBoundarySubspaces} with $m=2$ and $n=1$. Here $\mathcal{L}^{\otimes}_{2} = L_{2}([-\tau,0]^{2};\mu^{\otimes 2};\mathbb{R})$ is decomposed into the sum of the boundary subspaces $\partial_{0}\mathcal{L}^{\otimes}_{2}$, $\partial_{1}\mathcal{L}^{\otimes}_{2}$, $\partial_{2}\mathcal{L}^{\otimes}_{2}$, and $\partial_{12}\mathcal{L}^{\otimes}_{2}$ over the faces $\mathcal{B}_{0}$, $\mathcal{B}_{1}$, $\mathcal{B}_{2}$, and $\mathcal{B}_{12}$, respectively. These subspaces are naturally isomorphic to appropriate $L_{2}$-spaces via the restriction operators $R_{0}$, $R_{1}$, $R_{2}$, and $R_{12}$, respectively.}
		\label{FIG: ExampleFunction}
	\end{figure}
	For $m=2$ and $n=1$, any $\Phi \in \mathcal{L}^{\otimes}_{m}$ is determined by its four restrictions: one value $R_{0}\Phi \in \mathbb{R}$, two functions $R_{1}\Phi$ and $R_{2}\Phi$ from $L_{2}(-\tau,0;\mathbb{R})$, and one function $R_{12} \Phi$ from $L_{2}((-\tau,0)^{2};\mathbb{R})$, see Fig. \ref{FIG: ExampleFunction}. It should be noted that even if $R_{12}\Phi$, $R_{1}\Phi$, or $R_{2}\Phi$ admit continuous representations, it is not true in general that they need to be related on intersections of faces. In particular, the values $(R_{12}\Phi)(0,0)$, $(R_{1}\Phi)(0)$, $(R_{2}\Phi)(0)$, and $R_{0}\Phi$ need not be related.
\end{remark}

Analogously to \eqref{EQ: ThetaSigmaDefinition}, it is convenient to introduce operators $\Theta^{(k)}_{\bar{\sigma}}$ permuting (with respect to $\bar{\sigma} \in \mathbb{S}_{k}$) arguments of functions of $k$ variables, whose space should be understood from the context. In the case of $\Phi \in \mathcal{L}^{\wedge}_{m}$, one can describe the relations from \eqref{EQ: AntisymmetricFunctionDefinition} in terms of the restrictions $R_{j_{1}\ldots j_{k}}\Phi$ as follows. 
\begin{proposition}\cite[Proposition 4.1]{Anikushin2023Comp}
	\label{PROP: AntisymmetricRelationsGeneral}
	An element $\Phi \in \mathcal{L}^{\otimes}_{m}$ belongs to $\mathcal{L}^{\wedge}_{m}$ if and only if for all $k \in \{0,\ldots,m\}$, $1 \leq j_{1} < \cdots < j_{k} \leq m$, and $\sigma \in \mathbb{S}_{m}$, we have
	\begin{equation}
		\label{EQ: AntisymmetricRelationsGeneral}
		\begin{split}
			R_{j_{1}\ldots j_{k}}\Phi = (-1)^{\sigma}T_{\sigma} \Theta^{(k)}_{\bar{\sigma}} R_{\sigma(j_{1}) \ldots \sigma(j_{k})}\Phi,
		\end{split}
	\end{equation}
	where $\bar{\sigma} \in \mathbb{S}_{k}$ is such that $\sigma( j_{\bar{\sigma}(1)} ) < \cdots < \sigma( j_{\bar{\sigma}(k)})$.
	
	In particular, we have that\footnote{Here in \eqref{EQ: AntisymmetricRestrictionRelationOuterPermutation} the tail of $\sigma$, i.e., $\sigma(l)$ for $l \geq k+1$, is arbitrary.}
	\begin{equation}
		\label{EQ: AntisymmetricRestrictionRelationOuterPermutation}
		 R_{1\ldots k}\Phi = (-1)^{\sigma}T_{\sigma} R_{j_{1}\ldots j_{k}}\Phi \qquad \text{for any} \quad \sigma = \begin{pmatrix}
			1 & \ldots & k & \ldots \\
			j_{1} & \ldots & j_{k} & \ldots
		\end{pmatrix} \in \mathbb{S}_{m},
	\end{equation}
	and, as a consequence, for almost all $(\theta_{1},\ldots,\theta_{k}) \in (-\tau,0)^{k}$, we have
	\begin{equation}
		\label{EQ: AntisymmetricRestrictionRelationValues}
		(R_{1\ldots k}\Phi)(\theta_{1},\ldots,\theta_{k}) \in (\mathbb{R}^{n})^{\otimes k} \otimes (\mathbb{R}^{n})^{\wedge (m-k)}.
	\end{equation}
\end{proposition}

For $n=1$, the operator $T_{\sigma}$ becomes identical, since $(\mathbb{R}^{n})^{\otimes m} = \mathbb{R}$ for any $m$. Moreover, $(\mathbb{R}^{n})^{\wedge (m-k)} = 0$ for $m - k >1$, and the antisymmetric relations can be simplified as follows.
\begin{corollary}
	\label{COR: AntisymmetricRelationsScalar}
	Suppose $n=1$. Then the relations from \eqref{EQ: AntisymmetricRelationsGeneral} are equivalent to the relations
	\begin{alignat}{2}
		\label{EQ: AntisymmRelationsScalar}
		&R_{j_{1}\ldots j_{k}}\Phi = 0 \qquad &&\text{for any} \quad k \in \{0,\ldots, m-2\},\notag\\
		&R_{\hat{j}}\Phi \qquad &&\text{is antisymmetric for any} \quad j \in \{1, \ldots, m\},\\
		&R_{\hat{i}}\Phi = (-1)^{j-i} R_{\hat{j}}\Phi \qquad &&\text{for all} \quad i,j \in \{1, \ldots, m\},\notag\\
		&R_{1\ldots m}\Phi \qquad &&\text{is antisymmetric}.\notag
	\end{alignat}
\end{corollary}
\noindent We leave the proof as an exercise to the reader and refer to the proof of Proposition \ref{PROP: ControlSpaceCompoundDelayDescriptionScalarCase} below, where necessary arguments are applied in a similar context.

Note that the antisymmetric relations \eqref{EQ: AntisymmetricRelationsGeneral} link each $\partial_{j_{1} \ldots j_{k}} \mathcal{L}^{\otimes}_{m}$ with other boundary subspaces over $k$-faces (i.e., with the same $k$). Thus, for a given $k \in \{0,\ldots, m\}$, it is convenient to introduce the subspace (recall $\Pi^{\wedge}_{m}$ given by \eqref{EQ: WedgeProjectorHilbertSpace})
\begin{equation}
	\label{EQ: AntisymmetricSubspaceOverKfaces}
	\partial_{k}\mathcal{L}^{\wedge}_{m} \coloneq \left\{ \Phi \in \bigoplus\limits_{j_{1}\ldots j_{k}}\partial_{j_{1} \ldots j_{k}}\mathcal{L}^{\otimes}_{m} \ | \ \Phi \text{ satisfies } \eqref{EQ: AntisymmetricRelationsGeneral} \right\} = \Pi^{\wedge}_{m} \bigoplus\limits_{j_{1}\ldots j_{k}}\partial_{j_{1} \ldots j_{k}}\mathcal{L}^{\otimes}_{m}.
\end{equation}
Clearly, $\mathcal{L}^{\wedge}_{m}$ decomposes into the inner orthogonal sum of these subspaces,
\begin{equation}
	\label{EQ: DecompositionAntisymmetricSubspaceDelayCompound}
	\mathcal{L}^{\wedge}_{m} = \bigoplus_{k=0}^{m} \partial_{k}\mathcal{L}^{\wedge}_{m}.
\end{equation}

We say that $k$ is \textit{improper} if $\partial_{k}\mathcal{L}^{\wedge}_{m}$ is the zero subspace. Otherwise, we say that $k$ is \textit{proper}. For example, when $n=1$, Corollary \ref{COR: AntisymmetricRelationsScalar} gives that any $k \leq m-2$ is improper, and only $k=m-1$ and $k=m$ are proper. More generally, it can be shown that any $k \geq m-n$ is proper, and $k < m-n$ is improper, where the latter immediately follows from \eqref{EQ: AntisymmetricRestrictionRelationValues}.

Below we will identify $\mathbb{H}^{\otimes m}$ and $\mathbb{H}^{\wedge m}$ with $\mathcal{L}^{\otimes m}$ and $\mathcal{L}^{\wedge m}$ according to the isomorphism \eqref{EQ: IsomorphismHWedgeLm} and its restriction \eqref{EQ: DelayCompoundIsomorphismAntisymmetricDecomp}, respectively. Moreover, we will use the same notations for the corresponding additive and multiplicative compound operators, semigroups, and cocycles induced from the abstract spaces via the isomorphisms.

To describe the infinitesimal generator of the $m$-fold compound cocycle $\Xi_{m}$ in a form similar to \eqref{EQ: DelayLinearCocAbsract}, we need to introduce the corresponding control and measurement operators. It may be convenient for the reader to have in mind the final result \eqref{EQ: InfinitesimaldDescriptionXim}.

\subsection{Induced control operators on tensor and exterior products}
\label{SUBSEC: InducedControlOperators}
For any $j \in \{1,\ldots,m\}$, we set $\mathbb{R}_{1,j} \coloneq (\mathbb{R}^{n})^{\otimes(j-1)}$, $\mathbb{R}_{2,j} \coloneq (\mathbb{R}^{n})^{\otimes (m-j)}$, and $\mathbb{U}_{j} \coloneq \mathbb{R}_{1,j} \otimes \mathbb{U} \otimes \mathbb{R}_{2,j}$. Here $\mathbb{U} = \mathbb{R}^{r_{1}}$ is endowed with some inner product as it was defined above \eqref{EQ: DelayRnLinearized}.

With $\widetilde{B}$ from \eqref{EQ: DelayRnLinearized}, for all $k \in \{ 0, \ldots, m-1 \}$ and $1 \leq j_{1} < \cdots < j_{k} \leq m$, we associate a linear bounded operator $B^{j_{1} \ldots j_{k}}_{j}$, which takes an element $\Phi_{\mathbb{U}}$ from $L_{2}((-\tau,0)^{k}; \mathbb{U}_{j})$ to the element from the boundary subspace $\partial_{j_{1} \ldots j_{k}} \mathcal{L}^{\otimes}_{m}$ (see \eqref{EQ: TensorSpaceDelayCompoundDecompositionBoundarySubspaces}) defined by
\begin{equation}
	\label{EQ: BoundaryOpeartorCompoundDelay}
	\left(B^{j_{1}\ldots j_{k}}_{j}\Phi_{\mathbb{U}}\right)(\theta_{1},\ldots,\theta_{m}) \coloneq (\operatorname{Id}_{\mathbb{R}_{1,j}} \otimes \widetilde{B} \otimes \operatorname{Id}_{ \mathbb{R}_{2,j} } )\Phi_{\mathbb{U}}(\theta_{j_{1}},\ldots,\theta_{j_{k}})
\end{equation}
for almost all $(\theta_{1},\ldots,\theta_{m}) \in \mathcal{B}^{(m)}_{j_{1}\ldots j_{k}}$ in the sense of the $k$-dimensional Lebesgue measure on $\mathcal{B}^{(m)}_{j_{1}\ldots j_{k}}$, see \eqref{EQ: DefinitionOfBoundaryFace}.

We define the \textit{control space} $\mathbb{U}^{\otimes}_{m}$ via the outer orthogonal sum as follows:
\begin{equation}
	\label{EQ: ControlSpaceDelayCompoundDefinition}
	\mathbb{U}^{\otimes}_{m} \coloneq \bigoplus_{k=0}^{m-1}\bigoplus_{j_{1} \ldots j_{k}} \bigoplus_{j \notin \{ j_{1},\ldots,j_{k} \}} L_{2}((-\tau,0)^{k};\mathbb{U}_{j}).
\end{equation}
For convenience, we write $\eta = (\eta^{j}_{j_{1}\ldots j_{k}})$ for an element of $\mathbb{U}^{\otimes}_{m}$, where each $\eta^{j}_{j_{1}\ldots j_{k}}$ belongs to the corresponding summand from \eqref{EQ: ControlSpaceDelayCompoundDefinition}. 

Now the \textit{control operator} $B^{\otimes}_{m} \in \mathcal{L}(\mathbb{U}^{\otimes}_{m};\mathcal{L}^{\otimes}_{m})$ associated with $\widetilde{B}$ is defined by
\begin{equation}
	\label{EQ: OperatorBDelayCompoundDefinition}
	B^{\otimes}_{m}\eta \coloneq  \sum_{k=0}^{m-1}\sum_{j_{1} \ldots j_{k}}\sum_{j \notin \{j_{1},\ldots,j_{k}\}} B^{j_{1}\ldots j_{k}}_{j}\eta^{j}_{j_{1}\ldots j_{k}} \quad \text{for} \quad \eta = (\eta^{j}_{j_{1}\ldots j_{k}}) \in \mathbb{U}^{\otimes}_{m},
\end{equation}
where the sum is just a sum in $\mathcal{L}^{\otimes}_{m}$. In more detail, the inner sum is the sum in $\partial_{j_{1} \ldots j_{k}} \mathcal{L}^{\otimes}_{m}$, and the other sums can be understood according to \eqref{EQ: TensorSpaceDelayCompoundDecompositionBoundarySubspaces}.

Now we are going to define an analog of $B^{\otimes}_{m}$ for the antisymmetric subspace $\mathcal{L}^{\wedge}_{m}$. First, consider $\eta=(\eta^{j}_{j_{1}\ldots j_{k}}) \in \mathbb{U}^{\otimes}_{m}$ satisfying analogous to \eqref{EQ: AntisymmetricRelationsGeneral} antisymmetric relations. Namely, for all $k \in \{0, \ldots, m-1\}$, $1 \leq j_{1} < \cdots < j_{k} \leq m$, $j \notin \{ j_{1},\ldots, j_{k} \}$, and $\sigma \in \mathbb{S}_{m}$, we require $\eta$ to satisfy
\begin{equation}
	\label{EQ: ControlSpaceRelationsAntisymmetricCompoundDelay}
	\eta^{j}_{j_{1}\ldots j_{k}} = (-1)^{\sigma}T_{\sigma} \Theta^{(k)}_{\bar{\sigma}} \eta^{\sigma(j)}_{\sigma(j_{\bar{\sigma}(1)})\ldots \sigma(j_{\bar{\sigma}(k)})},
\end{equation}
where $\bar{\sigma} \in \mathbb{S}_{k}$ is such that $\sigma(j_{\bar{\sigma}(1)}) < \cdots < \sigma(j_{ \bar{\sigma}(k)})$. Here $\Theta^{(k)}_{\bar{\sigma}}$ is analogous to the one used in \eqref{EQ: AntisymmetricRelationsGeneral}, and $T_{\sigma} \colon \mathbb{U}_{j} \to \mathbb{U}_{\sigma^{-1}(j)}$ is defined similarly to \eqref{EQ: TranspositionOperatorOnValues} for any $j \in \{1,\ldots,m\}$.

\begin{remark}
	Relations from \eqref{EQ: ControlSpaceRelationsAntisymmetricCompoundDelay} may be simplified in the case $n=1$, see Proposition \ref{PROP: ControlSpaceCompoundDelayDescriptionScalarCase} below.
\end{remark}

Recall that $k \in \{ 0, \ldots, m\}$ is called improper if $\partial_{k}\mathcal{L}^{\wedge}_{m}$ from \eqref{EQ: AntisymmetricSubspaceOverKfaces} is zero. Now we define $\mathbb{U}^{\wedge}_{m}$ as follows:
\begin{equation}
	\label{EQ: AntisymmetricControlSpaceDefinition}
	\begin{split}
		\mathbb{U}^{\wedge}_{m}\coloneq\{ \eta=(\eta^{j}_{j_{1}\ldots j_{k}}) \in \mathbb{U}^{\otimes}_{m} \ | \ &\eta \ \text{satisfies} \ \eqref{EQ: ControlSpaceRelationsAntisymmetricCompoundDelay} \ \text{and } \ \\&\eta^{j}_{j_{1}\ldots j_{k}}=0 \ \text{for improper} \ k \}.
	\end{split}
\end{equation}

Let $B^{\wedge}_{m}$ denote the restriction to $\mathbb{U}^{\wedge}_{m}$ of the operator $B^{\otimes}_{m}$ from \eqref{EQ: OperatorBDelayCompoundDefinition}. By \cite[Proposition 6.1]{Anikushin2023Comp}, we have that $B^{\otimes}_{m} \eta \in \mathcal{L}^{\wedge}_{m}$ provided that $\eta \in \mathbb{U}^{\otimes}_{m}$ satisfies \eqref{EQ: ControlSpaceRelationsAntisymmetricCompoundDelay}, and, as a consequence, $B^{\wedge}_{m}$ belongs to $\mathcal{L}(\mathbb{U}^{\wedge}_{m};\mathcal{L}^{\wedge}_{m})$. For a general $\eta$ satisfying \eqref{EQ: ControlSpaceRelationsAntisymmetricCompoundDelay}, it is not necessary for $\eta^{j}_{j_{1} \ldots j_{k}}$ to be zero for improper $k$. However, the cumulative impact via \eqref{EQ: OperatorBDelayCompoundDefinition} of such components is an element of $\partial_{k}\mathcal{L}^{\wedge}_{m}$, and it vanishes for improper $k$ since $B^{\otimes}_{m} \eta \in \mathcal{L}^{\wedge}_{m}$. This is why we force these components to be zero in the definition \eqref{EQ: AntisymmetricControlSpaceDefinition}. See Remark \ref{REM: ZeroingControlComponents} for a more specific example.

\subsection{Induced measurement operators on tensor and exterior products}
\label{SUBSEC: InducedMeasurementOperators}
Applying the Riesz representation theorem to the operator $C$ from \eqref{EQ: DelayRnLinearized}, we get an $(r_{2} \times n)$-matrix-valued function $\gamma(\cdot)$ of bounded variation on $ [-\tau,0]$ such that
\begin{equation}
	\label{EQ: MeasurementOperatorCviaGamma}
	C\phi = \int_{-\tau}^{0}d\gamma(\theta)\phi(\theta) \qquad \text{for any} \quad \phi\in C([-\tau,0];\mathbb{R}^{n}).
\end{equation}

Recall that $\mathbb{R}_{1,j} = (\mathbb{R}^{n})^{\otimes (j-1)}$ and $\mathbb{R}_{2,j} = (\mathbb{R}^{n})^{\otimes (m-j)}$ and set $\mathbb{M}_{j} \coloneq \mathbb{R}_{1,j} \otimes \mathbb{M} \otimes \mathbb{R}_{2,j}$, where $\mathbb{M} = \mathbb{R}^{r_{2}}$ as in \eqref{EQ: DelayRnLinearized}. Then for all $k \in \{ 0, \ldots, m-1\}$, $J \in \{1,\ldots,k+1\}$, and $j \in \{ 1,\ldots,m \}$, we define an operator $C^{(k)}_{j,J}$ taking a function $\Phi \in C([-\tau,0]^{k+1};(\mathbb{R}^{n})^{\otimes m})$ to an element $C^{(k)}_{j,J}\Phi$ of $C([-\tau,0]^{k};\mathbb{M}_{j} )$ as follows:
\begin{equation}
	\label{EQ: OperatorCkjJDefinition}
	(C^{(k)}_{j,J}\Phi)(\theta_{1},\ldots,\hat{\theta}_{J},\ldots,\theta_{k+1}) \coloneq \int_{-\tau}^{0}d\gamma_{j}(\theta_{J})\Phi(\theta_{1},\ldots,\theta_{k+1}),
\end{equation}
where $\gamma_{j}(\theta) = \operatorname{Id}_{\mathbb{R}_{1,j}} \otimes \gamma(\theta) \otimes \operatorname{Id}_{\mathbb{R}_{2,j}}$ for $\theta \in [-\tau,0]$. In other words, the operator $C^{(k)}_{j,J}$ integrates over the $J$th argument with respect to the operator-valued measure $d\gamma_{j}$, which acts by $\gamma$ on the $j$th component of the tensor product in the space of values.

We need to consider $C^{(k)}_{j,J}$ in a wider context. For this, we define the space $\mathbb{E}_{k+1}((\mathbb{R}^{n})^{\otimes m})$ of all functions $\Phi \in L_{2}((-\tau,0)^{k+1}; (\mathbb{R}^{n})^{\otimes m})$ such that for any $j \in \{1, \ldots, k+1\}$ there exists $\Phi^{b}_{j}(\cdot) \in C([-\tau,0];L_{2}((-\tau,0)^{k};(\mathbb{R}^{n})^{\otimes m}))$ satisfying the identity
\begin{equation}
	\label{EQ: SpaceEmDelayPhiBDefinition}
	\restr{\Phi}{\mathcal{B}^{(k+1)}_{\hat{j}}+\theta e_{j}} = \Phi^{b}_{j}(\theta) \qquad \text{for almost all} \quad \theta \in [-\tau,0],
\end{equation}
where $e_{j}$ is the $j$th vector in the standard basis of $\mathbb{R}^{k+1}$, and we naturally identify $\mathcal{B}^{(k+1)}_{\hat{j}}+\theta e_{j}$ with $[-\tau,0]^{k}$ by omitting the $j$th coordinate.

In the above notations, we endow the space $\mathbb{E}_{k+1}( (\mathbb{R}^{n})^{\otimes m})$ with the norm
\begin{equation}
	\label{EQ: NormInSpaceEmDelay}
	\| \Phi \|_{\mathbb{E}_{k+1}( (\mathbb{R}^{n})^{\otimes m})} \coloneq \sup_{1 \leq j \leq k+1} \sup_{\theta \in [-\tau,0]}\| \Phi^{b}_{j}(\theta) \|_{L_{2}( (-\tau,0)^{k}; (\mathbb{R}^{n})^{\otimes m} )}
\end{equation}
that clearly makes it a Banach space. We also set $\mathbb{E}_{0}((\mathbb{R}^{n})^{\otimes m}) \coloneq (\mathbb{R}^{n})^{\otimes m}$.

Since $\Phi^{b}_{j}(\theta)$ continuously depends on $\theta \in [-\tau,0]$, it is not hard to show that the space $C([-\tau,0]^{k+1}; (\mathbb{R}^{n})^{\otimes m})$ is dense in $\mathbb{E}_{k+1}( (\mathbb{R}^{n})^{\otimes m})$. We have the following proposition.
\begin{proposition}\cite[Theorem A.3]{Anikushin2023Comp}
	\label{TH: OperatorCExntesionOntoEk}
	The operator $C^{(k)}_{j,J}$ defined by \eqref{EQ: OperatorCkjJDefinition} can be extended to a bounded operator from $\mathbb{E}_{k+1}((\mathbb{R}^{n})^{\otimes m})$ to $L_{2}((-\tau,0)^{k};\mathbb{M}_{j})$. Moreover, its norm does not exceed the total variation $\operatorname{Var}_{[-\tau,0]}(\gamma)$ of $\gamma$ on $[-\tau,0]$, where $\gamma$ is given by \eqref{EQ: MeasurementOperatorCviaGamma}.
\end{proposition}

Let us define a Banach space $\mathbb{E}^{\otimes}_{m}$ via the outer direct sum as follows:
\begin{equation}
	\label{EQ: SpaceEmDefinitionDelayCompoundGeneral}
	\mathbb{E}^{\otimes}_{m} \coloneq \bigoplus_{k=0}^{m} \bigoplus_{j_{1}\ldots j_{k}} \mathbb{E}_{k}((\mathbb{R}^{n})^{\otimes m}).
\end{equation}
We endow $\mathbb{E}^{\otimes}_{m}$ with any of the standard norms and embed it into $\mathcal{L}^{\otimes}_{m}$ by naturally sending each element from the $j_{1}\ldots j_{k}$th summand in \eqref{EQ: SpaceEmDefinitionDelayCompoundGeneral} to the boundary subspace $\partial_{j_{1}\ldots j_{k}}\mathcal{L}^{\otimes}_{m}$ from \eqref{EQ: TensorSpaceDelayCompoundDecompositionBoundarySubspaces}. Moreover, let $\mathbb{E}^{\wedge}_{m}$ be the subspace of $\mathbb{E}^{\otimes}_{m}$ that is mapped into $\mathcal{L}^{\wedge}_{m}$ under the embedding. We will identify the spaces and their images under the embedding. Then we just have $\mathbb{E}^{\wedge}_{m} = \mathbb{E}^{\otimes}_{m} \cap \mathcal{L}^{\wedge}_{m}$.

Analogously to the control space $\mathbb{U}^{\otimes}_{m}$, we introduce the \textit{measurement space} $\mathbb{M}^{\otimes}_{m}$ via the outer orthogonal sum
\begin{equation}
	\label{EQ: MeasurementSpaceCompoundTensorDefinition}
	\mathbb{M}^{\otimes}_{m}\coloneq\bigoplus_{k=0}^{m-1}\bigoplus_{j_{1}\ldots j_{k}} \bigoplus_{j \notin \{j_{1},\ldots,j_{k}\}}L_{2}((-\tau,0)^{k};\mathbb{M}_{j}).
\end{equation}

Given $k \in \{0,\ldots,m-1\}$, $1 \leq j_{1} < \cdots < j_{k} \leq m$, and $j \notin \{ j_{1},\ldots,j_{k} \}$, we define $J(j) = J(j;j_{1},\ldots,j_{k})$ as the integer $J$ such that $j$ is the $J$th element in the set $\{ j,j_{1},\ldots, j_{k} \}$ arranged by increasing. Now define the \textit{measurement operator} $C^{\otimes}_{m} \in \mathcal{L}(\mathbb{E}^{\otimes}_{m};\mathbb{M}^{\otimes}_{m})$ by
\begin{equation}
	\label{EQ: DelayMeasumerentAntisymmetricSumSpace}
	C^{\otimes}_{m}\Phi \coloneq \sum_{k=0}^{m-1}\sum_{j_{1}\ldots j_{k}}\sum_{j \notin \{j_{1},\ldots,j_{k}\}} C^{(k)}_{j,J(j)}R_{jj_{1}\ldots j_{k}}\Phi,
\end{equation}
where the sum is taken in $\mathbb{M}^{\otimes}_{m}$ according to \eqref{EQ: MeasurementSpaceCompoundTensorDefinition}, and the action of $C^{(k)}_{j,J(j)}$ is understood in the sense of Proposition \ref{TH: OperatorCExntesionOntoEk}.

Let us define an analog of the above constructions for the antisymmetric case. For this, we consider such elements $M=(M^{j}_{j_{1}\ldots j_{k}}) \in \mathbb{M}^{\otimes}_{m}$ that satisfy for all $k \in \{0, \ldots, m-1\}$, $1 \leq j_{1} < \cdots < j_{k} \leq m$, $j \notin \{ j_{1},\ldots, j_{k} \}$, and $\sigma \in \mathbb{S}_{m}$ the relations
\begin{equation}
	\label{EQ: MeasurementSpaceRelationsAntisymmetricCompoundDelay}
		M^{j}_{j_{1}\ldots j_{k}} = (-1)^{\sigma}T_{\sigma} \Theta^{(k)}_{\bar{\sigma}} M^{\sigma(j)}_{\sigma(j_{\bar{\sigma}(1)})\ldots \sigma(j_{\bar{\sigma}(k)})},
\end{equation}
where $\bar{\sigma} \in \mathbb{S}_{k}$ is such that $\sigma(j_{\bar{\sigma}(1)}) < \cdots < \sigma(j_{ \bar{\sigma}(k)})$. Here $\Theta^{(k)}_{\bar{\sigma}}$ is analogous to the one used in \eqref{EQ: AntisymmetricRelationsGeneral}, and $T_{\sigma} \colon \mathbb{M}_{j} \to \mathbb{M}_{\sigma^{-1}(j)}$ is defined similarly to \eqref{EQ: TranspositionOperatorOnValues} for any $j \in \{1,\ldots,m\}$.

Now we define $\mathbb{M}^{\wedge}_{m}$ as follows:
\begin{equation}
	\label{EQ: AntisymmetricMeasurementSpaceDefinition}
	\begin{split}
		\mathbb{M}^{\wedge}_{m}\coloneq\{ M=(M^{j}_{j_{1}\ldots j_{k}}) \in \mathbb{M}^{\otimes}_{m} \ | \ &M \ \text{satisfies} \ \eqref{EQ: MeasurementSpaceRelationsAntisymmetricCompoundDelay} \ \text{and} \ \\&M^{j}_{j_{1}\ldots j_{k}}=0 \ \text{for improper} \ k \}.
	\end{split}
\end{equation}

Let $C^{\wedge}_{m}$ be given by restricting $C^{\otimes}_{m}$ to $\mathbb{E}^{\wedge}_{m}$ and zeroing the components for improper $k$, i.e.,
\begin{equation}
	(C^{\wedge}_{m}\Phi)^{j}_{j_{1}\ldots j_{k}} \coloneq \begin{cases}
		(C^{\otimes}_{m}\Phi)^{j}_{j_{1}\ldots j_{k}} &\text{ if $k$ is proper},\\
		0 &\text{ otherwise}.
	\end{cases}
\end{equation}
Similarly to the operator $B^{\wedge}_{m}$, it can be shown that $C^{\wedge}_{m}$ belongs to $\mathcal{L}(\mathbb{E}^{\wedge}_{m};\mathbb{M}^{\wedge}_{m})$, see \cite[Proposition 6.2]{Anikushin2023Comp}.

\subsection{Infinitesimal description of compound cocycles}

For any $\wp \in \mathcal{P}$, $k \in \{0,\ldots, m-1 \}$, and $j \in \{1, \ldots, m\}$, we consider an operator $F'_{j}(\wp)$ induced by $F'(\wp)$ from \eqref{EQ: DelayRnLinearized}, which takes $\Phi_{\mathbb{M}} \in L_{2}((-\tau,0)^{k}; \mathbb{M}_{j} )$ to an element $F'_{j}(\wp) \Phi_{\mathbb{M}}$ from $L_{2}((-\tau,0)^{k};\mathbb{U}_{j})$ given by
\begin{equation}
	\label{EQ: DelayCompoundOperatorFJdefinition}
	(F'_{j}(\wp) \Phi_{\mathbb{M}})(\theta_{1},\ldots,\theta_{k}) \coloneq ( \operatorname{Id}_{\mathbb{R}_{1,j}} \otimes F'(\wp) \otimes \operatorname{Id}_{ \mathbb{R}_{2,j} } )\Phi_{\mathbb{M}}(\theta_{1},\ldots,\theta_{k}),
\end{equation}
where, as usual, $\mathbb{R}_{1,j}=(\mathbb{R}^{n})^{\otimes (j-1)}$, $\mathbb{R}_{2,j} = (\mathbb{R}^{n})^{\otimes (m-j)}$, and $\mathbb{U}_{j} = \mathbb{R}_{1,j} \otimes \mathbb{U} \otimes \mathbb{R}_{2,j}$ with $\mathbb{U}$ defined above \eqref{EQ: DelayRnLinearized}. Note that we omit the dependence of $F'_{j}(\wp)$ on $k$ for convenience.

These operators induce an operator $F^{\otimes}_{m}(\wp)$ from $\mathbb{M}^{\otimes}_{m}$ to $\mathbb{U}^{\otimes}_{m}$ as follows:
\begin{equation}
	\label{EQ: FPrimeTensorOperatorDef}
	F^{\otimes}_{m} M = \sum_{k=0}^{m-1}\sum_{j_{1}\ldots j_{k}} \sum_{ j \notin \{j_{1}\dots j_{k}\} } F'_{j}(\wp) M^{j}_{j_{1}\ldots j_{k}}
\end{equation}
for all $M = (M^{j}_{j_{1}\ldots j_{k}}) \in \mathbb{M}^{\otimes}_{m}$. Note that the overall sum is taken in $\mathbb{U}^{\otimes}_{m}$ according to \eqref{EQ: ControlSpaceDelayCompoundDefinition}. Moreover, restricting $F^{\otimes}_{m}(\wp)$ to $\mathbb{M}^{\wedge}_{m}$ gives a well-defined mapping $F^{\wedge}_{m}(\wp)$ into $\mathbb{U}^{\wedge}_{m}$. 

Recall that $A$ given by \eqref{EQ: OperatorAScalarDelayEquations} is the generator of a $C_{0}$-semigroup in $\mathbb{H}$. Let $A^{[\otimes m]}$ be the $m$-fold additive compound of $A$. Since we identified $\mathbb{H}^{\otimes m}$ with $\mathcal{L}^{\otimes}_{m}$ via the isomorphism from \eqref{EQ: IsomorphismHWedgeLm}, it is reasonable to give the description of $A^{[\otimes m]}$ (in particular, $A^{[\wedge m]}$) in terms of the space $\mathcal{L}^{\otimes}_{m}$.

For this, for any $k \in \{1,\ldots,m\}$, we consider the space $\mathcal{W}^{2}_{D}( (-\tau,0)^{k}; (\mathbb{R}^{n})^{\otimes m} )$ consisting of all $\Phi \in L_{2}((-\tau,0)^{k}; (\mathbb{R}^{n})^{\otimes m})$ having a square-summable diagonal derivative (in the generalized sense), i.e.,
\begin{equation}
	\sum_{j=1}^{k} \frac{\partial}{\partial\theta_{j}}\Phi \in L_{2}( (-\tau,0)^{k};(\mathbb{R}^{n})^{\otimes m}),
\end{equation}
and such that the trace (restriction) of $\Phi$ to each $\mathcal{B}^{(k)}_{\hat{j}}$, where $j \in \{1,\ldots,m\}$, belongs to $L_{2}((-\tau,0)^{k-1};(\mathbb{R}^{n})^{\otimes m})$ after $\mathcal{B}^{(k)}_{\hat{j}}$ is naturally identified with $(-\tau,0)^{k-1}$. We call $\mathcal{W}^{2}_{D}( (-\tau,0)^{k}; (\mathbb{R}^{n})^{\otimes m} )$ a \textit{diagonal Sobolev space}. It can be shown that, when being endowed with a natural norm\footnote{Its square is given by the sum of squares of $L_{2}$-norms of diagonal derivatives and $L_{2}$-norms of traces on the $(k-1)$-faces $\mathcal{B}^{(k)}_{\hat{j}}$, see \cite[Proposition A.1]{Anikushin2023Comp}.}, it becomes a Hilbert space, which naturally continuously embeds into the space $\mathbb{E}_{k}((\mathbb{R}^{n})^{\otimes m})$ from \eqref{EQ: NormInSpaceEmDelay}, see \cite[Proposition A.2]{Anikushin2023Comp}.

Analogously to the operators $C^{(k)}_{j,J}$ from \eqref{EQ: OperatorCkjJDefinition}, one can define operators $\widetilde{A}^{(k)}_{j,J}$ associated with $\widetilde{A}$ from \eqref{EQ: DelayRnLinearized}. Then we have the following theorem.
\begin{theorem}\cite[Theorems 4.2 and 4.3]{Anikushin2023Comp}
	\label{TH: AdditiveCompoundDelayDescription}
	Let $A^{[\otimes m]}$ be the $m$-fold additive compound of $A$ given by \eqref{EQ: OperatorAScalarDelayEquations}. Then for each $\Phi \in \mathcal{D}(A^{[\otimes m]})$, any restriction $R_{j_{1}\ldots j_{k}}\Phi$ with $k \in \{1,\ldots,m\}$ belongs to $\mathcal{W}^{2}_{D}((-\tau,0)^{k};(\mathbb{R}^{n})^{\otimes m})$. Moreover, for such $\Phi$ and all $k \in \{0, \ldots, m\}$ and $1 \leq j_{1} < \cdots < j_{k} \leq m$, we have\footnote{Here $R_{j_{1}\ldots j_{k}}\Phi$ is considered as a function of $\theta_{1},\ldots,\theta_{k}$. Recall also that $J(j) = J(j;j_{1},\ldots,j_{k})$ is the integer $J$ such that $j$ is the $J$th element in the set $\{ j,j_{1},\ldots, j_{k} \}$ arranged by increasing.}
	\begin{equation}
		\label{EQ: OperatorAmFormula}
		R_{j_{1}\ldots j_{k}}\left( A^{[\otimes m]} \Phi  \right) = 
		\sum_{l=1}^{k}\frac{\partial}{\partial \theta_{l}}R_{j_{1}\ldots j_{k}}\Phi + \sum_{j \notin \{ j_{1},\ldots,j_{k} \}} \widetilde{A}^{(k)}_{j,J(j)}R_{jj_{1}\ldots j_{k}}\Phi.
	\end{equation}
	Moreover, the graph norm on $\mathcal{D}(A^{[\otimes m]})$ is equivalent to the norm $\| \cdot \|_{\mathcal{W}^{2}_{D}}$, where
	\begin{equation}
		\| \Phi \|^{2}_{\mathcal{W}^{2}_{D}} \coloneq \sum_{k=1}^{m}\sum_{j_{1} \ldots j_{k}} \| R_{j_{1}\ldots j_{k}} \Phi \|^{2}_{ \mathcal{W}^{2}_{D}((-\tau,0)^{k}; (\mathbb{R}^{n})^{\otimes m} ) }.
	\end{equation}
\end{theorem}
\begin{remark}
	\label{REM: DescriptionOfDomainAdditiveCompound}
	Theorem \ref{TH: AdditiveCompoundDelayDescription} does not fully describe the domain $\mathcal{D}(A^{[\otimes m]})$. In fact, functions from the diagonal Sobolev space $\mathcal{W}^{2}_{D}( (-\tau,0)^{k}; (\mathbb{R}^{n})^{\otimes m} )$ have well-defined $L_{2}$-traces on sections of $[-\tau,0]^{k}$ by hyperplanes transversal to the diagonal line in $\mathbb{R}^{k}$. Then for each $\Phi \in \mathcal{D}(A^{[\otimes m]})$, the trace of $R_{j_{1}\ldots j_{k}}\Phi$ on $\mathcal{B}^{(k)}_{\hat{l}}$ agrees with the restriction $R_{j_{1}\ldots \hat{j_{l}} \ldots j_{k}}\Phi$ for any $l \in \{1,\ldots,k\}$. This was proven in \cite[Theorems 4.2 and 4.3]{Anikushin2023Comp}. We do not get into details of this fact in the present work although we will recall it in the proof of Proposition \ref{PROP: IntegralHeadComputation}. Moreover, this agreement of traces and restrictions completely characterizes the domain $\mathcal{D}(A^{[\otimes m]})$, see \cite[Remark 4.2]{Anikushin2023Comp}.
\end{remark}

In particular, we have the following continuous embeddings
\begin{equation}
	\mathcal{D}(A^{[\otimes m]}) \subset \mathbb{E}^{\otimes}_{m} \subset \mathcal{L}^{\otimes}_{m} \qquad \text{and} \qquad \mathcal{D}(A^{[\wedge m]}) \subset \mathbb{E}^{\wedge}_{m} \subset \mathcal{L}^{\wedge}_{m},
\end{equation}
where the intermediate (or auxiliary) Banach spaces $\mathbb{E}^{\otimes}_{m}$ and $\mathbb{E}^{\wedge}_{m}$ are important for our study, see Theorems \ref{TH: QuadraticFunctionalDelayCompoundTheorem} and \ref{TH: ResolventDelayCompoundBound}.

Using the above introduced notations, we may give the infinitesimal description of the $m$-fold multiplicative compound $\Xi_{m}$ of $\Xi$ in the space $\mathcal{L}^{\wedge}_{m}$ as follows:
\begin{equation}
	\label{EQ: InfinitesimaldDescriptionXim}
	\dot{\Phi}(t) = A^{[\wedge m]}\Phi(t) + B^{\wedge}_{m} F^{\wedge}_{m}(\pi^{t}(\wp))C^{\wedge}_{m}\Phi(t).
\end{equation}

By \cite[Theorem 6.1]{Anikushin2023Comp}, we have that for any $\Phi_{0} = \xi_{1} \wedge \cdots \wedge \xi_{m}$ with $\xi_{1},\ldots,\xi_{m} \in \mathcal{D}(A)$, there exists a classical solution, i.e., $\Phi(\cdot) \in C^{1}([0,\infty);\mathcal{L}^{\wedge}_{m}) \cap C([0,\infty);\mathcal{D}(A^{[\wedge m]}))$ such that $\Phi(0) = \Phi_{0}$ and $\Phi(t)$ satisfies \eqref{EQ: InfinitesimaldDescriptionXim} for all $t \geq 0$. It is given by the trajectory of $\Phi_{0}$ over $\wp$, i.e., $\Xi^{t}_{m}(\wp,\Phi_{0}) = \Phi(t)$ for $t \geq 0$. Consequently, on a dense subset of $\mathcal{L}^{\wedge}_{m}$, the trajectories of $\Xi_{m}$ are classical solutions to \eqref{EQ: InfinitesimaldDescriptionXim}. 

\subsection{Frequency conditions for the uniform exponential stability of compound cocycles}
To describe conditions for the uniform exponential stability of $\Xi_{m}$, we consider it as a nonautonomous perturbation of the $C_{0}$-semigroup $G^{\wedge m}$ generated by $A^{[\wedge m]}$. From this view, \eqref{EQ: InfinitesimaldDescriptionXim} describes the generator of $\Xi_{m}$ as a nonautonomous boundary perturbation of $A^{[\wedge m]}$. 

We are going to consider quadratic constraints associated with such perturbations. For this, let $\mathcal{G}(M,\eta)$ be a bounded quadratic form of $M \in \mathbb{M}^{\wedge}_{m}$ and $\eta \in \mathbb{U}^{\wedge}_{m}$. Then we consider a quadratic form $\mathcal{F}$ on $\mathbb{E}^{\wedge}_{m} \times \mathbb{U}^{\wedge}_{m}$ defined as follows:
\begin{equation}
	\label{EQ: QuadraticFormRealCompoundDelayGeneral}
	\mathcal{F}(\Phi,\eta) \coloneq \mathcal{G}(C^{\wedge}_{m}\Phi,\eta) \qquad \text{for} \quad \Phi \in \mathbb{E}^{\wedge}_{m} \quad \text{and} \quad \eta \in \mathbb{U}^{\wedge}_{m}.
\end{equation}
We say that $\mathcal{F}$ defines a \textit{quadratic constraint} with respect to \eqref{EQ: InfinitesimaldDescriptionXim} if $\mathcal{F}(\Phi,\eta) \geq 0$ for any $\eta = F^{\wedge}_{m}(\wp)C^{\wedge}_{m}\Phi$ with arbitrary $\Phi \in \mathbb{E}^{\wedge}_{m}$ and $\wp \in \mathcal{P}$ and, in addition, $\mathcal{F}(\Phi,0) \geq 0$.

Let us describe the Hermitian extension $\mathcal{F}^{\mathbb{C}}$ of $\mathcal{F}$ defined on the complexifications $(\mathbb{E}^{\wedge}_{m})^{\mathbb{C}}$ and $(\mathbb{U}^{\wedge}_{m})^{\mathbb{C}}$ of $\mathbb{E}^{\wedge}_{m}$ and $\mathbb{U}^{\wedge}_{m}$, respectively, as $\mathcal{F}^{\mathbb{C}}(\Phi_{1} + i\Phi_{2}, \eta_{1}+i\eta_{2}) \coloneq \mathcal{F}(\Phi_{1},\eta_{1}) + \mathcal{F}(\Phi_{2},\eta_{2})$ for all $\Phi_{1},\Phi_{2}\in \mathbb{E}^{\wedge}_{m}$ and $\eta_{1},\eta_{2} \in \mathbb{U}^{\wedge}_{m}$. Any $\mathcal{G}$ as above can always be represented as follows:
\begin{equation}
	\mathcal{G}(M,\eta) = \langle M, \mathcal{G}_{1} M \rangle_{\mathbb{M}^{\wedge}_{m}} + \langle \eta,\mathcal{G}_{2}M \rangle_{\mathbb{U}^{\wedge}_{m}} + \langle \eta, \mathcal{G}_{3}\eta \rangle_{ \mathbb{U}^{\wedge}_{m}},
\end{equation}
where $\mathcal{G}_{1} \in \mathcal{L}(\mathbb{M}^{\wedge}_{m})$ and $\mathcal{G}_{3} \in \mathcal{L}( \mathbb{U}^{\wedge}_{m})$ are self-adjoint and $\mathcal{G}_{2} \in \mathcal{L}(\mathbb{M}^{\wedge}_{m}; \mathbb{U}^{\wedge}_{m})$. Then for all $\Phi \in (\mathbb{E}^{\wedge}_{m})^{\mathbb{C}}$ and $\eta \in (\mathbb{U}^{\wedge}_{m})^{\mathbb{C}}$, the value $\mathcal{F}^{\mathbb{C}}(\Phi,\eta)$ is given by
\begin{equation}
	\label{EQ: QuadraticFormComplexificationCompoundDelayGeneral}
	\begin{split}
		\mathcal{F}^{\mathbb{C}}(\Phi,\eta) = \mathcal{G}^{\mathbb{C}}(C^{\wedge}_{m}\Phi,\eta) = \langle C^{\wedge}_{m}\Phi, \mathcal{G}_{1} C^{\wedge}_{m}\Phi \rangle_{(\mathbb{M}^{\wedge}_{m})^{\mathbb{C}}} + \operatorname{Re}\langle \eta, \mathcal{G}_{2} C^{\wedge}_{m}\Phi \rangle_{(\mathbb{U}^{\wedge}_{m})^{\mathbb{C}}} \\+ \langle \eta, \mathcal{G}_{3}\eta \rangle_{(\mathbb{U}^{\wedge}_{m})^{\mathbb{C}}},
	\end{split}
\end{equation}
where we omitted mentioning complexifications of the operators $C^{\wedge}_{m}$, $\mathcal{G}_{1}$, $\mathcal{G}_{2}$, and $\mathcal{G}_{3}$ for convenience. 

With each such $\mathcal{F}$ and $\nu_{0} \in \mathbb{R}$, we associate the following \textit{frequency inequality} on the line $-\nu_{0} + i \mathbb{R}$, avoiding the spectrum of $A^{[\wedge m]}$.
\begin{description}[before=\let\makelabel\descriptionlabel]
	\item[\textbf{(FI)}\refstepcounter{desccount}\label{DESC: FI}] For some $\delta>0$ and any $p$ with $\operatorname{Re}p = -\nu_{0}$, we have
	\begin{equation}
		\label{EQ: FreqIneqDelayCompoundGeneralForm}
		\mathcal{F}^{\mathbb{C}}(-(A^{[\wedge m]} - p I)^{-1}B^{\wedge}_{m}\eta , \eta ) \leq -\delta \left|\eta\right|^{2}_{(\mathbb{U}^{\wedge}_{m})^{\mathbb{C}}} \qquad \text{for any} \quad \eta \in \left(\mathbb{U}^{\wedge}_{m}\right)^{\mathbb{C}}.
	\end{equation}
\end{description}

Recall that $A$ generates an eventually compact semigroup $G$. Then $G^{\otimes m}$ (and, consequently, $G^{\wedge m}$) is also eventually compact. In particular, the spectrum of $A^{[\wedge m]}$ consists of eigenvalues with finite algebraic multiplicities. According to Proposition \ref{PROP: SpectralBoundAwedgeViaA}, the spectral bound $s(A^{[\wedge m]})$ of $A^{[\wedge m]}$ can be described as the sum of the first $m$ eigenvalues of $A$ if it has at least $m$ eigenvalues, or $s(A^{[\wedge m]}) = -\infty$ otherwise.

We have the following theorem.
\begin{theorem}\cite[Theorem 6.2]{Anikushin2023Comp}
	\label{TH: QuadraticFunctionalDelayCompoundTheorem}
	Let $\mathcal{F}$, as in \eqref{EQ: QuadraticFormRealCompoundDelayGeneral}, define a quadratic constraint with respect to \eqref{EQ: InfinitesimaldDescriptionXim}. Suppose that there exists $\nu_{0} \in \mathbb{R}$ such that $-\nu_{0} > s(A^{[\wedge m]})$ and
	\nameref{DESC: FI} is satisfied. Then there exists a bounded self-adjoint operator $P \in \mathcal{L}(\mathcal{L}^{\wedge}_{m})$ such that its quadratic form $V(\Phi)\coloneq \langle \Phi,P\Phi \rangle_{\mathcal{L}^{\wedge}_{m}}$ is positive-definite, and for some $\delta_{V}>0$ the compound cocycle $\Xi_{m}$ in $\mathcal{L}^{\wedge}_{m}$ generated by \eqref{EQ: InfinitesimaldDescriptionXim} satisfies
	\begin{equation}
		\label{EQ: QuadraticIneqDelayCompound}
		e^{2\nu_{0} t}V( \Xi^{t}_{m}(\wp, \Phi ) ) - V(\Phi) \leq -\delta_{V} \int_{0}^{t} e^{2\nu_{0} s} \left| \Xi^{s}_{m}(\wp,\Phi) \right|^{2}_{\mathcal{L}^{\wedge}_{m}}ds
	\end{equation}
	for any $t \geq 0$, $\wp \in \mathcal{P}$, and $\Phi \in \mathcal{L}^{\wedge}_{m}$.
\end{theorem}
\begin{proof}
	Let us give a sketch of the proof. For the existence of $P$, we apply \cite[Theorem 2.1]{Anikushin2020FreqDelay} to the pair $(A^{[\wedge m]}+\nu_{0} I, B^{\wedge}_{m})$ and the quadratic form $\mathcal{F}$. There are three key elements that constitute conditions of the theorem. Namely, 
	\begin{enumerate}
		\item[1).] Well-posedness of the integral quadratic functional associated with $\mathcal{F}$ and its computation via the Fourier transform, which is discussed around \cite[Lemma 6.2]{Anikushin2023Comp}. This is the heaviest part of the theory concerned with the structural Cauchy formula for linear inhomogeneous problems associated with $A^{[\wedge m]}+\nu_{0} I$ and its relation to pointwise measurement operators constituting the functional;
		\item[2).] Boundedness of the resolvent of $A^{[\wedge m]}+\nu_{0} I$ in $\mathcal{L}(\mathcal{L}^{\wedge}_{m};\mathbb{E}^{\wedge}_{m})$ uniformly on the line $-\nu_{0} + i\mathbb{R}$, which is guaranteed by \cite[Corollary 6.1]{Anikushin2023Comp}, see also Theorem \ref{TH: ResolventDelayCompoundBound} and Remark \ref{REM: UniformBoundResolventDelayCompound} of the present paper;
		\item[3).] Validity of the frequency inequality \eqref{EQ: FreqIneqDelayCompoundGeneralForm}.
	\end{enumerate}
	After \cite[Theorem 2.1]{Anikushin2020FreqDelay} is applied, the proof of \eqref{EQ: QuadraticIneqDelayCompound} is standard. 
\end{proof}

From \eqref{EQ: QuadraticIneqDelayCompound} we have that the cocycle $\Xi_{m}$ has a uniform exponential growth bounded from above by the exponent $\nu_{0}$. This is contained in the following corollary.
\begin{corollary}
	\label{COR: DelayCompoundUniformExponentialStability}
	Under \eqref{EQ: QuadraticIneqDelayCompound}, there exists a constant $M>0$ such that for all $\Phi \in \mathcal{L}^{\wedge}_{m}$ and $\wp \in \mathcal{P}$, we have
	\begin{equation}
		\label{EQ: CorExpStablCompCoc}
		| \Xi^{t}_{m}(\wp,\Phi) |_{\mathcal{L}^{\wedge}_{m}} \leq M e^{-\nu_{0} t} |\Phi|_{\mathcal{L}^{\wedge}_{m}} \qquad \text{for any} \quad t \geq 0.
	\end{equation}
    In particular, for $\nu_{0} > 0$, the cocycle $\Xi_{m}$ is uniformly exponentially stable.
\end{corollary}
\begin{proof}
	Since $\Xi_{m}$ is uniformly continuous, the value
	\begin{equation}
		M_{\Xi} \coloneq \sup_{\wp \in \mathcal{P}}\sup_{s \in [0,1]} \| \Xi^{s}_{m}(\wp,\cdot) \|_{\mathcal{L}(\mathcal{L}^{\wedge}_{m})}
	\end{equation}
    is finite. Then from the cocycle property for any $\Phi \in \mathcal{L}^{\wedge}_{m}$, $\wp \in \mathcal{P}$, $t \geq 0$, and $s_{0} \in [t,t+1]$, we have the following estimate:
    \begin{equation}
    	\label{EQ: CocUnifExpStabUnContAppl}
    	|\Xi^{t+1}_{m}(\wp, \Phi)|_{\mathcal{L}^{\wedge}_{m}} = |\Xi^{t+1-s_{0}}(\pi^{s_{0}}(\wp), \Xi^{s_{0}}_{m}(\wp,\Phi))|_{\mathcal{L}^{\wedge}_{m}} \leq M_{\Xi} \cdot |\Xi^{s_{0}}_{m}(\wp, \Phi)|_{\mathcal{L}^{\wedge}_{m}}.
    \end{equation}
	
	Using the positive-definiteness of $V$, \eqref{EQ: QuadraticIneqDelayCompound}, the mean value theorem, and \eqref{EQ: CocUnifExpStabUnContAppl}, for any $t \geq 0$ we get that
	\begin{equation}
		\begin{split}
			\delta^{-1}_{V}V(\Phi) \geq \int_{0}^{\infty} e^{2 \nu_{0} s}| \Xi^{s}_{m}(\wp,\Phi) |^{2}_{\mathcal{L}^{\wedge}_{m}}ds \geq \int_{t}^{t+1} e^{2 \nu_{0} s}| \Xi^{s}_{m}(\wp,\Phi) |^{2}_{\mathcal{L}^{\wedge}_{m}}ds =\\= e^{2 \nu_{0} s_{0}}| \Xi^{s_{0}}_{m}(\wp,\Phi) |^{2}_{\mathcal{L}^{\wedge}_{m}} \geq e^{2 \nu_{0} t} (e^{|\nu_{0}|}M_{\Xi})^{-2} | \Xi^{t+1}_{m}(\wp,\Phi) |^{2}_{\mathcal{L}^{\wedge}_{m}},
		\end{split}
	\end{equation}
    where $s_{0} \in [t,t+1]$. This gives that \eqref{EQ: CorExpStablCompCoc} is satisfied for a sufficiently large $M>0$.
\end{proof}

\begin{remark}
	\label{REM: RobustnessLargLyapExp}
	Note that \eqref{EQ: CorExpStablCompCoc} implies that the largest uniform Lyapunov exponent $\lambda_{1}(\Xi_{m})$ of $\Xi_{m}$ satisfies $\lambda_{1}(\Xi_{m}) \leq -\nu_{0}$ or, equivalently,
	\begin{equation}
		\label{EQ: LyapExpMtimesSumNegative}
		\lambda_{1}(\Xi) + \cdots + \lambda_{m}(\Xi) \leq -\nu_{0},
	\end{equation} 
    where $\lambda_{1}(\Xi), \lambda_{2}(\Xi), \ldots$ are the uniform Lyapunov exponents of $\Xi$ defined by induction from the relations $\lambda_{1}(\Xi) + \cdots + \lambda_{k}(\Xi) = \lambda_{1}(\Xi_{k})$ for $k=1, 2, \ldots$ (see \cite{Anikushin2023LyapExp, Temam1997}). In \cite{Anikushin2023LyapExp}, it is shown that the largest uniform exponent of $\Xi_{m}$ is upper semicontinuous with respect to $\Xi$ in an appropriate topology. In applications, where $\mathcal{P}$ is a positively invariant region for $\pi$ localizing an attractor, and $\Xi$ is the derivative cocycle for $\pi$ in this region, this gives the upper semicontinuity with respect to $C^{1}$-perturbations of $\pi$, which are uniformly small in a neighborhood of the attractor and preserve the invariance of $\mathcal{P}$. In particular, for $\nu_{0}>0$, the inequality \eqref{EQ: LyapExpMtimesSumNegative} gives negativity of the sum that is preserved under smallness of such perturbations. As we have discussed in the introduction, this is the condition that is verified in the works concerned with applications of the generalized Bendixson criterion.
\end{remark}

There is a natural choice of a quadratic constraint $\mathcal{F}$ for general $F'(\cdot)$ satisfying \eqref{EQ: LipschitzFprimeDelay}. Namely, for $\Phi \in \mathbb{E}^{\wedge}_{m}$ and $\eta \in \mathbb{U}^{\wedge}_{m}$, consider
\begin{equation}
	\label{EQ: DelayCompStandardChoiceQC}
	\mathcal{F}(\Phi,\eta) \coloneq \Lambda^{2} |C^{\wedge}_{m}\Phi|^{2}_{\mathbb{M}^{\wedge}_{m}} - |\eta|^{2}_{\mathbb{U}^{\wedge}_{m}}. 
\end{equation}
To see that it is indeed a quadratic constraint, note that $\eta = F^{\otimes}_{m}(\wp)C^{\otimes}_{m}\Phi$ with $\Phi \in \mathbb{E}^{\otimes}_{m}$ according to \eqref{EQ: DelayMeasumerentAntisymmetricSumSpace} and \eqref{EQ: FPrimeTensorOperatorDef} is equivalent to
\begin{equation}
	\label{EQ: ClosedFeedbackDelayCompoundRelation}
	\eta^{j}_{j_{1}\ldots j_{k}} = F'_{j}(\wp) C^{(k)}_{j,J(j)}R_{jj_{1}\ldots j_{k}}\Phi
\end{equation}
for any $k \in \{0,\ldots,m-1\}$, $1 \leq j_{1} < \cdots < j_{k} \leq m$, and $j \notin \{ j_{1},\ldots,j_{k} \}$. Here the norm of each operator $F'_{j}(\wp)$, due to its definition in \eqref{EQ: DelayCompoundOperatorFJdefinition}, coincides with the norm of $F'(\wp)$ and, therefore, is bounded from above by $\Lambda$. From this and the definitions of $\mathbb{U}^{\wedge}_{m}$ and $\mathbb{M}^{\wedge}_{m}$, see \eqref{EQ: AntisymmetricControlSpaceDefinition} and \eqref{EQ: AntisymmetricMeasurementSpaceDefinition}, it follows that $\mathcal{F}(\Phi,\eta) \geq 0$ for $\eta = F^{\wedge}_{m}(\wp)C^{\wedge}_{m}\Phi$. Since the inequality $\mathcal{F}(\Phi,0) \geq 0$ is obvious, $\mathcal{F}$ indeed defines a quadratic constraint with respect to \eqref{EQ: InfinitesimaldDescriptionXim}.

In terms of the transfer operator $W(p) \coloneq -C^{\wedge}_{m}(A^{[\wedge m]} - p I)^{-1}B^{\wedge}_{m}$, the frequency inequality \eqref{EQ: FreqIneqDelayCompoundGeneralForm} associated with $\mathcal{F}$ from \eqref{EQ: DelayCompStandardChoiceQC} takes the following form:
\begin{equation}
\label{EQ: DelayCompoundSmithIneqGeneral}
\sup_{\omega \in \mathbb{R}}\| W(-\nu_{0} + i \omega) \|_{\mathcal{L}((\mathbb{U}^{\wedge}_{m})^{\mathbb{C}};(\mathbb{M}^{\wedge}_{m})^{\mathbb{C}})} < \Lambda^{-1}.
\end{equation}
Note that the norms of $(A^{[\wedge m]}-pI)^{-1}$ in $\mathcal{L}(\mathcal{L}^{\wedge}_{m};\mathbb{E}^{\wedge}_{m})$ for $p=-\nu_{0} + i \omega$ are bounded uniformly in $\omega \in \mathbb{R}$, see Theorem \ref{TH: ResolventDelayCompoundBound} below. From this it is clear that \eqref{EQ: DelayCompoundSmithIneqGeneral} is satisfied for all sufficiently small $\Lambda > 0$. This reflects the general circumstance that the uniform exponential stability (of the semigroup $G^{\wedge m}$ in our case) is preserved under uniformly small perturbations (controlled by $\Lambda$ in our case).

In general, \eqref{EQ: FreqIneqDelayCompoundGeneralForm} (in particular, \eqref{EQ: DelayCompoundSmithIneqGeneral}) represents a nonlocal condition that may be useful to verify for particular problems. From \eqref{EQ: OperatorAmFormula} it is clear that computation of such conditions requires solving a first-order PDE with boundary conditions containing both partial derivatives and delays. This makes it hard to study the problem purely analytically. Moreover, solutions to the associated resolvent equations belong to the domain $\mathcal{D}(A^{[\wedge m]})$ and are not usual smooth functions. Therefore, the development of appropriate numerical methods for studying such problems is required. 

In the next section, we present a general approach to the problem and, by its means, state an approximation scheme for verification of some frequency inequalities in the case of scalar equations. For $m=2$, we also present a direct approach to examine the resolvent equations, which leads to explicit representations of transfer operators, and show its agreement with the approximation scheme, see Section \ref{SUBSEC: ExampleResolventEquations}.

%% file: VerificationOfFreqIneq.tex
\section{Computation of frequency inequalities}
\label{SEC: ComputationOfFreqIneqDelayComp}
\subsection{Quadratic constraints for self-adjoint derivatives}
\label{SEC: ConstraintsSelfAdjointCompoundDelay}
Before we begin developing approximation schemes, let us consider quadratic constraints that are somewhat more delicate than those in \eqref{EQ: DelayCompStandardChoiceQC}. Such constraints arise in the case when in terms of \eqref{EQ: LipschitzFprimeDelay} we have $\mathbb{M}=\mathbb{U}$ and $F'(\wp)$ is self-adjoint. For example, these conditions are satisfied in the study of equations with scalar nonlinearities and measurements.

First, let us formulate an auxiliary lemma.
\begin{lemma}
	\label{LEM: AuxiliaryLemmaSAQuadraticConst}
	Suppose that $S$ is a bounded self-adjoint operator in a real separable Hilbert space $\mathbb{F}$ with inner product $\langle \cdot, \cdot \rangle_{\mathbb{F}}$ such that for some constants $\Lambda_{1} \leq \Lambda_{2}$ we have
	\begin{equation}
		\label{EQ: SelfAdjointOperatorConstraintLemma}
		\Lambda_{1} \langle f,f \rangle_{\mathbb{F}} \leq \langle f,Sf \rangle_{\mathbb{F}} \leq \Lambda_{2} \langle f,f \rangle_{\mathbb{F}} \qquad \text{for any} \quad f \in \mathbb{F}.
	\end{equation} 
	Then the quadratic form $\mathcal{G}(f,\eta)$ of $f,\eta \in \mathbb{F}$ given by
	\begin{equation}
		\label{EQ: QuadraticFormSelfAdjoint}
		\mathcal{G}(f,\eta) \coloneq - \Lambda_{1} \Lambda_{2} \langle f,f \rangle_{\mathbb{F}} + (\Lambda_{1} + \Lambda_{2})\langle f,\eta \rangle_{\mathbb{F}}  - \langle \eta,\eta \rangle_{\mathbb{F}}
	\end{equation}
	satisfies $\mathcal{G}(f,\eta) \geq 0$ provided that $\eta = Sf$.
\end{lemma}
\begin{proof}
	Using the spectral theorem for bounded self-adjoint operators, we may assume that $\mathbb{F} = L_{2}(\mathcal{X};\mu;\mathbb{R})$, where $(\mathcal{X},\mu)$ is a measure space, $\mu$ is $\sigma$-finite, and $S$ is a multiplication on a $\mu$-essentially bounded function $\varphi$ on $\mathcal{X}$ with $\varphi(x) \in [\Lambda_{1},\Lambda_{2}]$ for $\mu$-almost all $x \in \mathcal{X}$. Substituting $\eta = Sf$ into \eqref{EQ: QuadraticFormSelfAdjoint}, we obtain
	\begin{equation}
		\mathcal{G}(f,Sf) = \int_{\mathcal{X}}\left[ -\varphi^{2}(x) + (\Lambda_{1}+\Lambda_{2})\varphi(x) - \Lambda_{1}\Lambda_{2} \right] |f(x)|^{2} d\mu(x) \geq 0,
	\end{equation}
	since the multiplier in the square brackets is nonnegative $\mu$-almost everywhere.
\end{proof}

Next, in terms of Section \ref{SEC: StabilityOfDelayCompounds}, we assume that $\mathbb{F} \coloneq \mathbb{M} = \mathbb{U}$ and the operator $S \coloneq F'(\wp)$ is self-adjoint for any $\wp \in \mathcal{P}$ and satisfies \eqref{EQ: SelfAdjointOperatorConstraintLemma} for some $\Lambda_{1} \leq 0 \leq \Lambda_{2}$ (independent of $\wp$). Then the same holds for $\mathbb{F} \coloneq \mathbb{U}^{\wedge}_{m} = \mathbb{M}^{\wedge}_{m}$ and the operator $S \coloneq F^{\wedge}_{m}(\wp)$ by similar arguments as used near \eqref{EQ: ClosedFeedbackDelayCompoundRelation}. Thus, for the quadratic form $\mathcal{G}(M,\eta)$ of $M \in \mathbb{M}^{\wedge}_{m}$ and $\eta \in \mathbb{U}^{\wedge}_{m}$,
\begin{equation}
	\mathcal{G}(M,\eta) \coloneq -\Lambda_{1} \Lambda_{2} \langle M, M\rangle_{\mathbb{M}^{\wedge}_{m}} + (\Lambda_{1} + \Lambda_{2}) \langle M, \eta\rangle_{\mathbb{U}^{\wedge}_{m}} - \langle \eta,\eta \rangle_{\mathbb{U}^{\wedge}_{m}},
\end{equation}
the associated quadratic form $\mathcal{F}(\Phi,\eta)\coloneq\mathcal{G}(C^{\wedge}_{m}\Phi,\eta)$ of $\Phi \in \mathbb{E}^{\wedge}_{m}$ and $\eta \in \mathbb{U}^{\wedge}_{m}$ defines a quadratic constraint with respect to \eqref{EQ: InfinitesimaldDescriptionXim} according to Lemma \ref{LEM: AuxiliaryLemmaSAQuadraticConst}.

For the corresponding frequency inequality \eqref{EQ: FreqIneqDelayCompoundGeneralForm}, we have to satisfy for some $\delta>0$ and all $p=-\nu_{0}+i\omega$, where $\omega \in \mathbb{R}$, and $\eta \in (\mathbb{U}^{\wedge}_{m})^{\mathbb{C}}$, the following inequality:
\begin{equation}
	\label{EQ: DelayCompoundSelfAdjointFreqIneqGeneral}
	\Lambda_{1}\Lambda_{2} \left|W(p)\eta \right|^{2}_{(\mathbb{M}^{\wedge}_{m})^{\mathbb{C}}} - (\Lambda_{1} + \Lambda_{2})\operatorname{Re}\langle W(p)\eta, \eta \rangle_{(\mathbb{U}^{\wedge}_{m})^{\mathbb{C}}} + |\eta|^{2}_{(\mathbb{U}^{\wedge}_{m})^{\mathbb{C}}} \geq \delta |\eta|^{2}_{(\mathbb{U}^{\wedge}_{m})^{\mathbb{C}}}.
\end{equation}

In the case $\Lambda_{1} = -\Lambda$ and $\Lambda_{2} = \Lambda$ for some $\Lambda>0$, one may see that \eqref{EQ: DelayCompoundSelfAdjointFreqIneqGeneral} is equivalent to \eqref{EQ: DelayCompoundSmithIneqGeneral}, i.e.,
\begin{equation}
	\label{EQ: DelayCompoundSelfAdjContraintIdenticalBoundsFreqC}
	\sup_{\omega \in \mathbb{R}}\| W(-\nu_{0} + i \omega) \|_{\mathcal{L}((\mathbb{U}^{\wedge}_{m})^{\mathbb{C}};(\mathbb{M}^{\wedge}_{m})^{\mathbb{C}})}  < \Lambda^{-1}.
\end{equation}

Now let us assume that $\Lambda_{1} = 0$ and $\Lambda_{2} = \Lambda>0$ (the case of monotone nonlinearities). Then \eqref{EQ: DelayCompoundSelfAdjointFreqIneqGeneral} is equivalent to the following condition:
\begin{equation}
	\label{EQ: SelfAdjointAdditiveSymmetriztionOptimizationFreqCond}
	\sup_{\omega \in \mathbb{R}}\sup_{\substack{\eta \in (\mathbb{U}^{\wedge}_{m})^{\mathbb{C}},\\ \eta\not=0}} \frac{\left\langle S_{W}(-\nu_{0} + i \omega)\eta, \eta \right\rangle_{(\mathbb{U}^{\wedge}_{m})^{\mathbb{C}}}}{|\eta^{2}|_{(\mathbb{U}^{\wedge}_{m})^{\mathbb{C}}}} < \Lambda^{-1},
\end{equation}
where $S_{W}(p) \coloneq \frac{1}{2}(W(p)+W^{*}(p))$ is the additive symmetrization of $W(p)=-C^{\wedge}_{m}(A^{[\wedge m]} - pI)^{-1}B^{\wedge}_{m}$. Since $S_{W}(p)$ is self-adjoint, the inner supremum in \eqref{EQ: SelfAdjointAdditiveSymmetriztionOptimizationFreqCond} is the upper spectral bound of $S_{W}(p)$.

In the forthcoming sections, we will develop an approximation scheme to verify \eqref{EQ: DelayCompoundSelfAdjContraintIdenticalBoundsFreqC} and \eqref{EQ: SelfAdjointAdditiveSymmetriztionOptimizationFreqCond} for $n = r_{1} = r_{2} =1$ (in terms of \eqref{EQ: DelayRnLinearized}) and report some experimental results (see Remark \ref{REM: FrequencyCompoundGeneralCase} for the general case). Note also that in our experiments we use only the condition \eqref{EQ: DelayCompoundSelfAdjContraintIdenticalBoundsFreqC}, which, as it turned out, provides better results. However, \eqref{EQ: SelfAdjointAdditiveSymmetriztionOptimizationFreqCond} may be useful in some other applications.

\subsection{Approximation scheme: preliminaries}
\label{SEC: ApproximationFreqIneqDelayCompound}
For the computation of \eqref{EQ: DelayCompoundSelfAdjContraintIdenticalBoundsFreqC} and \eqref{EQ: SelfAdjointAdditiveSymmetriztionOptimizationFreqCond}, we have the following standard lemma.
\begin{lemma}
	\label{LEM: RaleighQuotientSelfAdjointOptimizationApproximation}
	Suppose that $\mathbb{F}_{1}$ and $\mathbb{F}_{2}$ are separable complex Hilbert spaces with orthonormal bases $\{ e^{1}_{k} \}_{k \geq 1}$ and $\{ e^{2}_{k} \}_{k \geq 1}$, respectively. Let $S$ be a bounded linear operator from $\mathbb{F}_{1}$ to $\mathbb{F}_{2}$. Consider the orthogonal projectors $P^{1}_{N}$ and $P^{2}_{N}$ onto the linear spans of $e^{1}_{1},\ldots,e^{1}_{N}$ and $e^{2}_{1},\ldots,e^{2}_{N}$, respectively. Then as $N \to \infty$ we have
	\begin{equation}
		\label{EQ: RaleighQuotientOptimizationApproximationConvergence}
		\alpha_{N} \coloneq \|P^{2}_{N} S P^{1}_{N} \|_{\mathcal{L}(\mathbb{F}_{1}, \mathbb{F}_{2})}  \to \alpha \coloneq \| S \|_{ \mathcal{L}(\mathbb{F}_{1}, \mathbb{F}_{2}) },
	\end{equation}
    and $\alpha_{N} \leq \alpha_{N+1}$ for any $N = 1,2,\ldots\,$. Moreover, if $\mathbb{F}_{1}=\mathbb{F}_{2}$, $P^{1}_{N}=P^{2}_{N}=P_{N}$, and $S$ is self-adjoint, then the analogous convergence holds for the largest eigenvalues $\alpha_{N}$ of $P_{N} S P_{N}$ and the upper spectral bound $\alpha$ of $S$.
\end{lemma}
Below we will deal only with the case $\mathbb{F}_{1} = \mathbb{F}_{2} \eqcolon \mathbb{F}$ and $P^{1}_{N} = P^{2}_{N} \eqcolon P_{N}$.

We aim to apply Lemma \ref{LEM: RaleighQuotientSelfAdjointOptimizationApproximation} to the operators $S \coloneq W(p)$ (see \eqref{EQ: DelayCompoundSelfAdjContraintIdenticalBoundsFreqC}) or $S\coloneq S_{W}(p)$ (see \eqref{EQ: SelfAdjointAdditiveSymmetriztionOptimizationFreqCond}) for $p=-\nu_{0} + i\omega$ with a fixed $\nu_{0} \in \mathbb{R}$ and all $\omega \in \mathbb{R}$. Note that in this case we have $\alpha_{N} = \alpha_{N}(-\nu_{0} + i\omega)$ and $\alpha = \alpha(-\nu_{0} + i\omega)$ and, consequently, the convergence in \eqref{EQ: RaleighQuotientOptimizationApproximationConvergence} depends on $\omega$. Using Theorem \ref{TH: ResolventDelayCompoundBound} below and the first resolvent identity, one can show the following. 
\begin{lemma}\cite[Lemma 7.2]{Anikushin2023Comp}
	\label{LEM: ApproxSchemeAlphLipsch}
	In the above context, the functions $\mathbb{R} \ni \omega \mapsto \alpha_{N}(-\nu_{0} + i\omega)$ and $\mathbb{R} \ni \omega \mapsto \alpha(-\nu_{0} + i\omega)$ are globally Lipschitz with a uniform in $N$ Lipschitz constant.
\end{lemma}
This implies that $\alpha_{N}(-\nu_{0} + i\omega)$ converges to $\alpha(-\nu_{0} + i\omega)$ uniformly on compact subsets of $\omega$ as $N \to \infty$. However, to verify frequency inequalities, we have to investigate them for $\omega$ from the entire $\mathbb{R}$. Below we conjecture that it is sufficient to work on a finite segment due to an asymptotically almost periodic (or even constant) behavior of $\alpha_{N}(-\nu_{0}+i\omega)$ as $|\omega| \to \infty$, see Conjecture \ref{CONJ: DelayCompoundOscillationPattern}.

Thus, from the perspective given by Lemma \ref{LEM: RaleighQuotientSelfAdjointOptimizationApproximation}, for numerical verification of frequency inequalities, it is required to compute $-C^{\wedge}_{m}(A^{[\wedge m]} - p I)^{-1}B^{\wedge}_{m} \eta$ for a finite number of $\eta$ belonging to an orthonormal basis in $(\mathbb{U}^{\wedge}_{m})^{\mathbb{C}}$. 

In the general case, we leave open the problem of developing appropriate numerical schemes for direct computations of the resolvent via solving the corresponding first-order PDE with boundary conditions on the $m$-cube $(-\tau,0)^{m}$ according to the description of $A^{[\wedge m]}$ from Theorem \ref{TH: AdditiveCompoundDelayDescription}. We refer to Section \ref{SUBSEC: ExampleResolventEquations} for an illustration by means of the simplest example, where an explicit representation of solutions is possible.

Below, we develop an alternative approach that is concerned with the computation of trajectories of the semigroup $G$ only. It relies on the following proposition, which is the well-known representation of the resolvent via the Laplace transform of the semigroup. For convenience, hereinafter we use the same notations to denote the complexifications of operators defined in Section \ref{SEC: StabilityOfDelayCompounds}, but we emphasize complexifications of the spaces.
\begin{proposition}\cite[Theorem 1.10, Chapter II]{EngelNagel2000}
	\label{PROP: ResolventDelayCompoundDescription}
	Let $p \in \mathbb{C}$ be such that $\operatorname{Re}p > \omega(G^{\wedge m})$, where $\omega(G^{\wedge m})$ is the growth bound of $G^{\wedge m}$. Then for any $\Phi \in (\mathcal{L}^{\wedge}_{m})^{\mathbb{C}}$ we have
	\begin{equation}
	\label{EQ: ResolventEquationForDelayCompoundGeneral}
			-(A^{[\wedge m]}-pI)^{-1} \Phi = \int_{0}^{\infty} e^{-p t}G^{\wedge m}(t)\Phi dt.
	\end{equation}
    In particular, for $\psi_{1},\ldots,\psi_{m} \in \mathbb{H}^{\mathbb{C}}$, and $\Phi = \psi_{1} \wedge \cdots \wedge \psi_{m}$, we have
    \begin{equation}
    \label{EQ: ResolventEquationForDelayCompoundComputation}
    	-(A^{[\wedge m]}-pI)^{-1} (\psi_{1}\wedge \cdots \wedge \psi_{m}) = \int_{0}^{\infty} e^{-pt}G(t)\psi_{1} \wedge \cdots \wedge G(t)\psi_{m}dt.
    \end{equation}
\end{proposition}

Now our aim is to provide appropriate uniform estimates for the tail of the integral from \eqref{EQ: ResolventEquationForDelayCompoundGeneral}. For this, the following fundamental property of additive compound operators arising from delay equations is essential. Recall here the intermediate space $\mathbb{E}^{\wedge}_{m}$ defined below \eqref{EQ: SpaceEmDefinitionDelayCompoundGeneral}.
\begin{theorem}\cite[Theorem 4.4]{Anikushin2023Comp}
	\label{TH: ResolventDelayCompoundBound}
	For regular (i.e., non-spectral) points $p \in \mathbb{C}$ of $A^{[\wedge m]}$ we have the following estimate:
	\begin{equation}
		\label{EQ: ResolventEstimateEmDelayCompound}
		\| (A^{[\wedge m]} - p I)^{-1} \|_{\mathcal{L}((\mathcal{L}^{\wedge}_{m})^{\mathbb{C}};(\mathbb{E}^{\wedge}_{m})^{\mathbb{C}}) } \leq C_{1}(p) \cdot \| (A^{[\wedge m]} - p I)^{-1} \|_{\mathcal{L}((\mathcal{L}^{\wedge}_{m})^{\mathbb{C}}) } + C_{2}(p),
	\end{equation}
	where the constants $C_{1}(p)$ and $C_{2}(p)$ in fact depend only on $\max\{1, e^{-\tau \operatorname{Re}p}\}$ in a monotonically increasing manner.
\end{theorem}
\begin{remark}
	\label{REM: UniformBoundResolventDelayCompound}
	Note that the norms of the resolvent $(A^{[\wedge m]}-pI)^{-1}$ in $\mathcal{L}((\mathcal{L}^{\wedge}_{m})^{\mathbb{C}})$ are uniformly bounded on any vertical line $-\nu_{0} + i\mathbb{R}$ that avoids the spectrum of $A^{[\wedge m]}$. For the case of our interest $-\nu_{0} > \omega(G^{\wedge m}) = s(A^{[\wedge m]})$ (the semigroup $G^{\wedge m}$ is eventually compact), this follows directly from \eqref{EQ: ResolventEquationForDelayCompoundGeneral}. For the general case, this can be shown via spectral decompositions, see, for example, \cite[Theorem 4.2]{Anikushin2020FreqDelay}. Using this and \eqref{EQ: ResolventEstimateEmDelayCompound}, we immediately get a uniform bound in $\mathcal{L}((\mathcal{L}^{\wedge}_{m})^{\mathbb{C}}; (\mathbb{E}^{\wedge}_{m})^{\mathbb{C}})$. It is also uniform in $\nu_{0}$ from compact subsets. This, along with the first resolvent identity, ensures that frequency inequalities are satisfied on an open set of $\nu_{0}$.
\end{remark}

From Theorem \ref{TH: ResolventDelayCompoundBound}, we derive the following main result that justifies the forthcoming approximation scheme.
\begin{theorem}
	\label{TH: DelayCompoundComputationTailEstimate}
	Let $\nu_{0} \in \mathbb{R}$ be such that $-\nu_{0} > \omega(G^{\wedge m})$. Then for any $p = -\nu_{0} + i \omega$, where $\omega \in \mathbb{R}$, $\Phi \in (\mathcal{L}^{\wedge}_{m})^{\mathbb{C}}$, and $T \geq 0$, we have
	\begin{equation}
		-(A^{[\wedge m]}-pI)^{-1} \Phi = \int_{0}^{T}e^{-p t}G^{\wedge m}(t)\Phi dt + R_{T}(p;\Phi),
	\end{equation}
    where $R_{T}(p;\Phi) \in \mathcal{D}(A^{[\wedge m]}) \subset (\mathbb{E}^{\wedge}_{m})^{\mathbb{C}}$, and for any $\varkappa \in (0, -\nu_{0} - \omega(G^{\wedge m}))$ there exists $M_{\varkappa} > 0$ such that for any $\Phi \in (\mathcal{L}^{\wedge}_{m})^{\mathbb{C}}$ and $T \geq 0$, the estimate
    \begin{equation}
    	\label{EQ: DelayCompoundIntegralDecayEstimate}
    	\|R_{T}(p;\Phi)\|_{(\mathbb{E}^{\wedge}_{m})^{\mathbb{C}}} \leq M_{\varkappa} e^{-\varkappa T} \cdot |\Phi|_{(\mathcal{L}^{\wedge}_{m})^{\mathbb{C}}},
    \end{equation}
    holds uniformly in $p = -\nu_{0} + i \omega$ with $\omega \in \mathbb{R}$.
\end{theorem}
\begin{proof}
	Using \eqref{EQ: ResolventEquationForDelayCompoundGeneral}, we obtain that $R_{T}=R_{T}(p;\Phi)$ satisfies the following equation:
	\begin{equation}
		\label{EQ: DelayCompoundTailDecayEstimate}
		\begin{split}
			R_{T} = \int_{T}^{\infty}e^{-p t} G^{\wedge m}(t)\Phi dt = e^{-p T} (A^{[\wedge m]} - pI)^{-1}G^{\wedge m}(T)\Phi.
		\end{split}
	\end{equation}
    Suppose that $\varkappa \in (0, -\nu_{0} - \omega(G^{\wedge m}))$. Then there exists $\widetilde{M}_{\varkappa}>0$ such that the estimate
    \begin{equation}
    	| G^{\wedge m}(T) \Phi |_{(\mathcal{L}^{\wedge}_{m})^{\mathbb{C}}} \leq \widetilde{M}_{\varkappa} e^{-(\nu_{0} + \varkappa)T} | \Phi |_{(\mathcal{L}^{\wedge}_{m})^{\mathbb{C}}}
    \end{equation}
    is satisfied for all $T \geq 0$ and $\Phi \in (\mathcal{L}^{\wedge}_{m})^{\mathbb{C}}$. Combining this with \eqref{EQ: DelayCompoundTailDecayEstimate}, \eqref{EQ: ResolventEstimateEmDelayCompound}, and Remark \ref{REM: UniformBoundResolventDelayCompound}, we get \eqref{EQ: DelayCompoundIntegralDecayEstimate}.
\end{proof}

Now we are going to exploit \eqref{EQ: ResolventEquationForDelayCompoundComputation}, sticking to the case $r_{1}=r_{2}=1$, i.e., $\mathbb{M}=\mathbb{U}=\mathbb{R}$, and $n = 1$. In fact, $r_{2}=1$ is considered only for simplicity, and $r_{2} = \dim\mathbb{M}$ can be arbitrary, which allows a possibility of several measurements involved in the operator $C$. Although the restrictions $r_{1}=\dim\mathbb{U}=1$ and $n=1$ are essential for what follows, theoretically there is no problem to develop the approach in the general case, but at the cost of much heavier computations, see Remark \ref{REM: FrequencyCompoundGeneralCase}.

We start with the following proposition describing the antisymmetric relations \eqref{EQ: ControlSpaceRelationsAntisymmetricCompoundDelay} in the control space $\mathbb{U}^{\wedge}_{m}$. Clearly, the same description holds for the measurement space $\mathbb{M}^{\wedge}_{m}$. Although we need to describe it only for $k=m-1$ (since all $k \leq m-2$ are improper), we give such a description for any $k$ to illustrate the exclusion of improper $k$ in the definitions of $\mathbb{U}^{\wedge}_{m}$ and $\mathbb{M}^{\wedge}_{m}$, see Remark \ref{REM: ZeroingControlComponents}.
\begin{proposition}
	\label{PROP: ControlSpaceCompoundDelayDescriptionScalarCase}
	Suppose that $n = 1$. Then the relations \eqref{EQ: ControlSpaceRelationsAntisymmetricCompoundDelay} take the following form:
	\begin{alignat}{2}
		\label{EQ: ControlAntisymmetricRelationsScalar}
			&\eta^{j}_{j_{1} \ldots j_{k}} = 0 \quad &&\text{for any} \quad k \in \{ 0, \ldots, m-3 \},\notag\\
			&\eta^{j}_{\hat{j}} \quad &&\text{is antisymmetric for any} \quad j \in \{1, \ldots, m\},\notag\\
			&\eta^{j}_{\hat{j}} = (-1)^{j-i} \eta^{i}_{\hat{i}} \quad &&\text{for all} \quad i,j \in \{ 1,\ldots, m \},\\
			&\eta^{j}_{\widehat{ij}} = - \eta^{i}_{\widehat{ij}} \quad &&\text{for all distinct} \quad i,j \in \{1,\ldots, m\},\notag\\
			&\eta^{j}_{\widehat{ij}} \quad &&\text{is antisymmetric for all distinct} \quad i,j \in \{1,\ldots,m\},\notag\\
			&\eta^{j}_{\widehat{ij}} = (-1)^{(k-i)+1}\eta^{j}_{\widehat{kj}} \quad &&\text{for all distinct} \quad i,j,k \in \{1,\ldots, m\}.\notag
	\end{alignat}
\end{proposition}
\begin{proof}
	For $n = 1$ we have that $\mathbb{U}_{j} = \mathbb{U}$ (see the beginning of Subsection \ref{SUBSEC: InducedControlOperators}), and hence $T_{\sigma}$ in \eqref{EQ: ControlSpaceRelationsAntisymmetricCompoundDelay} is the identity operator. Let us start from $k \in \{0, \ldots, m-3\}$ (of course, if $m \geq 3$). Consider \eqref{EQ: ControlSpaceRelationsAntisymmetricCompoundDelay} with $j_{l} = l$ for $l \in \{1,\ldots, k\}$, $j=k+1$, and $\sigma \in \mathbb{S}_{m}$ such that
	\begin{equation}
		\sigma = \begin{pmatrix}
			1 & \ldots & k + 1 & \ldots\\
			1 & \ldots & k + 1 & \ldots
		\end{pmatrix},
	\end{equation}
	where the tail, i.e., $\sigma(l)$ for $l >k+1$, is arbitrary. This leads to
	\begin{equation}
		\eta^{k+1}_{1\ldots k} = (-1)^{\sigma} \eta^{k+1}_{1\ldots k}.
	\end{equation}
	Since there are at least $2$ undetermined positions in $\sigma$, by a proper choice of $\sigma$ we get $\eta^{k+1}_{1\ldots k} = - \eta^{k+1}_{1\ldots k} = 0$. This shows the first relation from \eqref{EQ: ControlAntisymmetricRelationsScalar}, since any $\eta^{j}_{j_{1}\ldots j_{k}}$ can be determined from $\eta^{k+1}_{1\ldots k}$ using \eqref{EQ: ControlSpaceRelationsAntisymmetricCompoundDelay}.
	
	Next, we consider the case $k=m-1$. Let $j \in \{1, \ldots, m\}$. Then \eqref{EQ: ControlSpaceRelationsAntisymmetricCompoundDelay} takes the form
	\begin{equation}
		\label{EQ: ControlFunctionsRelationsScalarMminusOne}
		\eta^{j}_{\hat{j}} = (-1)^{\sigma} \Theta^{(m-1)}_{\bar{\sigma}} \eta^{\sigma(j)}_{\widehat{\sigma(j)}}.
	\end{equation}
	Taking this condition over all $\sigma$ such that $\sigma(j) = j$ gives the antisymmetricity of $\eta^{j}_{\hat{j}}$, i.e., the second relation from \eqref{EQ: ControlAntisymmetricRelationsScalar}. Furthermore, taking the cycle $\sigma = (i j)$ for some $i,j \in \{ 1,\ldots, m \}$ and utilizing the antisymmetricity, we obtain the third identity from \eqref{EQ: ControlAntisymmetricRelationsScalar}. It is not hard to show that these two identities are sufficient to derive \eqref{EQ: ControlFunctionsRelationsScalarMminusOne} for general permutations.
	
	For $k=m-2$, we consider \eqref{EQ: ControlSpaceRelationsAntisymmetricCompoundDelay} with $\sigma = (ij)$ to show the fourth relation from \eqref{EQ: ControlAntisymmetricRelationsScalar}; any $\sigma$ such that $\sigma(i)=i$ and $\sigma(j)=j$ to obtain the fifth relation from \eqref{EQ: ControlAntisymmetricRelationsScalar}; and $\sigma = (i k)$ for the sixth relation from \eqref{EQ: ControlAntisymmetricRelationsScalar}. We leave details to the reader.
\end{proof}
\begin{remark}
	\label{REM: ZeroingControlComponents}
	By the fourth series of relations from \eqref{EQ: ControlAntisymmetricRelationsScalar}, we may illustrate the discussion given below \eqref{EQ: AntisymmetricControlSpaceDefinition} on forcing $\eta^{j}_{\widehat{ij}}$ to zero in the definition \eqref{EQ: AntisymmetricControlSpaceDefinition} of $\mathbb{U}^{\wedge}_{m}$. Although single components $\eta^{j}_{\widehat{ij}}$ satisfying the relations need not be zero, the corresponding inner sum in the definition \eqref{EQ: OperatorBDelayCompoundDefinition} of $B^{\otimes}_{m}$ vanishes since
	\begin{equation}
		 B^{\widehat{ij}}_{i} \eta^{i}_{\widehat{ij}} + B^{\widehat{ij}}_{j} \eta^{j}_{\widehat{ij}} = \widetilde{B} \left(\eta^{i}_{\widehat{ij}} + \eta^{j}_{\widehat{ij}}\right) = 0.
	\end{equation}
	Thus, these components do not cause any effect on the control system.
\end{remark}

In virtue of Proposition \ref{PROP: ControlSpaceCompoundDelayDescriptionScalarCase}, it is convenient to identify any element $\eta \in (\mathbb{U}^{\wedge}_{m})^{\mathbb{C}}$ with an $m$-tuple $\eta = (\eta^{j}_{\hat{j}})_{j=1}^{m}$, where each $\eta^{j}_{\hat{j}}$ is an antisymmetric function from $L_{2}((-\tau,0)^{m-1};\mathbb{C})$ and $\eta^{j}_{\hat{j}} = (-1)^{1-j} \eta^{1}_{\hat{1}}$ for any $j \in \{1,\ldots,m\}$. Clearly, this establishes an isometric isomorphism between the subspace of such tuples in the orthogonal sum $\bigoplus_{j=1}^{m}L_{2}((-\tau,0)^{m-1};\mathbb{C})$ and $(\mathbb{U}^{\wedge}_{m})^{\mathbb{C}}$. Since any $\eta^{j}_{\hat{j}}$ is uniquely determined from $\eta^{1}_{\hat{1}}$, it is sufficient to construct an orthonormal basis in the subspace of antisymmetric functions from $L_{2}((-\tau,0)^{m-1};\mathbb{C})$.

Consider a family of functions $\phi_{k}(\theta)$, where $k \in \mathbb{Z}$, forming an orthonormal basis in $L_{2}(-\tau,0;\mathbb{C})$. Then the functions $\phi_{k_{1}}(\theta_{1}) \cdots \phi_{k_{m-1}}(\theta_{m-1})$, taken over all $k_{1},\ldots,k_{m-1} \in \mathbb{Z}$, form an orthonormal basis in the space $L_{2}((-\tau,0)^{m-1};\mathbb{C})$, and the functions
\begin{equation}
	\begin{split}
		U_{k_{1}\ldots k_{m-1}}(\theta_{1},\ldots,\theta_{m-1}) \coloneq \\ \frac{1}{\sqrt{(m-1)!}} \sum_{\widetilde{\sigma} \in \mathbb{S}_{m-1}} (-1)^{\widetilde{\sigma}} \phi_{k_{1}}(\theta_{\widetilde{\sigma}(1)}) \cdots \phi_{k_{m-1}}(\theta_{\widetilde{\sigma}(m-1)}),
	\end{split}
\end{equation}
taken over all integers $k_{1} < k_{2} < \cdots < k_{m-1}$, form an orthonormal basis in the subspace of antisymmetric functions from $L_{2}((-\tau,0)^{m-1};\mathbb{C})$. Consequently, the $m$-tuples $U^{\wedge}_{k_{1}\ldots k_{m-1}} = (U^{j}_{k_{1}\ldots k_{m-1}})_{j=1}^{m}$, where
\begin{equation}
	\label{EQ: ExampleCompoundDelayCalcBasisControlSpace}
	U^{j}_{k_{1}\ldots k_{m-1}} \coloneq (-1)^{1-j} m^{-1/2} U_{k_{1}\ldots k_{m-1}},
\end{equation}
taken over all integers $k_{1} < k_{2} < \cdots < k_{m-1}$, form an orthonormal basis in the control space $(\mathbb{U}^{\wedge}_{m})^{\mathbb{C}}$.

For $k \in \mathbb{Z}$, let $\psi_{k} \in \mathbb{H}$ be such that $R^{(1)}_{1}\psi_{k} = \phi_{k}$ and $R^{(1)}_{0}\psi_{k} \in \mathbb{R}$ is arbitrary\footnote{In virtue of \eqref{EQ: PropResolventCompoundBActionCompBdescr}, the values $R^{(1)}_{0}\psi_{l}$ affect only the restrictions $R_{j_{1}\ldots j_{k}}$ with $k \leq m-2$, which all vanish due to Corollary \ref{COR: AntisymmetricRelationsScalar}.}. If $\phi_{k}$ is continuous, it is convenient to set $R^{(1)}_{0}\psi_{k} \coloneq \phi_{k}(0)$.
\begin{proposition}
	\label{PROP: ResolventCompoundBActionBasisDescription}
	In the context of \eqref{EQ: DelayRnLinearized}, suppose that $n=r_{1}=1$. Then for $U^{\wedge}_{k_{1}\ldots k_{m-1}} = (U^{j}_{k_{1}\ldots k_{m-1}})_{j=1}^{m}$ defined by \eqref{EQ: ExampleCompoundDelayCalcBasisControlSpace}, we have
	\begin{equation}
		\label{EQ: PropResolventCompoundBActionCompBdescr}
		B^{\wedge}_{m} U^{\wedge}_{k_{1}\ldots k_{m-1}} = \psi_{k_{1}} \wedge \cdots \wedge \psi_{k_{m-1}} \wedge \psi_{\infty},
	\end{equation}
	where $\psi_{\infty} \in \mathbb{H}$ is such that $R^{(1)}_{1}\psi_{\infty} = 0$ and 
	\begin{equation}
		\label{EQ: PropResolventCompoundBActionCompDeltaValue}
		R^{(1)}_{0}\psi_{\infty} = (-1)^{m + 1} \sqrt{m!} \cdot \widetilde{B}.
	\end{equation}
\end{proposition}
\begin{proof}
	Let $b_{j} \colon \{ 1,\ldots, m-1 \} \to \{ 1,\ldots, \hat{j}, \ldots, m \}$ be the bijection such that $b_{j}(l) = l$ for $l < j$ and $b_{j}(l) = l+1$ for $l \geq j$. Then from \eqref{EQ: DelayCompoundIsomorphismAntisymmetricDecomp} for $\mu^{\otimes m}$-almost all $(\theta_{1},\ldots,\theta_{m}) \in [-\tau,0]^{m}$, we have
	\begin{equation}
		\label{EQ: PropResolventCompoundBActionCompScalarIdentity}
		\begin{split}
			(\psi_{k_{1}} \wedge \cdots \wedge \psi_{k_{m-1}} \wedge \psi_{\infty})(\theta_{1},\ldots,\theta_{m}) =\\= \frac{1}{m!} \sum_{ \sigma \in \mathbb{S}_{m}} (-1)^{\sigma} \psi_{k_{1}}(\theta_{ \sigma(1) } ) \cdots \psi_{k_{m-1}}(\theta_{\sigma(m-1)}) \cdot \psi_{\infty}(\theta_{\sigma(m)}) =\\=
			\frac{1}{m}\sum_{j=1}^{m} \psi_{\infty}(\theta_{j}) \frac{1}{(m-1)!}\sum_{\widetilde{\sigma} \in \mathbb{S}_{m-1}} (-1)^{\hat{\sigma}_{j}} \psi_{k_{1}}(\theta_{ \widetilde{\sigma}_{j}(1) } ) \cdots \psi_{k_{m-1}}(\theta_{\widetilde{\sigma}_{j}(m-1)}),
		\end{split}
	\end{equation}
	where $\widetilde{\sigma}_{j} \coloneq b_{j} \circ \widetilde{\sigma}$, and $\hat{\sigma}_{j} \in \mathbb{S}_{m}$ is given by
	\begin{equation}
		\hat{\sigma}_{j} = \begin{pmatrix}
			1 & \ldots & m-1 & m\\
			\widetilde{\sigma}_{j}(1) & \dots & \widetilde{\sigma}_{j}(m-1) & j
		\end{pmatrix}.
	\end{equation}
	It is easy to see that any inversion $\hat{\sigma}_{j}(l_{1}) > \hat{\sigma}_{j}(l_{2})$ for some $1 \leq l_{1} < l_{2} \leq m-1$ is equivalent to $\widetilde{\sigma}(l_{1}) > \widetilde{\sigma}(l_{2})$, and there are exactly $m-j$ inversions in $\hat{\sigma}_{j}$ for $l_{2}=m$. Thus, $(-1)^{\hat{\sigma}_{j}} = (-1)^{\widetilde{\sigma}} \cdot (-1)^{m-j}$. Applying the restriction operator $R^{(m)}_{\hat{j}}$ to \eqref{EQ: PropResolventCompoundBActionCompScalarIdentity} (only the $j$th summand survives) and using \eqref{EQ: PropResolventCompoundBActionCompDeltaValue} and \eqref{EQ: ExampleCompoundDelayCalcBasisControlSpace}, we obtain
	\begin{equation}
		\begin{split}
			R^{(m)}_{\hat{j}}(\psi_{k_{1}} \wedge \cdots \wedge \psi_{k_{m-1}} \wedge \psi_{\infty}) = R^{(1)}_{0}\psi_{\infty} \cdot \frac{(-1)^{m+1}}{\sqrt{m!}} \cdot U^{j}_{k_{1} \ldots k_{m-1}} = \widetilde{B} U^{j}_{k_{1}\ldots k_{m-1}}.
		\end{split}
	\end{equation}
	This shows \eqref{EQ: PropResolventCompoundBActionCompBdescr} according to the definition of $B^{\wedge}_{m}$ as the restriction of $B^{\otimes}_{m}$ from \eqref{EQ: OperatorBDelayCompoundDefinition} to $(\mathbb{U}^{\wedge}_{m})^{\mathbb{C}}$.
\end{proof}
\begin{remark}
	\label{REM: FrequencyCompoundGeneralCase}
	For general $r_{1}$ and $n$, there may not exist such a simple expression as in \eqref{EQ: PropResolventCompoundBActionCompBdescr} for images under $B^{\wedge}_{m}$ of an orthonormal basis in $\left(\mathbb{U}^{\wedge}_{m}\right)^{\mathbb{C}}$. Anyway, one may take an orthonormal basis in $\mathcal{L}^{\wedge}_{m}$ consisting of decomposable antisymmetric functions and express the images in a Fourier series. Then \eqref{EQ: ResolventEquationForDelayCompoundComputation} can be applied elementwise, and for practical computations we need to truncate the series, which introduces another parameter to the approximation scheme.
\end{remark}

Combining Propositions \ref{PROP: ResolventDelayCompoundDescription} and \ref{PROP: ResolventCompoundBActionBasisDescription}, we obtain the following.
\begin{corollary}
	\label{COR: ScalarComputationFormulaWithTailEstimate}
	In the context of Proposition \ref{PROP: ResolventCompoundBActionBasisDescription}, suppose that $p \in \mathbb{C}$ is such that $\operatorname{Re}p > \omega(G^{\wedge m})$. Then for all integers $k_{1} < k_{2} < \cdots < k_{m-1}$, we have
	\begin{equation}
		\label{EQ: DelayCompoundScalarComputationIntegralWithBCorollary}
		\begin{split}
			-(A^{[\wedge m]}-pI)^{-1}B^{\wedge}_{m} U^{\wedge}_{k_{1}\ldots k_{m-1}} =\\= \int_{0}^{\infty} e^{-pt}G(t)\psi_{k_{1}} \wedge \cdots \wedge G(t)\psi_{k_{m-1}} \wedge G(t)\psi_{\infty}dt =\\= \int_{0}^{T} e^{-pt}G(t)\psi_{k_{1}} \wedge \cdots \wedge G(t)\psi_{k_{m-1}} \wedge G(t)\psi_{\infty}dt + R_{T},
		\end{split}
	\end{equation}
    where $R_{T}=R_{T}(p;B^{\wedge}_{m} U^{\wedge}_{k_{1}\ldots k_{m-1}})$ admits the exponential decay estimate as in \eqref{EQ: DelayCompoundIntegralDecayEstimate}. In particular, the decay is uniform in $k_{1},\ldots,k_{m-1} \in \mathbb{Z}$ and $p$ with a fixed real part.
\end{corollary}

Thus, \eqref{EQ: DelayCompoundScalarComputationIntegralWithBCorollary} expresses the boundary action (i.e., via $B^{\wedge}_{m}$) of the resolvent of $A^{[\wedge m]}$ on the basis vector $U^{\wedge}_{k_{1} \ldots k_{m-1}}$ through the integral over $[0,T]$ involving solutions $G(t) \psi_{k}$ and $G(t) \psi_{\infty}$ of the linear system corresponding to $A$ plus a term $R_{T}$ that admits uniform exponential decay as $T \to \infty$. Here $G(t)\psi_{\infty}$ is the so-called fundamental solution up to the multiplier $R^{(1)}_{0}\psi_{\infty}$.

For the computation of the integral from \eqref{EQ: DelayCompoundScalarComputationIntegralWithBCorollary}, we have the following.
\begin{proposition}
	\label{PROP: IntegralHeadComputation}
	In the context of \eqref{EQ: DelayCompoundScalarComputationIntegralWithBCorollary}, suppose that $\phi_{k} \in L_{2}(-\tau,0;\mathbb{C})$ is taken to be continuous for all $k \in \mathbb{Z}$ and set
	\begin{equation}
		\label{EQ: ApproxSchemeIntegralToCompute1}
		\mathcal{I} \coloneq \int_{0}^{T} e^{-pt}G(t)\psi_{k_{1}} \wedge \cdots \wedge G(t)\psi_{k_{m-1}} \wedge G(t)\psi_{\infty}dt \in \mathcal{D}(A^{[\wedge m]}).
	\end{equation}
	Then $R^{(m)}_{1\ldots m}\mathcal{I}$ belongs to $C([-\tau,0]^{m};\mathbb{C})$ and can be expressed as follows:
	\begin{equation}
		\label{EQ: ApproxSchemeIntegralToCompute2}
		(R^{(m)}_{1\ldots m}\mathcal{I})(\theta_{1},\ldots,\theta_{m}) = \int_{0}^{T} e^{-pt}\bigwedge_{k_{1}\ldots k_{m-1}}(t)(\theta_{1},\ldots,\theta_{m}) dt
	\end{equation}
	for all $(\theta_{1},\ldots,\theta_{m}) \in [-\tau,0]^{m}$, where for $t \in [0,T]$ we set\footnote{Recall that for all $\mathbb{C}$-valued functions $f_{1},\ldots,f_{m}$ on $[-\tau,0]$, we set
	\begin{equation}
		\label{EQ: ScalarFunctionsWedgeProductDef}
		(f_{1} \wedge \cdots \wedge f_{m})(\theta_{1},\ldots,\theta_{m})\coloneq \frac{1}{m!} \sum_{\sigma \in \mathbb{S}_{m}}(-1)^{\sigma} f_{1}(\theta_{\sigma(1)}) \cdots f_{m}(\theta_{\sigma(m)})
	\end{equation}
    for all $(\theta_{1},\ldots,\theta_{m}) \in [-\tau,0]^{m}$.
}
	\begin{equation}
		\label{EQ: ApproxSchemeIntegralToCompute3}
		\bigwedge_{k_{1}\ldots k_{m-1}}(t) \coloneq R^{(1)}_{1}G(t)\psi_{k_{1}} \wedge \cdots \wedge R^{(1)}_{1}G(t)\psi_{k_{m-1}} \wedge R^{(1)}_{1}G(t)\psi_{\infty}.
	\end{equation}
	Furthermore, $\mathcal{I}$ and $R^{(m)}_{1\ldots m}\mathcal{I}$ coincide $\mu^{\otimes m}$-almost everywhere on $[-\tau,0]^{m}$.
\end{proposition}
\begin{proof}
	Since $R^{(m)}_{1\ldots m}$ is a bounded operator in $\mathcal{L}^{\wedge}_{m}$, from \eqref{EQ: ApproxSchemeIntegralToCompute1} and \eqref{EQ: DelayCompoundIsomorphismAntisymmetricDecomp} we have
	\begin{equation}
		\label{EQ: FreqCompPropFiniteTimeIntegralComputation}
		R^{(m)}_{1\ldots m}\mathcal{I} = \int_{0}^{T} e^{-pt}\bigwedge_{k_{1}\ldots k_{m-1}}(t) dt.
	\end{equation}
	From this, the validity of \eqref{EQ: ApproxSchemeIntegralToCompute2} for almost all $(\theta_{1},\ldots,\theta_{m}) \in [-\tau,0]^{m}$ is a well-known measure-theoretic fact. Note that for any $k \in \mathbb{Z}$, the function
	\begin{equation}
		x_{k}(t) \coloneq \begin{cases}
			R^{(1)}_{0}G(t)\psi_{k} &\text{if} \quad t \geq 0,\\
			\phi_{k}(t) &\text{if} \quad  t \in [-\tau,0],
		\end{cases} 
	\end{equation} 
	is continuous in $t \geq -\tau$ since it is the boundary part of the solution with continuous initial data, and $(R^{(1)}_{1}G(t)\psi_{k})(\theta) = x_{k}(t+\theta)$ for all $\theta \in [-\tau,0]$ and $t \geq 0$. Moreover, $x_{\infty}(t)$ (defined by the above formula for $k=\infty$) is continuous on $[-\tau,0)$ and in $t \geq 0$. From this, \eqref{EQ: ApproxSchemeIntegralToCompute3}, and \eqref{EQ: ScalarFunctionsWedgeProductDef}, it is clear that the integral in \eqref{EQ: ApproxSchemeIntegralToCompute2} can be represented as a finite sum of integrals depending continuously on $(\theta_{1},\ldots,\theta_{m}) \in [-\tau,0]^{m}$. Thus, the entire integral (and, consequently, $R^{(m)}_{1\ldots m}\mathcal{I}$) belongs to $C([-\tau,0]^{m};\mathbb{C})$.
	
	To show that $\mathcal{I}$ and $R^{(m)}_{1\ldots m}\mathcal{I}$ coincide $\mu^{\otimes m}$-almost everywhere on $[-\tau,0]^{m}$, we use the fact that $\mathcal{I} \in \mathcal{D}(A^{[\wedge m]})$. By \cite[Theorem 4.2]{Anikushin2023Comp}, for all $k \in \{1,\ldots,m\}$ and $1 \leq j_{1} < \cdots < j_{k} \leq m$, the restriction $R^{(k)}_{j_{1}\ldots j_{k}}\mathcal{I}$ belongs to $\mathcal{W}^{2}_{D}((-\tau,0)^{k};\mathbb{C})$ and has traces on the $k$-faces $\mathcal{B}^{(k)}_{\hat{l}}$ for $l \in \{1,\ldots,k\}$ which agree in the $L_{2}$-sense with the restrictions of order $k-1$. By \cite[Theorem A.2]{Anikushin2023Comp}, taking the trace of a continuous function is equivalent to taking its usual restriction. Thus, the restriction $R^{(m)}_{\hat{j}}\mathcal{I}$ belongs to $C([-\tau,0]^{m-1};\mathbb{C})$ since it agrees with the usual restriction of $R^{(m)}_{1\ldots m}\mathcal{I}$ to $\mathcal{B}^{(m)}_{\hat{j}}$ for any $j \in \{1,\ldots,m\}$. One may repeat this argument starting from $R^{(m)}_{\hat{j}}\mathcal{I}$ and pass to the restrictions of order $m-2$ and so on. Note that they actually vanish in our case due to Corollary \ref{COR: AntisymmetricRelationsScalar}.
\end{proof}

\subsection{Approximation scheme: statement}
\label{SEC: ApproxSchemeStatement}
Now we are ready to describe an approximation scheme for verification of frequency inequalities from \eqref{EQ: DelayCompoundSelfAdjContraintIdenticalBoundsFreqC} and \eqref{EQ: SelfAdjointAdditiveSymmetriztionOptimizationFreqCond} in the case of scalar equations. For simplicity, we suppose\footnote{This can be relaxed to allow the possibility of several measurements, i.e., general $r_{2}$, but then we also need to construct an orthonormal basis for $\left(\mathbb{M}^{\wedge}_{m}\right)^{\mathbb{C}}$.} that the measurement operator $C$ for some $\tau_{0} \in [0,\tau]$ is given by $C\phi = \phi(-\tau_{0})$ for any $\phi \in C([-\tau,0];\mathbb{R})$. In the forthcoming applications, we encounter the cases $\tau_{0} = 0$ and $\tau_{0} = \tau$.

For convenience, we rewrite \eqref{EQ: DelayRnLinearized} in the case $n=r_{1}=r_{2} = 1$ and $C$ as above. Thus, we have
\begin{equation}
	\label{EQ: ScalarDelayEquationExample}
	\dot{x}(t) = \widetilde{A}x_{t} + \widetilde{B}F'(\pi^{t}(\wp))x(t-\tau_{0}),
\end{equation}
where the operators $\widetilde{B}$ and $F'(\wp)$ can be identified with real numbers. Let $\Lambda > 0$. Then the frequency condition \eqref{EQ: DelayCompoundSelfAdjContraintIdenticalBoundsFreqC} is associated with the case
\begin{equation}
	\label{EQ: ScalarDelaySymmetricDerivative}
	|F'(\wp)| \leq \Lambda \qquad \text{for any} \quad \wp \in \mathcal{P},
\end{equation}
and the frequency condition \eqref{EQ: SelfAdjointAdditiveSymmetriztionOptimizationFreqCond} is associated with the case
\begin{equation}
	\label{EQ: ScalarDelayMonotoneDerivative}
	0 \leq F'(\wp) \leq \Lambda \qquad \text{for any} \quad \wp \in \mathcal{P}.
\end{equation}

By the Riesz representation theorem, there exists a function of bounded variation $\alpha(\cdot)$ on $[-\tau,0]$ such that
\begin{equation}
	\widetilde{A}\phi = \int_{-\tau}^{0}\phi(\theta)d\alpha(\theta) \qquad \text{for any} \quad \phi \in C([-\tau,0];\mathbb{R}).
\end{equation}
It is well known, see \cite{Hale1977, Anikushin2020FreqDelay}, that the spectrum of $A$, which is associated with $\widetilde{A}$ via \eqref{EQ: OperatorAScalarDelayEquations}, is given by the roots $p \in \mathbb{C}$ of
\begin{equation}
	\label{EQ: SpectrumScalarDelayOperator}
	\int_{-\tau}^{0}e^{p\theta}d\alpha(\theta) - p = 0.
\end{equation}
For example, if $\widetilde{A}\phi = a \phi(0) + b\phi(-\tau)$ for some $a,b \in \mathbb{R}$, then \eqref{EQ: SpectrumScalarDelayOperator} takes the form $a + be^{-\tau p} - p = 0$. Let $\lambda_{1}(A), \lambda_{2}(A), \ldots$ be the eigenvalues of $A$ arranged by nonincreasing their real parts and according to their multiplicities. By Proposition \ref{PROP: SpectralBoundAwedgeViaA}, the spectral bound $s(A^{[\wedge m]})$ of $A^{[\wedge m]}$ is given by $\sum_{j=1}^{m}\lambda_{j}(A)$ or, if there are less than $m$ eigenvalues, by $-\infty$.

Below, as the orthonormal system we take $\phi_{k}(\theta) = \tau^{-1/2} \exp(i\tau^{-1} 2\pi k\theta)$, although it is only essential that each function $\phi_{k}$ is continuous. Since for $n=r_{1}=r_{2} = 1$ we have $\mathbb{M}^{\wedge}_{m} = \mathbb{U}^{\wedge}_{m}$, we may use the same orthonormal basis in both spaces. So, the approximation scheme is described as follows.
\begin{description}[before=\let\makelabel\descriptionlabel]
	\item[\textbf{(AS.1)}\refstepcounter{desccount}\label{DESC: AS1DelayCompound}] Fix parameters of the scheme: an integer $N > 0$ and reals $T > 0$, $\Omega > 0$, and $\nu_{0} > 0$ such that $-\nu_{0} > s(A^{[\wedge m]})$, see below \eqref{EQ: SpectrumScalarDelayOperator};
	\item[\textbf{(AS.2)}] For the linear delay equation $\dot{x}(t) = \widetilde{A}x_{t}$, compute\footnote{In terms of the semigroup, we have $R^{(1)}_{0}G(t)\psi_{\infty}=x_{\infty}(t)$, $R^{(1)}_{1}G(t)\psi_{\infty} = x_{\infty}(t+\cdot)$, $R^{(1)}_{0}G(t)\psi_{k} = x_{k}(t)$, and $R^{(1)}_{1}G(t)\psi_{k} = x_{k}(t+\cdot)$ for any $t \in [0,T]$.} the scaled fundamental solution $x_{\infty}(\cdot) \colon [-\tau,T] \to \mathbb{R}$ with initial data $x_{\infty}(\theta) = 0$ for $\theta \in [-\tau,0)$ and $x_{\infty}(0) = R^{(1)}_{0}\psi_{\infty}$ (see \eqref{EQ: PropResolventCompoundBActionCompDeltaValue}) and the classical solution $x_{k}(\cdot) \colon [-\tau,T] \to \mathbb{C}$ for each $k \in \{ -N, \ldots, N \}$ with initial data
	$x_{k}(\theta) = \phi_{k}(\theta)$ for $\theta \in [-\tau,0]$;
	\item[\textbf{(AS.3)}] For each $p=-\nu_{0} + i \omega$ with $\omega \in [-\Omega, \Omega]$, compute the following: 
	
	\item[\textbf{(AS.3.1)}] For all integers $-N \leq k_{1} < \cdots < k_{m-1} \leq N$, compute\footnote{This is justified by Proposition \ref{PROP: IntegralHeadComputation}.}
	\begin{equation}
		\begin{split}
			\label{EQ: DelayCompApproximationScheme31M1Comp}
			M^{1}_{k_{1} \ldots k_{m-1}}(\theta_{2},\ldots,\theta_{m}) \coloneq \int_{0}^{T} e^{-pt}\bigwedge_{k_{1}\ldots k_{m-1}}(t)\left(-\tau_{0},\theta_{2},\ldots,\theta_{m} \right)dt
		\end{split}
	\end{equation}
    for $(\theta_{2},\ldots,\theta_{m}) \in [-\tau,0]^{m-1}$, where $\bigwedge\limits_{k_{1}\ldots k_{m-1}}(t)$ is given by \eqref{EQ: ApproxSchemeIntegralToCompute3}. For example, for $m=2$, we have the following expression:
    \begin{equation}
    	\label{EQ: MeasurementComputationM2ApproxScheme}
    	M^{1}_{k}(\theta) = \int_{0}^{T}e^{-p t} \frac{1}{2}\left[ x_{k}(t-\tau_{0}) \cdot x_{\infty}(t + \theta) - x_{k}(t + \theta) \cdot x_{\infty}(t-\tau_{0})  \right]dt,
    \end{equation} 
    where $\theta \in [-\tau,0]$ and $k \in \{-N , \ldots, N\}$;
    \item[\textbf{(AS.3.2)}] For all $-N \leq k_{1} < \cdots < k_{m-1} \leq N$ and $-N \leq l_{1} < \cdots < l_{m-1} \leq N$, compute
    \begin{equation}
    	\label{EQ: DelayCompoundASCoefficientApproximateMatrices}
    	c^{l_{1}\ldots l_{m-1}}_{k_{1} \ldots k_{m-1}} \coloneq \int_{[-\tau,0]^{m-1}}M^{1}_{k_{1} \ldots k_{m-1}}(\bar{\theta}) ( U^{1}_{l_{1} \ldots l_{m-1}}(\bar{\theta}) )^{*}d\bar{\theta},
    \end{equation}
    where $\bar{\theta} = (\theta_{1},\ldots,\theta_{m-1})$ and $U^{1}_{l_{1} \ldots l_{m-1}}$ is given by \eqref{EQ: ExampleCompoundDelayCalcBasisControlSpace};
    \item[\textbf{(AS.3.3)}\refstepcounter{desccount}\label{DESC: DelayCompundAS33}] Let $\mathfrak{n}(\cdot)$ be a bijection from $\{ 1, \ldots, C^{m-1}_{2N+1} \}$, where $C^{m-1}_{2N+1}$ is the binomial coefficient ${2N+1\choose m-1}$, to the set of all multi-indices $k_{1} \ldots k_{m-1}$ with $-N \leq k_{1} < \cdots < k_{m-1} \leq N$. Compute the matrix (see \eqref{EQ: ApproxSchemeWTN})
    \begin{equation}
    	\label{EQ: DelayCompoundASWTNdef}
    	W_{T, N}(p) \coloneq m \cdot \left[ c^{ \mathfrak{n}(i) }_{ \mathfrak{n}(j)} \right]_{i,j = 1}^{C^{m-1}_{2N+1} };
    \end{equation}
    \item[\textbf{\textbf{(AS.3.4)}}\refstepcounter{desccount}\label{DESC: DelayCompundAS34}] In the case of \eqref{EQ: ScalarDelaySymmetricDerivative}, compute the largest singular value $\alpha_{T,N}(p)$ of $W_{T,N}(p)$, and in the case of \eqref{EQ: ScalarDelayMonotoneDerivative}, compute the largest eigenvalue $\alpha_{T,N}(p)$ of the Hermitian matrix
    \begin{equation}
    	S_{T,N}(p) \coloneq -\frac{1}{2} \left[ 	W_{T, N}(p) + 	W^{*}_{T, N}(p) \right],
    \end{equation}
    where $W^{*}_{T, N}(p)$ is the Hermitian transpose of $W_{T, N}(p)$;
    \item[\textbf{(AS.4)}\refstepcounter{desccount}\label{DESC: AS4DelayCompound}] For all $p = -\nu_{0} + i \omega$ and $\omega \in [-\Omega, \Omega]$, verify that $\alpha_{T,N}(p) < \Lambda^{-1}$.
\end{description}

\subsection{Approximation scheme: convergence}
\label{SEC: ApproxSchemeConverg}
Now let us discuss the choice of parameters $T$, $N$, and $\Omega$ in the approximation scheme \nameref{DESC: AS1DelayCompound}--\nameref{DESC: AS4DelayCompound}. Recall that the scheme is based on the approximation of the integral over $[0,T]$ from \eqref{EQ: DelayCompoundScalarComputationIntegralWithBCorollary} for all $-N \leq k_{1} < \cdots < k_{m-1} \leq N$ and $p = -\nu_{0} + i\omega$ with $\omega \in [-\Omega; \Omega]$. This gives approximations $W_{T,N}(p)$ to the finite-dimensional operator $-P_{N} W(p) P_{N}$, where $P_{N}$ is the orthogonal projector onto the span of all $U^{\wedge}_{k_{1}\ldots k_{m-1}}$, see \eqref{EQ: ExampleCompoundDelayCalcBasisControlSpace}. In its turn, $P_{N} W(p) P_{N}$ approximates the transfer operator $W(p)$ appearing in the frequency inequalities \eqref{EQ: DelayCompoundSelfAdjContraintIdenticalBoundsFreqC} and \eqref{EQ: SelfAdjointAdditiveSymmetriztionOptimizationFreqCond}.

For the choice of $T$, we have Corollary \ref{COR: ScalarComputationFormulaWithTailEstimate}, which gives an exponential decay of the integral tail uniformly in $\omega \in \mathbb{R}$ and arbitrary integers $k_{1} < \cdots < k_{m-1}$ as $T \to \infty$. Thus, the choice of $T$ is independent of the other parameters. In particular, we have the following uniform in $\omega \in \mathbb{R}$ and $N$ dynamically exponential (or numerically linear) convergence of matrices in the Euclidean (spectral) norm:
\begin{equation}
	\label{EQ: ApproxSchemeWTN}
	\lim_{T \to \infty} W_{T,N}(p) = -P_{N} W(p) P_{N},
\end{equation}
where $P_{N} W(p) P_{N}$ is identified with a $(C^{m-1}_{2N+1} \times C^{m-1}_{2N+1})$-matrix according to the enumeration $\mathfrak{n}$ from \nameref{DESC: DelayCompundAS33}. In practice, it is sufficient to compare results for several values of $T$. For example, in our experiments we chose $T \approx 15$ and $T \approx 25$ and did not observe any difference.

Regarding the approximations $\alpha_{T,N}(-\nu_{0} + i \omega)$ defined in \nameref{DESC: DelayCompundAS34}, from \eqref{EQ: ApproxSchemeWTN} we have the uniform in $N$ and $\omega \in \mathbb{R}$ dynamically exponential convergence:
\begin{equation}
	\lim_{T \to \infty} \alpha_{T,N}(-\nu_{0} + i\omega) = \alpha_{N}(-\nu_{0} + i\omega).
\end{equation}
Furthermore, Lemma \ref{LEM: RaleighQuotientSelfAdjointOptimizationApproximation} gives the monotone convergence:
\begin{equation}
	\label{EQ: ApproxSchemeAlphaNtoAlpha}
	\lim_{N \to \infty}\alpha_{N}(-\nu_{0} + i\omega) = \alpha(-\nu_{0} + i\omega),
\end{equation}
which is uniform in $\omega \in [-\Omega,\Omega]$ for any fixed $\Omega > 0$ due to Lemma \ref{LEM: ApproxSchemeAlphLipsch}. In practice, one should expect $\alpha_{T,N}(-\nu_{0} + i \omega)$ to stabilize in a given interval $\omega \in [-\Omega, \Omega]$ as $N$ increases. In our experiments reported below and conducted in the case $m=2$, we took $N \in \{2, 5, 10, 20, 30\}$ for $\Omega = 37.5$. Moreover, the experiments indicate that the choice of $T \approx 15$ and $N=10$ is sufficient, since for larger parameters the results become almost indistinguishable in the most interesting segment $\omega \in [-10,10]$.

For the choice of $\Omega$, we leave the following conjecture stated in the case of \eqref{EQ: DelayCompoundSelfAdjContraintIdenticalBoundsFreqC} (for \eqref{EQ: SelfAdjointAdditiveSymmetriztionOptimizationFreqCond} and in more general situations, the statements are analogous).
\begin{conjecture}
	\label{CONJ: DelayCompoundOscillationPattern}
	The norm of the transfer operator $W(-\nu_{0} + i \omega)$ as in \eqref{EQ: DelayCompoundSelfAdjContraintIdenticalBoundsFreqC} is asymptotically almost periodic (in the sense of Bohr) as $|\omega| \to \infty$.
\end{conjecture}
In the case $m=1$, the considered value vanishes as $|\omega| \to \infty$, see \cite{Anikushin2020FreqDelay}. This is not the case for $m \geq 2$, and, indeed, in the examples below, the value shows a repetitive pattern (small oscillations around a positive value) as $|\omega| \to \infty$. In fact, it is asymptotic to the positive value with the convergence of order $O(|\omega|^{-1})$, see Corollary \ref{COR: ExplicitComputationConjecture}, and this may also hold for the general case, so no almost periodicity is in fact required. This indicates that, as in the case $m=1$, frequency inequalities can be verified on a finite time interval $[-\Omega,\Omega]$. Thus, proving Conjecture \ref{CONJ: DelayCompoundOscillationPattern} for more general classes of delay operators should be of great practical interest.

\subsection{Explicit analysis of resolvent equations for $m=2$}
\label{SUBSEC: ExampleResolventEquations}
In the context of \eqref{EQ: ScalarDelayEquationExample}, i.e., assuming $n=r_{1}=r_{2}=1$, let us illustrate the resolvent equations for $m=2$ by means of the operator $A$ corresponding via \eqref{EQ: OperatorAScalarDelayEquations} to $\widetilde{A}$, where $\widetilde{A}\phi = a \phi(0) + b\phi(-\tau)$ for some $a,b \in \mathbb{R}$ and all $\phi \in C([-\tau,0];\mathbb{R})$. 

Consider $\Psi \in \mathcal{L}^{\wedge}_{2}$ such that $R_{12}\Psi = 0$, $R_{1}\Psi = \psi$, and $R_{2}\Psi=-\psi$ for some $\psi \in L_{2}(-\tau,0;\mathbb{C})$. Suppose that $p \in \mathbb{C}$ does not belong to the spectrum of $A^{[\wedge 2]}$. Then, according to Theorem \ref{TH: AdditiveCompoundDelayDescription} and Corollary \ref{COR: AntisymmetricRelationsScalar}, the equation $\Phi = (A^{[\wedge 2]}-pI)^{-1}\Psi$ is equivalent to
\begin{equation}
	\label{EQ: ResolventEquationExample}
	\begin{split}
		\left(\frac{\partial}{\partial \theta_{1}} + \frac{\partial}{\partial \theta_{2}}\right)\Phi(\theta_{1},\theta_{2}) - p \Phi(\theta_{1},\theta_{2}) &= 0,\\
		\frac{d}{d \theta} \Phi(\theta, 0) + a \Phi(\theta, 0) + b \Phi(\theta, -\tau) - p \Phi(\theta,0) &= \psi(\theta),\\
		\frac{d}{d \theta} \Phi(0, \theta) + a \Phi(0, \theta) + b \Phi(-\tau, \theta) - p \Phi(0, \theta) &= -\psi(\theta),
	\end{split}
\end{equation}
where $(\theta_{1},\theta_{2}) \in (-\tau,0)^{2}$ and $\theta \in (-\tau,0)$. For convenience, here we also set $\Phi(\theta,0) = (R_{1}\Phi)(\theta)$ and similarly for other substitutions.

For each $\theta \in [-\tau,0]$, it is convenient to introduce the functions $\Phi^{(1)}_{\theta}(s) \coloneq \Phi(\theta+s,s)$ and $\Phi^{(2)}_{\theta}(s) \coloneq \Phi(s, \theta+s)$ defined for $s \in [-\tau-\theta,0]$. Then we can resolve the first equation in \eqref{EQ: ResolventEquationExample} as $\Phi^{(1)}_{\theta}(s) = c_{1}(\theta) e^{ps}$ and $\Phi^{(2)}_{\theta}(s) = c_{2}(\theta) e^{ps}$, where in fact $c_{1}(\theta) = \Phi(\theta,0) = (R_{1}\Phi)(\theta)$ and $c_{2}(\theta) = \Phi(0,\theta)=(R_{2}\Phi)(\theta)$. In particular, $c_{1}(\theta) = - c_{2}(\theta)$. Note also that $\Phi(\theta,-\tau) = \Phi^{(2)}_{-\tau-\theta}(\theta) = c_{2}(-\tau-\theta)e^{p\theta}$. Substituting all this into the second equation in \eqref{EQ: ResolventEquationExample} gives
\begin{equation}
	\label{EQ: CompoundC1ExampleEquation}
	\frac{d}{d\theta}c_{1}(\theta) + (a-p)c_{1}(\theta) - be^{p\theta} c_{1}(-\tau-\theta) = \psi(\theta).	
\end{equation}
Due to the antisymmetricity, we additionally have $c_{1}(0) = 0$.

All the above transformations are justified by the description of $\mathcal{D}(A^{[\wedge 2]})$ discussed in Remark \ref{REM: DescriptionOfDomainAdditiveCompound}. By our assumptions, there exists a unique $c_{1}(\cdot) \in W^{1,2}(-\tau,0;\mathbb{C})$ satisfying the above conditions. From it the entire $\Phi$ can be obtained.

To determine $c_{1}$ from \eqref{EQ: CompoundC1ExampleEquation}, it is convenient to introduce $y_{1}(\theta) \coloneq c_{1}(\theta)$ and $y_{2}(\theta) \coloneq c_{1}(-\tau-\theta)$. Then \eqref{EQ: CompoundC1ExampleEquation} gives the following linear inhomogeneous system of ordinary differential equations on $[-\tau,0]$:
\begin{equation}
	\label{EQ: ExResolventEquationODE}
	\begin{split}
		&\dot{y}_{1} = (p-a)y_{1} + be^{p\theta}y_{2} + \psi(\theta),\\
		&\dot{y}_{2} = (a-p)y_{2} - be^{-p(\theta + \tau)}y_{1} - \psi(-\tau-\theta),
	\end{split}
\end{equation}
which should be complemented by the conditions 
\begin{equation}
	\label{EQ: ExResolventEquationBoundaryConds}
	y_{1}(0) = y_{2}(-\tau) = 0 \qquad \text{and} \qquad y_{1}(-\tau)=y_{2}(0).
\end{equation}

We immediately have the following lemma.
\begin{lemma}
	\label{LEM: ExResolventsEqsReductionToODE}
	In the above context, there exists a unique solution $(y_{1},y_{2})$ to \eqref{EQ: ExResolventEquationODE} satisfying \eqref{EQ: ExResolventEquationBoundaryConds}. Moreover, it also satisfies $y_{1}(\theta)=y_{2}(-\tau-\theta)$ for all $\theta \in [-\tau,0]$.
\end{lemma}
\begin{proof}
	The existence follows from the existence of $c_{1}$ resolving \eqref{EQ: CompoundC1ExampleEquation} with $c_{1}(0)=0$. For the uniqueness, note that the transformation
	\begin{equation}
		\label{EQ: ExampleCompoundEquationsSymmetry}
		(y_{1}(\theta), y_{2}(\theta)) \mapsto (y_{2}(-\tau-\theta), y_{1}(-\tau-\theta))
	\end{equation}
	is a symmetry for \eqref{EQ: ExResolventEquationODE}, i.e., it takes solutions into solutions. In particular, the difference $\Delta(\theta) \coloneq (y_{1}(\theta), y_{2}(\theta)) - (y_{2}(-\tau-\theta), y_{1}(-\tau-\theta))$ solves the linear homogeneous system, and it has zero initial data at $\theta=0$ under \eqref{EQ: ExResolventEquationBoundaryConds}. So, $\Delta(\theta) \equiv 0$. Consequently, any solution $(y_{1},y_{2})$ satisfying \eqref{EQ: ExResolventEquationBoundaryConds} is symmetric, and $c_{1}(\theta) \coloneq y_{1}(\theta)$ solves \eqref{EQ: CompoundC1ExampleEquation}. Since $c_{1}$ is unique, the same holds for $(y_{1},y_{2})$.
\end{proof}

Now we are going to resolve \eqref{EQ: ExResolventEquationODE}. Let us start with the following.
\begin{lemma}
	Let $D=D(\theta)$ be the matrix of the linear part of \eqref{EQ: ExResolventEquationODE}, i.e.,
	\begin{equation}
		D(\theta) = \begin{pmatrix}
			p-a & be^{p\theta}\\
			-be^{-p(\theta+\tau)} & a-p
		\end{pmatrix}.
	\end{equation}
	Then solutions to $\dot{y}(\theta) = D(\theta)y(\theta)$ are delivered by the formula
	\begin{equation}
		\label{EQ: ExResolventEqsLinearSolutionsFormula}
		y(\theta)=
		\begin{pmatrix}
			e^{p\theta} & 0\\
			0 & 1
		\end{pmatrix}
		e^{D_{0}(\theta+\tau)}
		\begin{pmatrix}
			e^{p\tau} & 0\\
			0 & 1
		\end{pmatrix} y(-\tau),
	\end{equation}
	where $D_{0}$ is given by
	\begin{equation}
		\label{EQ: ResolventComputationD0matrix}
		D_{0} = \begin{pmatrix}
			-a & b\\
			-be^{-p\tau} & a-p
		\end{pmatrix}.
	\end{equation}
\end{lemma}
\begin{proof}
	This can be achieved by applying the change of variables
	\begin{equation}
		(y_{1}(\theta),y_{2}(\theta)) \mapsto (e^{-p\theta}y_{1}(\theta), y_{2}(\theta))
	\end{equation} 
	in the equations, which results in a linear system with constant coefficients given by the matrix $D_{0}$.
\end{proof}

For convenience, we set $G_{D}(\theta)$ to be the product of matrices from \eqref{EQ: ExResolventEqsLinearSolutionsFormula}. Thus, solutions to the homogeneous linear system are given by $y(\theta) = G_{D}(\theta)y(-\tau)$. In virtue of the symmetry \eqref{EQ: ExampleCompoundEquationsSymmetry}, we have
\begin{equation}
	\label{EQ: ExampleResolventEqsSymmtryForG}
	G_{D}(\theta) = T G_{D}(-\tau-\theta) T G_{D}(0), \qquad \text{where} \quad T=\begin{pmatrix}
		0 & 1\\
		1 & 0
	\end{pmatrix}.
\end{equation}

By applying the variation of constants formula, we immediately obtain the following.
\begin{corollary}
	Solutions to \eqref{EQ: ExResolventEquationODE} are given by
	\begin{equation}
		\label{EQ: ExResolventEqsGeneralInhFormula}
		\begin{pmatrix}
			y_{1}(\theta)\\
			y_{2}(\theta)
		\end{pmatrix}
		= G_{D}(\theta) 		\begin{pmatrix}
			y_{1}(-\tau)\\
			y_{2}(-\tau)
		\end{pmatrix}
		+ \int_{-\tau}^{\theta} G_{D}(\theta) (G_{D}(s))^{-1} \begin{pmatrix}
			\psi(s)\\
			-\psi(-\tau-s)
		\end{pmatrix}
		ds,
	\end{equation}
	where
	\begin{equation}
		G_{D}(\theta) (G_{D}(s))^{-1} = \begin{pmatrix}
			e^{p\theta} & 0\\
			0 & 1
		\end{pmatrix}
		e^{D_{0}(\theta-s)}
		\begin{pmatrix}
			e^{-ps} & 0\\
			0 & 1
		\end{pmatrix}.
	\end{equation}
\end{corollary}

Substituting the boundary conditions \eqref{EQ: ExResolventEquationBoundaryConds} into \eqref{EQ: ExResolventEqsGeneralInhFormula} with $\theta=0$, we obtain
	\begin{equation}
	\label{EQ: ExResolventEqsGeneralInhFormulaSubstituted}
	\begin{pmatrix}
		0\\
		y_{1}(-\tau)
	\end{pmatrix}
	= G_{D}(0) 		\begin{pmatrix}
		y_{1}(-\tau)\\
		0
	\end{pmatrix}
	+ \int_{-\tau}^{0} G_{D}(0) (G_{D}(s))^{-1} \begin{pmatrix}
		\psi(s)\\
		-\psi(-\tau-s)
	\end{pmatrix}
	ds.
\end{equation}
By Lemma \ref{LEM: ExResolventsEqsReductionToODE}, there exists a unique solution $y_{1}(-\tau) \in \mathbb{C}$ to this equation. Suppose that 
\begin{equation}
	\label{EQ: ExampleResolventEqsGcoeffs}
	G_{D}(\theta) = \begin{pmatrix}
		g_{11}(\theta) & g_{12}(\theta)\\
		g_{21}(\theta) & g_{22}(\theta)
	\end{pmatrix}
	\quad \text{and} \quad
	G_{D}(\theta) (G_{D}(s))^{-1} = \begin{pmatrix}
		g'_{11}(\theta,s) & g'_{12}(\theta,s)\\
		g'_{21}(\theta,s) & g'_{22}(\theta,s)
	\end{pmatrix}.
\end{equation}
Then \eqref{EQ: ExResolventEqsGeneralInhFormulaSubstituted} and integration by parts give the following expressions\footnote{It is unclear to us how the equations $g_{11}(0)=0$ and $1-g_{21}(0) = 0$ with respect to $p \in \mathbb{C}$ are related to the spectrum of $A^{[\wedge 2]}$. However, from the uniqueness of $y_{1}(-\tau)$, at least one of the expressions in \eqref{EQ: DelayCompoundScalarExplicitYtauExpressions} must be well defined. Furthermore, from the asymptotic analysis given below, both expressions are well defined for all sufficiently large $\omega$. Then, again by the uniqueness of $y_{1}(-\tau)$, the corresponding kernels must coincide for such $\omega$. Since they are analytic functions of $\omega$, possible zeros of the denominator must cancel with the zeros of the enumerator, thereby delivering an analytic continuation. We do not know if this can actually happen (in our experiments both denominators are always nonzero).}:
\begin{equation}
	\label{EQ: DelayCompoundScalarExplicitYtauExpressions}
	\begin{split}
		y_{1}(-\tau) &= -\frac{1}{g_{11}(0)} \int_{-\tau}^{0}\left( g'_{11}(0,s) - g'_{12}(0,-\tau-s) \right) \psi(s)ds,\\ 
		y_{1}(-\tau) &= \frac{1}{1-g_{21}(0)} \int_{-\tau}^{0}\left( g'_{21}(0,s) - g'_{22}(0,-\tau-s) \right) \psi(s)ds.
	\end{split}
\end{equation}

By substituting this $y_{1}(-\tau)$ into \eqref{EQ: ExResolventEqsGeneralInhFormula} with $y_{2}(-\tau)=0$, one sees that $y_{1}$ and $y_{2}$ can be represented as follows:
\begin{equation}
	\label{EQ: ExampleCompoundIntegralOperatorsAbstract}
	\begin{split}
		y_{1}(\theta) &= \int_{-\tau}^{0}K_{1}(\theta,s)\psi(s)ds,\\
		y_{2}(\theta) &= \int_{-\tau}^{0}K_{2}(\theta,s)\psi(s)ds
	\end{split}
\end{equation}
with certain $L_{2}$-summable kernels $K_{1}$ and $K_{2}$ depending on $p$, see \eqref{EQ: ExampleResolventEquationsK2Formula} and \eqref{EQ: ExplicitComputationsK1Formula} for the explicit formulas.

Let us discuss how this is related to the frequency inequality \eqref{EQ: DelayCompoundSmithIneqGeneral}. For simplicity, we assume that $C\phi = \phi(-\tau_{0})$ with $\tau_{0}=0$ or $\tau_{0}=\tau$, see Remark \ref{REM: ExplicitComputationGeneralTau0} for the case of general $\tau_{0}$.

In terms of Sections \ref{SUBSEC: InducedControlOperators} and \ref{SUBSEC: InducedMeasurementOperators}, the action of the transfer operator $W(p)=-C^{\wedge}_{2}(A^{[\wedge 2]}-pI)^{-1}B^{\wedge}_{2}$ by components sends $\psi = -\widetilde{B} \eta^{2}_{1}$ into $M^{2}_{1}(\theta) = (C^{(2)}_{2,2}R_{12}\Phi)(\theta) = \Phi(\theta,-\tau_{0})$ and analogously for $\eta^{1}_{2} = -\eta^{2}_{1}$ and $M^{1}_{2} = -M^{2}_{1}$. We assume that $\widetilde{B} = \pm 1$, and since we deal with the norm of $W(p)$, we can drop the possible minus sign in the final result. Since $\Phi(\theta,-\tau) = -c_{1}(-\tau-\theta) e^{p\theta} = -y_{2}(\theta)e^{p\theta}$ and $\Phi(\theta,0) = c_{1}(\theta) = y_{1}(\theta)$, in terms of the present section, we are interested in the norm of the integral operator
\begin{equation}
	\label{EQ: IntegralOperatorExplicit}
	\psi \mapsto 
	\begin{cases}
		e^{p\theta}y_{2}(\cdot) \qquad &\text{if} \quad \tau_{0} = \tau,\\
		y_{1}(\cdot) \qquad &\text{if} \quad \tau_{0}=0.
	\end{cases}
\end{equation}

From the above considerations, we can estimate the norm of $W(p)$ from above via the $L_{2}$-norm of the kernel, i.e.,
\begin{equation}
	\label{EQ: ExplicitEstimateForWViaL2NormK}
	\| W(p) \|_{\mathcal{L}((\mathbb{U}^{\wedge}_{2})^{\mathbb{C}};(\mathbb{M}^{\wedge}_{2})^{\mathbb{C}})} \leq \begin{cases}
			\| e^{p\theta} K_{2} \|_{L_{2}((-\tau,0)^{2};\mathbb{C})} \qquad &\text{if} \quad \tau_{0} = \tau,\\
			\| K_{1} \|_{L_{2}((-\tau,0)^{2};\mathbb{C})} \qquad &\text{if} \quad \tau_{0}=0.
		\end{cases}
\end{equation}
However, the inequality can be strict, see Fig. \ref{FIG: CompoundMGTKernelApproximations}.

Let us discuss computations by means of the kernel $K_{2}$. It can be expressed using the coefficients from \eqref{EQ: ExampleResolventEqsGcoeffs} as follows:
\begin{equation}
	\label{EQ: ExampleResolventEquationsK2Formula}
	\begin{split}
		K_{2}(\theta,s) = \frac{g_{21}(\theta) (g'_{21}(0,s) - g'_{22}(0,-\tau-s))}{1-g_{21}(0)} + \rchi_{[-\tau,\theta]}(s)g'_{21}(\theta,s) -\\ \rchi_{[-\tau-\theta,0]}(s)g'_{22}(\theta,-\tau-s),
	\end{split}
\end{equation}
where $\rchi_{\mathcal{I}}$ denotes the characteristic function of the interval $\mathcal{I}$. In their turn, the coefficients can be expressed in terms of entries constituting the matrix exponential $e^{D_{0}t}$, which can be computed explicitly via the following well-known formula.
\begin{lemma}
	\label{LEM: ExplicitExponentialMatrix2t2}
	Let $D_{0}$ be a $2\times 2$-matrix with complex entries. Then for any $t \in \mathbb{R}$ we have
	\begin{equation}
		\label{EQ: ExplicitMatrixExponential}
		e^{D_{0} t} = e^{\alpha t} \left[ \left(\cosh(\delta t) - \alpha \frac{\sinh(\delta t)}{\delta}\right)I_{2} + \frac{\sinh(\delta t)}{\delta} D_{0} \right],
	\end{equation}
	where $\alpha \coloneq \operatorname{tr}D_{0}/2$, $\delta \coloneq \pm \sqrt{ - \det( D_{0} - \alpha I_{2}) }$, and $I_{2}$ is the identity $2\times 2$-matrix.
\end{lemma}

From the symmetry \eqref{EQ: ExampleResolventEqsSymmtryForG}, one can also express  $g'_{21}(\theta,s)$ and $g'_{22}(\theta,-\tau-s)$ via the sum of decomposable functions, i.e., products of functions depending only on $\theta$ or $s$. Thus, it is only the characteristic functions in \eqref{EQ: ExampleResolventEquationsK2Formula} that are indecomposable.

So, there is an explicit representation of the kernels and, consequently, of the transfer operator $W(p)$. However, we do not know whether the norm of $W(p)$ can be explicitly represented. In fact, using the explicit formulas, we may establish that the norm of $W(-\nu_{0} + i\omega)$ tends to a constant as $|\omega| \to \infty$, and even the explicit computation of the constant value is not known to us. Let us expound this in the case of $\tau_{0} = \tau$, omitting cumbersome transformations but emphasizing key relations. We refer to the experimental results at the end of this section that show the agreement with the theoretical investigations.

\begin{proposition}
	In terms of Lemma \ref{LEM: ExplicitExponentialMatrix2t2}, for the matrix $D_{0}$ from \eqref{EQ: ResolventComputationD0matrix} with $p=-\nu_{0} + i \omega$ we have $\alpha = -p/2$ and
	\begin{equation}
		\label{EQ: ExplicitComputationDeltaAsymptotic}
		\delta = \delta(-\nu_{0} + i\omega) = \left(a + \frac{1}{2}\nu_{0}\right) - i \frac{\omega}{2} + O\left(\frac{1}{|\omega|}\right) \qquad \text{as} \quad |\omega| \to \infty,
	\end{equation}
\end{proposition}
\begin{proof}
	This follows from an asymptotic analysis according to the definition of $\delta$. Since it requires cumbersome transformations, we leave it to the interested reader.
\end{proof}

From \eqref{EQ: ExplicitComputationDeltaAsymptotic} one can study the asymptotic behavior of the kernels as follows.
\begin{proposition}
	\label{PROP: ExplicitComputationAsymptoticKernel}
	For $p = -\nu_{0} + i \omega$ and $\delta_{0} = (a+\nu_{0}/2) - i \omega/2$, we have 
	\begin{equation}
		\label{EQ: ExplicitComputationsK2AsymptoticKernel}
			K_{2}(\theta,s) = \bar{K}_{2}(\theta,s) + O\left(\frac{1}{|\omega|}\right) \qquad \text{as} \quad |\omega| \to \infty
	\end{equation}
	uniformly in $(\theta,s) \in [-\tau,0]^{2}$, where 
	\begin{equation}
		\label{EQ: ExplicitComputationAsymptoticKernel}
		\bar{K}_{2}(\theta,s) \coloneq -e^{ -p\theta / 2} e^{\delta_{0} \theta} \rchi_{[-\tau-\theta,0]}(s) e^{-p(\tau+s)/2} e^{\delta_{0} (\tau + s)}.
	\end{equation}
\end{proposition}
\begin{proof}
	Suppose that
	\begin{equation}
		e^{D_{0}t} = \begin{pmatrix}
			g^{0}_{11}(t) & g^{0}_{12}(t)\\
			g^{0}_{21}(t) & g^{0}_{22}(t)
		\end{pmatrix} \qquad \text{for all} \quad t \geq 0.
	\end{equation}
	By \eqref{EQ: ExplicitComputationDeltaAsymptotic} and \eqref{EQ: ExplicitMatrixExponential}, it is clear that
	\begin{equation}
		\label{EQ: ExplicitComputationG0AsymptoticExpansion}
		\begin{pmatrix}
			g^{0}_{11}(t) & g^{0}_{12}(t)\\
			g^{0}_{21}(t) & g^{0}_{22}(t)
		\end{pmatrix} =
		\begin{pmatrix}
			e^{-pt / 2} e^{-\delta_{0} t} & 0\\
			0 & e^{-pt / 2} e^{\delta_{0} t}
		\end{pmatrix} + O\left(\frac{1}{|\omega|}\right).
	\end{equation}
	Moreover, for the matrices from \eqref{EQ: ExampleResolventEqsGcoeffs} we have
	\begin{equation}
		\label{EQ: ExplicitComputationGprimeGviaG0}
		\begin{split}
		\begin{pmatrix}
			g_{11}(\theta) & g_{12}(\theta)\\
			g_{21}(\theta) & g_{22}(\theta)
		\end{pmatrix} &=
		\begin{pmatrix}
			e^{p(\theta + \tau)}g^{0}_{11}(\tau+\theta) & e^{p\theta}g^{0}_{12}(\tau+\theta)\\
			e^{p\tau}g^{0}_{21}(\tau+\theta) & g^{0}_{22}(\tau+\theta)
		\end{pmatrix},\\
		\begin{pmatrix}
			g'_{11}(\theta,s) & g'_{12}(\theta,s)\\
			g'_{21}(\theta,s) & g'_{22}(\theta,s)
		\end{pmatrix} &=
		\begin{pmatrix}
			e^{p(\theta-s)}g^{0}_{11}(\theta-s) & e^{p\theta}g^{0}_{12}(\theta-s)\\
			e^{-ps}g^{0}_{21}(\theta-s) & g^{0}_{22}(\theta-s)
		\end{pmatrix}.
		\end{split}
	\end{equation}
	From this and \eqref{EQ: ExplicitComputationG0AsymptoticExpansion}, it is clear that all the entries are uniformly bounded, and, moreover, the subdiagonal entries vanish as $|\omega| \to \infty$ with the order of $O(|\omega|^{-1})$. Thus, in \eqref{EQ: ExampleResolventEquationsK2Formula}, all the terms vanish except the last one. Using its expression from \eqref{EQ: ExplicitComputationGprimeGviaG0} and \eqref{EQ: ExplicitComputationG0AsymptoticExpansion}, we obtain \eqref{EQ: ExplicitComputationsK2AsymptoticKernel}.
\end{proof}

Clearly, the norm of the integral operator with the kernel $\bar{K}_{2}$ from \eqref{EQ: ExplicitComputationAsymptoticKernel} equals the norm of the integral operator with the kernel $|\bar{K}_{2}|$, which is independent of $\omega$. 
\begin{remark}
	\label{REM: K1Kernel}
	For the kernel $K_{1}$, we have the following expression:
	\begin{equation}
		\label{EQ: ExplicitComputationsK1Formula}
		\begin{split}
			K_{1}(\theta,s) = \frac{g_{11}(\theta)(g'_{12}(0,-\tau-s) - g'_{11}(0,s))}{g_{11}(0)} + \rchi_{[-\tau,\theta]}(s) g'_{11}(\theta,s) -\\ \rchi_{[-\tau-\theta,0]}(s) g'_{12}(\theta,-\tau-s),
		\end{split}
	\end{equation}
	which coincides in the $L_{2}$-sense with $K_{2}(-\tau-\theta,s)$ according to \eqref{EQ: ExampleCompoundIntegralOperatorsAbstract} and the symmetry of solutions from Lemma \ref{LEM: ExResolventsEqsReductionToODE}. In particular, \eqref{EQ: ExplicitComputationsK2AsymptoticKernel} gives that
	\begin{equation}
		\bar{K}_{1}(\theta, s) \coloneq \bar{K}_{2}(-\tau-\theta,s) = - e^{p\theta/2} e^{\delta_{0} \theta} \rchi_{[\theta, 0]}(s) e^{-ps/2} e^{-\delta_{0} s}
	\end{equation}
	is the asymptotic kernel for $K_{1}(\theta,s)$, and the norms of the integral operators with the kernels $\bar{K}_{1}$ and $|\bar{K}_{1}|$ are the same and, in particular, do not depend on $\omega$.
\end{remark}

The above considerations give the following.
\begin{corollary}
	\label{COR: ExplicitComputationConjecture}
	In the above context, for $-\nu_{0} > s(A^{[\wedge 2]})$ and $\tau_{0} = \tau$ or $\tau_{0} = 0$, the norm of $W(-\nu_{0} + i \omega)$ tends to a constant $\bar{W}=\bar{W}(a, \tau, \nu_{0})$ with the order of $O(|\omega|^{-1})$ as $|\omega| \to \infty$. More precisely, $\bar{W}$ is the norm of the integral operator with the kernel
	\begin{equation}
		|e^{p\theta}\bar{K}_{2}|(\theta,s) = e^{a \theta} \rchi_{[-\tau-\theta,0]}(s) e^{(a+\nu_{0})(\tau+s)}
	\end{equation}
	in the case $\tau_{0} = \tau$ or with the kernel
	\begin{equation}
		|\bar{K}_{1}|(\theta,s) = e^{-(a + \nu_{0})\theta} \rchi_{[\theta,0]}(s) e^{(a + \nu_{0})s}
	\end{equation}
	in the case $\tau_{0} = 0$. In particular, Conjecture \ref{CONJ: DelayCompoundOscillationPattern} is valid in these cases.
\end{corollary}
\begin{remark}
	\label{REM: ExplicitComputationGeneralTau0}
	For general $\tau_{0} \in [0,\tau]$, one can also show an analog of Corollary \ref{COR: ExplicitComputationConjecture}. Here we have
	\begin{equation}
		\Phi(\theta,-\tau_{0}) = \begin{cases}
			y_{1}(\theta+\tau_{0}) e^{-\tau_{0} p} \qquad &\text{if} \quad \theta \in [-\tau, -\tau_{0}],\\
			-y_{1}(-\tau_{0} - \theta)e^{p\theta} \qquad &\text{if} \quad \theta \in [-\tau_{0},0],
		\end{cases}
	\end{equation}
	and the asymptotic kernel can be expressed as follows:
	\begin{equation}
		\begin{split}
			\bar{K}_{\tau_{0}}(\theta,s) = e^{p(\theta-\tau_{0})/2} [ \rchi_{[-\tau_{0},0]}(\theta) e^{-\delta_{0}(\theta + \tau_{0})} -\\  \rchi_{[-\tau,-\tau_{0}]}(\theta) e^{\delta_{0}(\theta+\tau_{0})} ] \rchi_{[\theta,0]}(s) e^{-ps/2} e^{-\delta_{0} s}.
		\end{split}
	\end{equation}
	Since the characteristic functions in the square brackets are complementary, the norms of the integral operators with the kernels $\bar{K}_{\tau_{0}}$ and $|\bar{K}_{\tau_{0}}|$ are the same and do not depend on $\omega$. Thus, Conjecture \ref{CONJ: DelayCompoundOscillationPattern} is also valid in this case.
\end{remark}

It seems that even the explicit computation of the asymptotic norm $\bar{W}$ is not possible\footnote{See the \href{https://mathoverflow.net/q/500667}{discussion} on MathOverflow: https://mathoverflow.net/q/500667.}. However, to justify the verification of frequency inequalities on a finite segment, it may be sufficient to use the $L_{2}$-norm of the asymptotic kernel, see Fig.~\ref{FIG: CompoundMGTKernelApproximations}.

We conducted numerical experiments by means of the Mackey--Glass equations, namely, \eqref{EQ: MackeyGlassNumericalSchemeExample} with $\gamma = 0.1$, $\beta = 0.2$, $\kappa = 10$, and $\Lambda$ given by \eqref{EQ: MackeyGlassStabilityLambdaDef}. To avoid confusion, let us denote $\tau$ from \eqref{EQ: MackeyGlassNumericalSchemeExample} by $\tau'$. Then in terms of the present section, we have $\tau'=4.5$ and
\begin{equation}
	\label{EQ: ExplicitResolventParamatersForTest}
	a = -\tau' \gamma, \quad b=(\tau'\beta - \Lambda), \quad \tau_{0}=\tau=1, \quad \text{and} \quad \Lambda=\frac{1}{2}\tau'\beta\left( \frac{(\kappa-1)^{2}}{\kappa} +1\right).
\end{equation}
Moreover, we considered the Suarez--Schopf model as in \eqref{EQ: Suarez-SchopfLinearizedRewriten} with $\alpha = 0.6$, $\tau=0.83$, and $\Lambda=\Lambda_{R}$ with $R=R_{0}(\alpha,\tau)$ given by Lemma \ref{LEM: SSmodelRadiusEstimateWithRestr}. In terms of the present section, this gives
\begin{equation}
	\label{EQ: ExplicitResolventParamatersForTestSuarezSchopf}
	a = 1-\Lambda, \quad b=-\alpha, \quad \text{and} \quad \tau_{0}=0.
\end{equation}

For these parameters, we computed\footnote{Integrals are approximated via the Simpson $1/3$-rule using the uniform grid of $1001$ points on $[-\tau,0]$.} the $L_{2}$-norms of the kernels $e^{p\theta}K_{2}$ and $K_{1}$ as in \eqref{EQ: ExplicitEstimateForWViaL2NormK}. Moreover, similarly to Section \ref{SEC: ApproximationFreqIneqDelayCompound}, we truncated the integral operator \eqref{EQ: IntegralOperatorExplicit} using its explicit representation and the basis of trigonometric monomials $\phi_{k}(\theta)=\tau^{-1/2}e^{i2\pi k \theta/\tau}$ with $-N \leq k \leq N$. In terms of \eqref{EQ: ApproxSchemeWTN}, such truncations correspond to the approximations $P_{N} W(p) P_{N}$. By \eqref{EQ: ApproxSchemeAlphaNtoAlpha}, their norms $\alpha_{N}(p)$ monotonically converge to the norm of $W(p)$. 

\begin{figure}[t]
	\begin{minipage}{.5\textwidth}
		\includegraphics[width=\textwidth,angle=0]{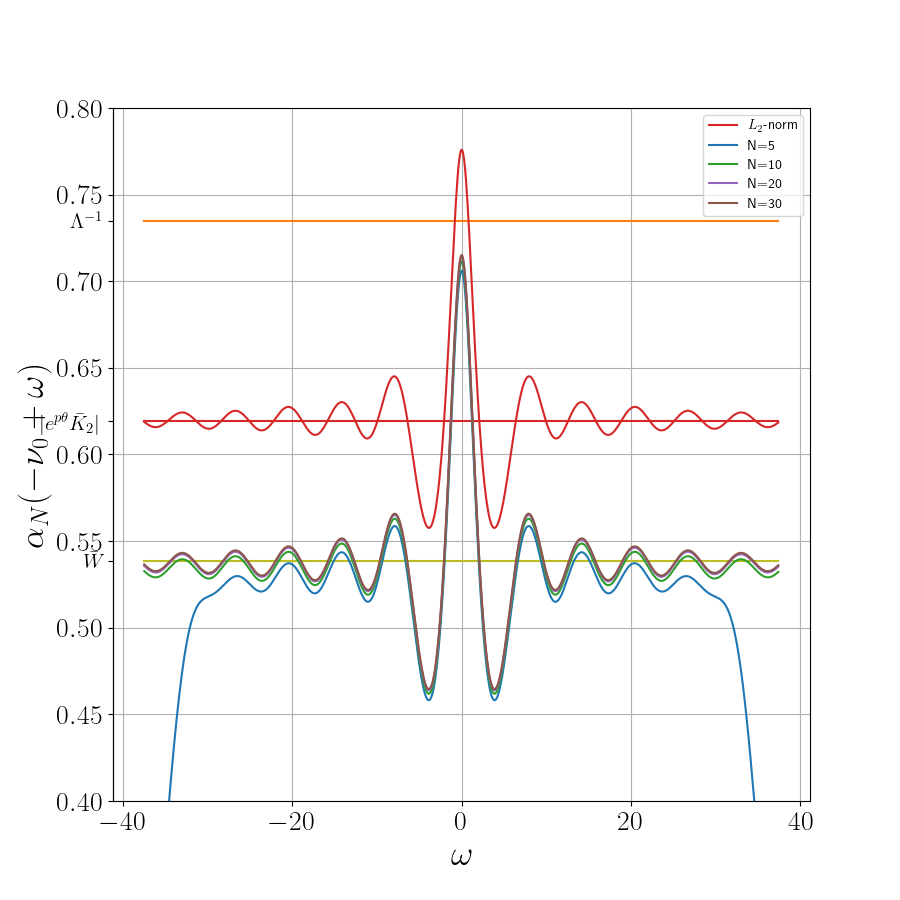}
	\end{minipage}%
	\begin{minipage}{.5\textwidth}
		\includegraphics[width=\textwidth,angle=0]{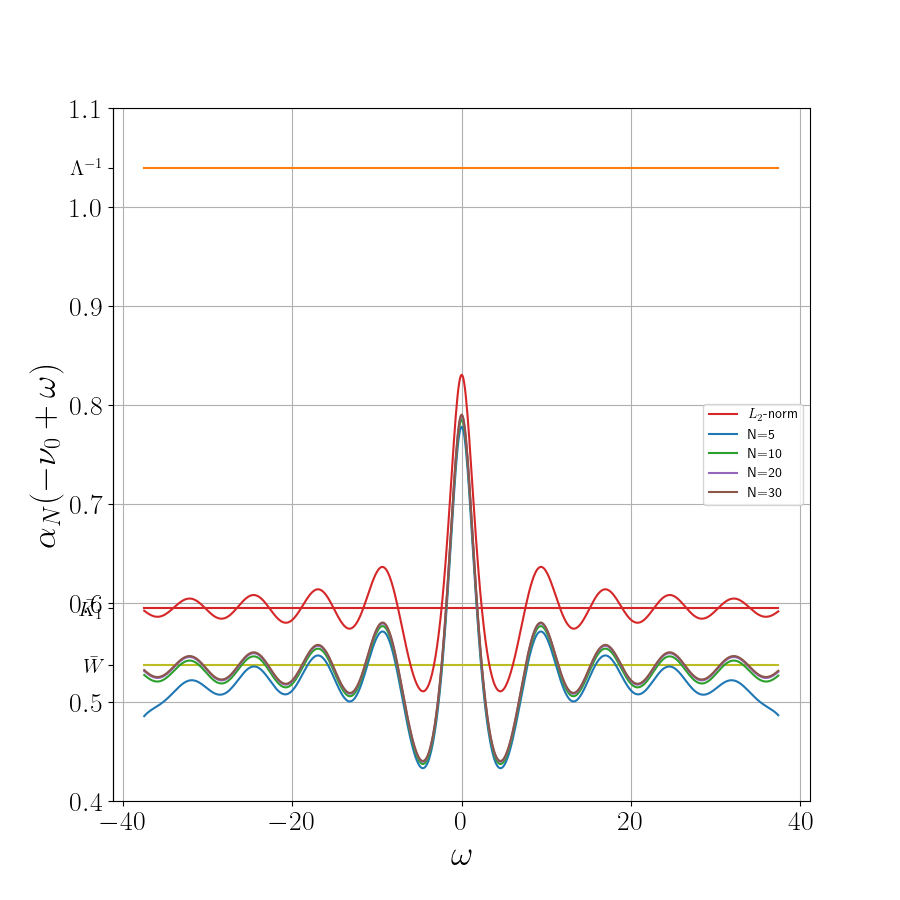}
	\end{minipage}%
	\caption{Graphs of the largest singular values $\alpha_{N}(-\nu_{0} + i \omega)$ for $\nu_{0}=0.01$ versus $\omega$ computed via the explicit representation of the transfer operator in the case of \eqref{EQ: ExplicitResolventParamatersForTest} (left) and \eqref{EQ: ExplicitResolventParamatersForTestSuarezSchopf} (right). Truncation parameters are taken as $N=5$ (blue), $N=10$ (green), $N=20$ (purple), and $N=30$ (brown). The red graph is the $L_{2}$-norm of the kernel from \eqref{EQ: ExplicitEstimateForWViaL2NormK}. The horizontal lines pass through the threshold value $\Lambda^{-1}$ (orange), the $L_{2}$-norms $|e^{p\theta}\bar{K}_{2}|$ (left) and $|\bar{K}_{1}|$ (right) of the asymptotic kernels (red), and the norm $\bar{W}$ of the asymptotic integral operator (olive) on the vertical axis. See the repository for implementation details.}
	\label{FIG: CompoundMGTKernelApproximations}
\end{figure}

Figure \ref{FIG: CompoundMGTKernelApproximations} shows some results for $N \in \{ 5,10,20,30\}$ in the case of \eqref{EQ: ExplicitResolventParamatersForTest} (left) and \eqref{EQ: ExplicitResolventParamatersForTestSuarezSchopf} (right). They indicate the presence of a gap between the norm of $W(p)$ and the $L_{2}$-norm of the kernel. Comparison with Fig.~\ref{FIG: CompoundSST=15} and Fig.~\ref{FIG: CompoundMGT=15} shows that the truncations agree with $W_{T,N}(p)$ delivered by the approximation scheme, which is expectable due to \eqref{EQ: ApproxSchemeWTN}. Moreover, the $L_{2}$-norm and the approximating norms agree with their constant asymptotic values\footnote{To approximate $\bar{W}$, we used analogous truncations with $N=1000$ and explicit formulas for the corresponding integrals.} as $|\omega| \to \infty$ according to Proposition \ref{PROP: ExplicitComputationAsymptoticKernel}, Remark \ref{REM: K1Kernel}, and Corollary \ref{COR: ExplicitComputationConjecture}.

For $m>2$, it is unclear to us whether the resolvent equations can be reduced to a problem amenable to direct computations via standard methods. We only note that the embedding of the diagonal Sobolev space $\mathcal{W}^{2}_{D}((-\tau,0)^{k};\mathbb{C})$ into $L_{2}((-\tau,0)^{k};\mathbb{C})$ is not compact for $k > 1$, see \cite[Remark 4.4]{Anikushin2023Comp}. Thus, for such problems the transfer operator is no longer compact and, in particular, cannot be expressed via integral operators with $L_{2}$-summable kernels. On the other hand, in the case of $m=2$, a similar approach for systems of equations can be developed. We plan to consider this case in future work.

As the truncations deliver only bounds from below for the norm of $W(p)$, it is interesting to obtain upper estimates, as these are more relevant for the verification of frequency inequalities. Recently, we found that the method of iterative nonlinear programming, which was suggested in our paper \cite{AnikushinRomanov2025RobustEstimates} to resolve minimax problems, works well for optimization of Schur test functions in order to obtain refined upper bounds for the norm of integral operators. In particular, this approach delivers much sharper upper bounds than \eqref{EQ: ExplicitEstimateForWViaL2NormK}, especially in a neighborhood of $\omega = 0$. Relevant discussions are given in our paper \cite{AnikushinRomanov2025Schur}.

Armed with explicit formulas for the test functions and kernels and effective bounds for the constant in $O(|\omega|^{-1})$ from \eqref{EQ: ExplicitComputationDeltaAsymptotic}, it shall be possible to make the upper estimates rigorous with the aid of interval arithmetic.

%% file: SuarezSchopfOscill.tex
\subsection{Suarez--Schopf delayed oscillator}
\label{SEC: SuarezSchopfCompoundStab}
In this section, we aim to apply the developed machinery to study the global stability of the delayed oscillator proposed by Suarez and Schopf in \cite{Suarez1988} as a model for the El Ni\~{n}o--Southern Oscillation (ENSO). It is given by a scalar equation with a single delay:
\begin{equation}
	\label{EQ: SuarezSchopfCompoundExp}
	\dot{x}(t) = x(t) - \alpha x(t-\tau) - x^{3}(t),
\end{equation}
where $\tau > 0$ and $\alpha \in (0,1)$ are parameters. 

It can be shown, see \cite[Section 4]{Anikushin2022Semigroups}, that \eqref{EQ: SuarezSchopfCompoundExp} generates a dissipative semiflow $\varphi$ in the space $\mathbb{E} = C([-\tau,0];\mathbb{R})$ given by $\varphi^{t}(\phi_{0}) = x_{t}$, where $x \colon [-\tau,\infty) \to \mathbb{R}$ is a classical solution to \eqref{EQ: SuarezSchopfCompoundExp} such that $x_{0} = \phi_{0}$, and $x_{t}(\theta) = x(t+\theta)$ for $\theta \in [-\tau,0]$ denotes the $\tau$-history segment of $x(\cdot)$ at $t \geq 0$. Moreover, the global attractor $\mathcal{A}$ of $\varphi$ lies in the ball of radius\footnote{We endow $\mathbb{E}$ with the supremum norm.} $\sqrt{1+\alpha}$ centered at zero, and any ball with a radius not smaller than that is positively invariant.

Essential limitations on the dynamics of \eqref{EQ: SuarezSchopfCompoundExp} follow from the fact that it belongs to the class of systems with monotone negative feedback studied in \cite{MalletParetSell1996}. In particular, the dynamics of $\varphi$ satisfies the Poincar\'{e}--Bendixson trichotomy, i.e., the $\omega$-limit set of any point can be either a single equilibrium, or a single periodic orbit, or be a subset of equilibria together with complete orbits connecting them. Below, this trichotomy will be used to show that $\varphi$ is globally stable if certain frequency conditions are satisfied, see Proposition \ref{PROP: DelayCompoundGlobStabSuarezSchopf}.

It is clear that the set of equilibria for $\varphi$ is constituted by the zero equilibrium $\phi^{0}(\cdot) \equiv 0$ and the pair of symmetric ones $\phi^{\pm}(\cdot) \equiv \pm \sqrt{1-\alpha}$. For the considered parameters, standard local analysis shows that $\phi^{0}$ always has a one-dimensional unstable manifold. Moreover, for relatively small $\alpha$ and $\tau$, the symmetric equilibria $\phi^{\pm}$ are linearly stable. They lose their stability with a pair of complex-conjugate characteristic roots crossing the imaginary axis. These parameters correspond to the so-called neutral curve on the plane $(\tau,\alpha)$ (see Fig.~1 in \cite{AnikushinRom2023SS}), and the parameters below this curve correspond to the region of linear stability. 

Usually, the model \eqref{EQ: SuarezSchopfCompoundExp} is considered with parameters above the neutral curve, where it demonstrates stable periodic oscillations. However, in our work \cite{AnikushinRom2023SS}, we used analytical-numerical techniques to show that in the region of linear stability the presence of unstable periodic orbits, hidden periodic orbits, and homoclinic ``figure eights'' is possible if the parameters are taken sufficiently close to the neutral curve. Since systems with such rich multistability may be sensitive to external disturbances, and ENSO exhibits irregular behavior, these parameters seem to be more related to the modeled phenomenon. In this direction, we demonstrated that the additive effect of a small periodic forcing can cause chaotic behavior in the model.

Thus, the global stability of \eqref{EQ: SuarezSchopfCompoundExp} cannot be determined from the linear stability of equilibria. In fact, the theory of normal forms shows that on the neutral curve, symmetric equilibria $\phi^{\pm}$ undergo the Andronov--Hopf bifurcation, which is subcritical, and hence there exist unstable periodic orbits surrounding the equilibria for some parameters below the neutral curve. It is expected that the region of global stability is the region below what we called in \cite{AnikushinRom2023SS} the lower hidden curve. On this curve, the system is expected to undergo a saddle-node bifurcation of two (stable and unstable) large periodic orbits. Using normal forms, this scenario can be rigorously justified in a small neighborhood of the parameter $(\tau,\alpha)=(1,1)$. In \cite{AnikushinRom2023SS}, we numerically demonstrated the possibility of continuing the corresponding bifurcation curves using two-dimensional inertial manifolds.

Since the boundary of global stability in \eqref{EQ: SuarezSchopfCompoundExp} is determined by nonlocal bifurcations\footnote{In particular, it is hidden in the terminology of \cite{Kuzetal2020Lorenz}.}, it seems impossible to analytically compute it. In \cite{Anikushin2022Semigroups}, it was conjectured that \eqref{EQ: SuarezSchopfCompoundExp} is globally stable in the smaller region determined by the inequality $\lambda_{1} + \lambda_{2} < 0$, where $\lambda_{1} = \lambda_{1}(\alpha,\tau) > 0$ and $\lambda_{2} = \lambda_{2}(\alpha,\tau) < 0$ are the first two (as the real part decreases) characteristic roots, which are always real, at the zero equilibrium $\phi^{0}$. It can be shown that $\lambda_{1} + \lambda_{2} < 0$ is equivalent to the following explicit inequality:
\begin{equation}
	\label{EQ: SSmodelDimensionStabilityRegion}
	\tau < \frac{\log\left( \frac{1+\sqrt{1-\alpha^{2}}}{\alpha}\right)}{\sqrt{1-\alpha^{2}}}.
\end{equation}
The essence of this conjecture is revealed in its stronger form, which requires proving that $\phi^{0}$ is the most unstable point of $\mathcal{A}$ or, in rigorous terms, that the local Lyapunov dimension at $\phi^{0}$ equals to the Lyapunov dimension of $\mathcal{A}$. Such statements are known as the Eden conjecture, see \cite{Anikushin2023LyapExp, KuzReit2020}.

As to the developed apparatus, here \eqref{EQ: SSmodelDimensionStabilityRegion} determines the maximal region of possible applications, since this inequality is always satisfied under \eqref{EQ: LyapExpContrCond}.

A partial answer to the conjecture is given in \cite{Anikushin2022Semigroups} under the additional restriction $2\alpha \tau < 1$. Such a restriction is concerned with the construction of more delicate invariant regions to localize the global attractor $\mathcal{A}$, see Lemma \ref{LEM: SSmodelRadiusEstimateWithRestr}. In \cite{Anikushin2022Semigroups}, a comparison principle with stationary systems is also used. It is based on the monotonicity property of compound cocycles corresponding to monotone feedback systems explored by Mallet-Paret and Nussbaum in \cite{MalletParretNussbaum2013}, the already mentioned Poincar\'{e}--Bendixson trichotomy, and the ergodic variational principle for subadditive families, see \cite{Anikushin2023LyapExp}. However, not all the restricted region $2 \alpha \tau < 1$, which lies strictly within \eqref{EQ: SSmodelDimensionStabilityRegion} for $\alpha \geq 0.23$, is covered by such an approach. Although its part corresponding to $\alpha \in [0.75,1)$ seems to be identical, the part corresponding to $\alpha \in [0.5, 0.75]$ is significantly different, see \cite[Figure 1]{Anikushin2022Semigroups}.

In \cite{Anikushin2023LyapExp}, the Liouville trace formula applied in adapted metrics was used to estimate the Lyapunov dimension of $\mathcal{A}$ from above by $C(\alpha) \tau + 1$, where $C(\alpha) = \alpha^{2}e^{p^{*} + 1}$, and $p^{*}$ is the unique root $p>0$ of $\alpha^{2} p e^{p+1} = 3$. This ensures the global stability for $C(\alpha) \tau < 1$. For example, by taking $\alpha = 0.75$, we have $C(\alpha) \approx 3.555$, and the inequality approximately reduces to $\tau < 0.281$. For such parameters, the inequality $\alpha \tau < 0.211$ always holds. In particular, this method does not even cover the aforementioned result from \cite{Anikushin2022Semigroups}. Below, we intend to apply the developed machinery to improve these results.

Linearization of \eqref{EQ: SuarezSchopfCompoundExp} over a given solution $y_{0}(\cdot) \colon [-\tau,\infty) \to \mathbb{R}$ gives the equation
\begin{equation}
	\label{EQ: Suarez-SchopfLinearized1}
	\dot{x}(t) = (1 - 3 y^{2}_{0}(t))x(t) - \alpha x(t-\tau).
\end{equation}
Given $R > 0$, we set $\Lambda_{R} \coloneq 3R^{2}/2$ and rewrite \eqref{EQ: Suarez-SchopfLinearized1} as follows:
\begin{equation}
	\label{EQ: Suarez-SchopfLinearizedRewriten}
	\dot{x}(t) = (1-\Lambda_{R})x(t) - \alpha x(t-\tau) - (3y^{2}_{0}(t) - \Lambda_{R}) x(t).
\end{equation}
We consider \eqref{EQ: Suarez-SchopfLinearizedRewriten} in the context of \eqref{EQ: DelayRnLinearized} with $r_{1}=r_{2}=n=1$, $\widetilde{A}\phi \coloneq (1-\Lambda_{R})\phi(0) - \alpha \phi(-\tau)$, $C\phi \coloneq \phi(0)$ for $\phi \in C([-\tau,0];\mathbb{R})$, $\widetilde{B} \coloneq -1$, $F'(\wp) \coloneq 3|\wp(0)|^{2} - \Lambda_{R}$ for $\wp \in \mathcal{P} \coloneq \mathcal{A}$, and $\pi$ taken as the restriction of $\varphi$ to $\mathcal{A}$.

Eigenvalues of the operator $A$ corresponding via \eqref{EQ: OperatorAScalarDelayEquations} to the operator $\widetilde{A}$ defined below \eqref{EQ: Suarez-SchopfLinearizedRewriten} are given by the roots $p \in \mathbb{C}$ of
\begin{equation}
	1-\Lambda_{R} - \alpha e^{-\tau p} - p = 0.
\end{equation}
Let $\lambda_{1}(A), \lambda_{2}(A), \ldots$ be the eigenvalues arranged by nonincreasing their real parts and according to their multiplicities. Then the spectral bound of $A^{[\wedge m]}$ is given by $\sum_{j=1}^{m}\operatorname{Re}\lambda_{j}(A)$. 

The following proposition illustrates how global stability criteria can be derived for \eqref{EQ: SuarezSchopfCompoundExp} with the aid of developed machinery.
\begin{proposition}
	\label{PROP: DelayCompoundGlobStabSuarezSchopf}
	Let the global attractor $\mathcal{A}$ of the semiflow $\varphi$ generated by \eqref{EQ: SuarezSchopfCompoundExp} be contained in the ball of radius $R$ centered at $0$. Consider \eqref{EQ: Suarez-SchopfLinearizedRewriten} in the context of \eqref{EQ: DelayRnLinearized} as it is stated below the former. Suppose that there exists $\nu_{0} \geq 0$ such that $\operatorname{Re}\lambda_{1}(A) + \operatorname{Re}\lambda_{2}(A) < - \nu_{0}$, and the frequency inequality\footnote{See Section \ref{SUBSEC: ExampleResolventEquations} for an explicit interpretation of such inequalities.} \eqref{EQ: DelayCompoundSmithIneqGeneral} is satisfied with $m=2$ and $\Lambda = \Lambda_{R}$. Then $\varphi$ is globally stable, i.e., any trajectory converges to one of the three equilibria $\phi^{+}$, $\phi^{-}$, or $\phi^{0}$.
\end{proposition}
\begin{proof}
	According to Remark \ref{REM: UniformBoundResolventDelayCompound}, we can always assume that $\nu_{0} > 0$.
	
	Next, by appropriately choosing a $C^{1}$-truncation with bounded derivative of the nonlinearity $x^{3}$ outside a closed positively invariant ball containing $\mathcal{A}$, say, the closed ball $\mathcal{B}_{R_{0}}(0)$ of radius $R_{0}$ centered at $0$ with $R_{0} \geq \sqrt{1+\alpha}$, we may consider $\varphi$ generated by \eqref{EQ: SuarezSchopfCompoundExp} with the truncated nonlinearity as a semiflow in the Hilbert space $\mathbb{H}$ from \eqref{EQ: DelayEqsHilbertSpace} (with $n=1$). In view of \cite[Theorem 1]{Anikushin2022Semigroups}, this semiflow coincides with the initial semiflow in the ball, and $\mathcal{A}$ is also a compact invariant set for the new $\varphi$. Then by \cite[Theorems 2 and 3]{Anikushin2022Semigroups}, the cocycle $\Xi$ generated by \eqref{EQ: DelayLinearCocAbsract} in our context is the derivative cocycle for $\varphi$ over $\mathcal{A}$. Since the frequency condition is satisfied, we may apply Theorem \ref{TH: QuadraticFunctionalDelayCompoundTheorem} and Corollary \ref{COR: DelayCompoundUniformExponentialStability} to obtain that the twofold multiplicative compound $\Xi_{2}$ of $\Xi$ is uniformly exponentially stable. Then \cite[Theorem 2.1]{ChepyzhovIlyin2004} implies that the fractal dimension\footnote{Due to the smoothing property of $\varphi$ from $\mathbb{H}$ to $\mathbb{E}$, see \cite[Theorem 1]{Anikushin2022Semigroups}, the fractal dimension of $\mathcal{A}$ is the same in any of the metrics induced from $\mathbb{H}$ or $\mathbb{E}$.} of $\mathcal{A}$ is strictly less than $2$. For our purposes, it is also sufficient to use the same estimate for the Hausdorff dimension of $\mathcal{A}$, see \cite[Theorem 3.1, Chapter V]{Temam1997}.
	
	Note that the ball $\mathcal{B}_{R_{0}}(0)$ is convex and invariant with respect to the original $\varphi$, so $\varphi$ is a semiflow in the ball, and $\mathcal{A}$ attracts compact (in fact, even bounded) subsets of it. Moreover, by \cite[Property 4.3, Section 3.4]{Hale1977} or \cite[p.~450]{MalletParetSell1996}, the mapping $\varphi^{t}$ is a homeomorphism of $\mathcal{A}$ for any $t \geq 0$. Now \cite[Corollary 2]{LiMuldowney1995LowBounds} guarantees that $\mathcal{A}$ does not contain closed invariant contours\footnote{Here a ``closed invariant contour'' should be understood as a simple $\delta$-linked $1$-boundary in the terminology of \cite{LiMuldowney1995LowBounds}. It is important that periodic orbits, homoclinic trajectories, and polycycles belong to such a class.}.
	
	Now we utilize the Poincar\'{e}--Bendixson trichotomy, namely \cite[Theorem 2.1]{MalletParetSell1996}, to get the desired conclusion. First, we note that since $\Xi_{2}$ is exponentially stable, the parameters $\alpha$ and $\tau$ must necessarily belong to the region \eqref{EQ: SSmodelDimensionStabilityRegion} lying below the neutral curve. In particular, all characteristic roots at the symmetric equilibria $\phi^{\pm}$ have negative real parts in this case. Then it is sufficient to show that points from the one-dimensional unstable manifold of the zero equilibrium $\phi^{0}$ tend to one of $\phi^{\pm}$. Indeed, since periodic orbits are excluded, any point $\phi_{0} \in \mathbb{E}$ must contain at least one equilibrium in its $\omega$-limit set $\omega(\phi_{0})$ due to the trichotomy. Clearly, in the case of $\phi^{+}$ or $\phi^{-}$, the entire $\omega(\phi_{0})$ must coincide with the equilibrium. 
	
	If $\phi^{0}$ belongs to $\omega(\phi_{0})$ and $\phi_{0}$ does not lie on the stable manifold of $\phi^{0}$ (in which case $\omega(\phi_{0}) = \phi^{0}$), we consider some sequence $t=t_{k}$, where $k=1,2,\ldots,$ for which the point $\varphi^{t_{k}}(\phi_{0})$ tends to $\phi^{0}$ as $k \to \infty$. Due to the hyperbolic behavior\footnote{It is well known that the conjugating homeomorphism in the Hartman--Grobman theorem may fail to exist in infinite-dimensional problems, including delay equations. Here we mean a weaker version of the Hartman--Grobman theorem, which is usually not considered in the literature. It is concerned with the existence of a foliation in a neighborhood of the hyperbolic point. Here the unstable manifold can be considered as an inertial manifold, and the foliation can be constructed by the approach developed in \cite{Anikushin2020Geom}. See the next footnote for a precise statement.} in a small neighborhood of $\phi^{0}$, for sufficiently large $k$ the trajectory of $\varphi^{t_{k}}(\phi_{0})$ leaves the neighborhood sufficiently close to the trajectory of a point from the unstable manifold\footnote{More precisely, there exist a bounded open neighborhood $\mathcal{U}$ of $\phi^{0}$ in $\mathbb{E}$ and constants $M>0$ and $\nu > 0$ such that for any $\phi_{0} \in \mathcal{U}$ there exists a unique point $\phi^{*}_{0}$ from the unstable manifold $\mathcal{W}^{u}(\phi^{0})$ of $\phi^{0}$ in $\mathcal{U}$ such that the inequality
	\begin{equation}
		\| \varphi^{t}(\phi_{0}) - \varphi^{t}(\phi^{*}_{0}) \|_{\mathbb{E}} \leq M \operatorname{dist}(\phi_{0}, \mathcal{W}^{u}(\phi^{0})) e^{-\nu t}
	\end{equation}
	is satisfied as long as the trajectories of both points remain in $\mathcal{U}$. Obviously, the closer $\phi_{0}$ is to $\phi^{0}$, the more time the trajectories spend in $\mathcal{U}$ and, therefore, the closer they become to each other at the exit time $t$.
	}. If trajectories of any such points tend to one of $\phi^{\pm}$, the same can be said about $\phi_{0}$, and we get a contradiction.
	
	Now let $\phi_{0}$ be a point from the unstable manifold of $\phi^{0}$ different from $\phi^{0}$ itself. If $\omega(\phi_{0})$ does not contain any of $\phi^{\pm}$, it must contain a complete trajectory for which $\alpha$- and $\omega$-limit sets coincide with $\phi^{0}$. But such a trajectory, along with $\phi^{0}$, forms a closed invariant contour, the existence of which is forbidden.
\end{proof}

Suppose that $\mathcal{P}=\mathcal{A}$ lies in the ball $\mathcal{B}_{R}(0)$ of radius $R$ centered at $0$ in $C([-\tau,0];\mathbb{R})$. It is clear that $|F'(\wp)| \leq \Lambda_{R}$ for any $\wp \in \mathcal{P}$, where $F'(\wp)$ is defined below \eqref{EQ: Suarez-SchopfLinearizedRewriten}. From this view, we wish to localize $\mathcal{A}$ by a ball with the smallest possible radius $R$. For this, the following estimate is appropriate.
\begin{lemma}\cite[Lemma 4.2]{Anikushin2022Semigroups}
	\label{LEM: SSmodelRadiusEstimateWithRestr}
	Suppose $2 \alpha \tau < 1$, and let $R_{0}=R_{0}(\alpha,\tau)$ be the unique positive root $p > 0$ of $-p^{3} + (1-\alpha) p + C(\alpha,\tau) = 0$, where 
	\begin{equation}
		C(\alpha,\tau) = \frac{4}{3} \cdot\frac{\alpha \tau(1-\alpha)}{ 1 - \alpha \tau} \sqrt{\frac{1-\alpha}{3}}.
	\end{equation}
	Then the global attractor $\mathcal{A}$ of \eqref{EQ: SuarezSchopfCompoundExp} lies in the ball of radius $R_{0}$ centered at $0$.
\end{lemma}

For the radius $R_{0}$ from Lemma \ref{LEM: SSmodelRadiusEstimateWithRestr}, it can be shown that $R_{0} < \sqrt{1+\alpha}$ for $\alpha > 0.233$. Moreover, we clearly have $R_{0} \to \sqrt{1-\alpha}$ as $\alpha \to 1-$. Thus, under the additional restriction\footnote{This restriction can be relaxed if one uses a more accurate estimate for $R_{2}$ in \cite[Lemma 4.2]{Anikushin2022Semigroups} by considering an undetermined constant $\varkappa \geq 0$ instead of $\alpha$ in formula (4.21) therein. For small $\tau$, the value $\varkappa = \alpha$ is optimal, but it decreases to $0$ with increasing $\tau$, in which case the resulting bound reduces to $\sqrt{1+\alpha}$. We do not know an explicit formula for the optimal $\varkappa$.} $2 \alpha \tau < 1$, $R_{0}$ provides a better estimate for the radius of a ball enclosing $\mathcal{A}$.

Now, we illustrate our method by means of concrete parameters. Namely, we take $\alpha = 0.6$ and $\tau = 0.83$. Such parameters satisfy $2 \alpha \tau < 1$ and are not covered by the approach from \cite{Anikushin2022Semigroups}. Here the linear operator $A$ has the leading eigenvalues $\lambda_{1,2}(A) \approx -0.89 \pm i 0.63$. We consider the approximation scheme \nameref{DESC: AS1DelayCompound}--\nameref{DESC: AS4DelayCompound} for \eqref{EQ: Suarez-SchopfLinearizedRewriten} with the given $\alpha$, $\tau$, and $R=R_{0}(\alpha,\tau)$ from Lemma \ref{LEM: SSmodelRadiusEstimateWithRestr}. Parameters of the scheme are taken as $m=2$, $\Lambda \coloneq \Lambda_{R}$, $\nu_{0} = 0.01$, $\Omega = 37.5$, $T \in \{15.77, 25.73\}$, and $N \in \{2, 5, 10, 20, 30\}$. We conduct numerical experiments using a realization of the scheme on Python.
\begin{remark}
	\label{REM: DelayCompoundNumerical}
	For numerical integration of delay equations, we use the JiTCDDE package for Python, see \cite{AnsmannGJitCDDE2018}. Parameters of the integration procedure are taken as $\operatorname{first\_step}=\operatorname{max\_step} = 10^{-5}$, $\operatorname{atol} = 10^{-8}$, and $\operatorname{rtol} = 0$. Numerical solutions are obtained on the time interval $[0,T]$ at points from a uniform grid with the step taken about $h_{0}=10^{-3}$ (see the next footnote). Integrals from \eqref{EQ: DelayCompApproximationScheme31M1Comp} and \eqref{EQ: DelayCompoundASCoefficientApproximateMatrices} are approximated via the Simpson $1/3$-rule using uniform grids with the step about\footnote{Since the most efficient implementation of the Simpson rule demands an odd number of points, it is convenient to choose $T$ as an odd multiple of $\tau$. Then we can use uniform grid partitions of $[-\tau,0]$ and $[-\tau,T]$ by an odd number of points that agree on $[-\tau,0]$. We choose a step $h$ corresponding to such a partition by possibly tweaking (decreasing a bit) $h_{0}$.} $h_{0}=10^{-3}$. The step in $\omega$ is taken as $0.1$. See the repository for more details.
\end{remark}
\begin{figure}[t]
	\begin{minipage}{.5\textwidth}
		\includegraphics[width=\textwidth,angle=0]{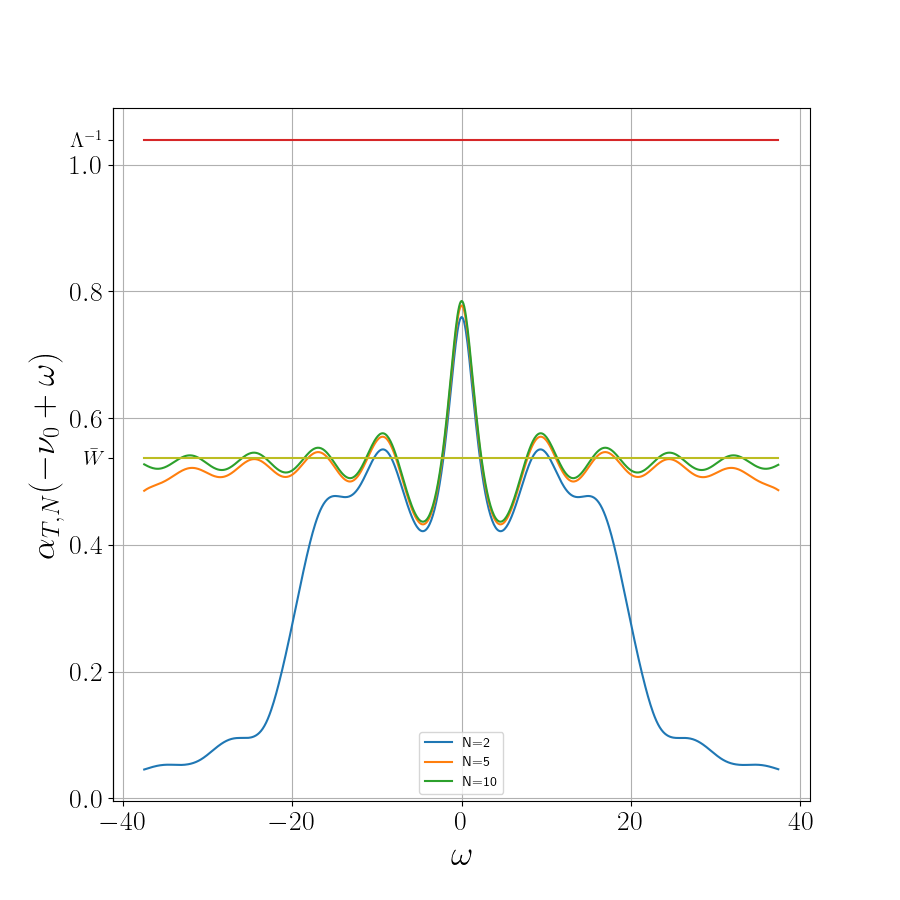}
	\end{minipage}%
	\begin{minipage}{.5\textwidth}
		\includegraphics[width=\textwidth,angle=0]{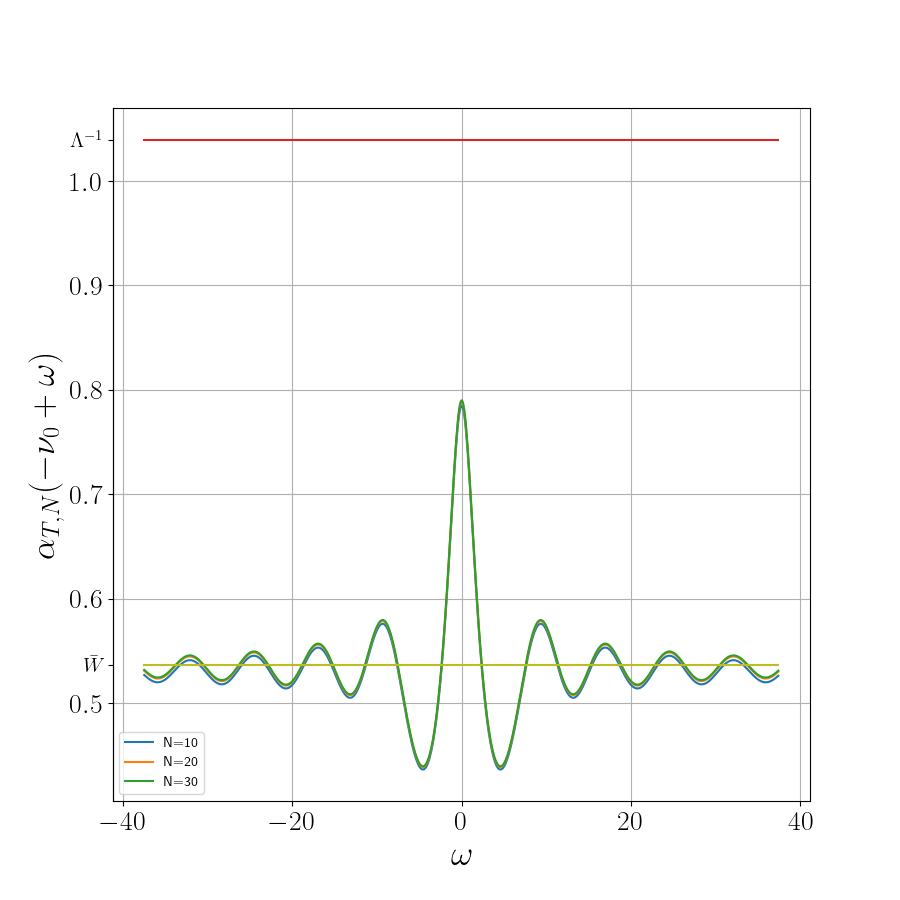}
	\end{minipage}%
	\caption{Graphs of the largest singular values $\alpha_{T,N}(-\nu_{0} + i \omega)$ versus $\omega$ computed via the numerical implementation of the approximation scheme \nameref{DESC: AS1DelayCompound}--\nameref{DESC: AS4DelayCompound} applied to \eqref{EQ: Suarez-SchopfLinearizedRewriten} with $\alpha = 0.6$, $\tau = 0.83$, and $R=R_{0}(\alpha,\tau)$ given by Lemma \ref{LEM: SSmodelRadiusEstimateWithRestr}. Parameters of the scheme are taken as $m=2$, $\Lambda = \Lambda_{R}$, $\nu_{0} = 0.01$, $T=15.77$, $\Omega = 37.5$, and (Left): $N=2$ (blue), $N=5$ (orange), $N=10$ (green), and (Right): $N=10$ (blue), $N=20$ (orange), $N=30$ (green). The horizontal lines pass through the threshold value $\Lambda^{-1}$ (red) and the norm $\bar{W}$ of the asymptotic integral operator (olive) on the vertical axis. See Remark \ref{REM: DelayCompoundNumerical} and the repository for implementation details.}
	\label{FIG: CompoundSST=15}
\end{figure}

For $T = 15.77 = 19\tau$, Figure \ref{FIG: CompoundSST=15} shows graphs of the largest singular values $\alpha_{T,N}(-\nu_{0} + i \omega)$ versus $\omega$ for $N \in \{2,5,10\}$ (left) and $N \in \{10,20,30\}$ (right). For $T= 25.73 = 31\tau$, the conducted experiments give indistinguishable figures. This indicates convergence of the numerical scheme. For $|\omega| > 10$, the curves tend to exhibit an oscillating pattern decaying to a positive value $\bar{W}$ in accordance with Corollary \ref{COR: ExplicitComputationConjecture}. Note also that the results are consistent with Fig. \ref{FIG: CompoundMGTKernelApproximations} (right), where truncations are obtained via the explicit representation of the transfer operator.

Analogous experiments justify the validity of the frequency inequality in the region $2 \alpha \tau < 1$ with $\alpha \in [0.5, 1)$. Consequently, this region is expected to be contained in the region of global stability for \eqref{EQ: SuarezSchopfCompoundExp}. For $\alpha \in (0,0.5)$, the scheme indicates violation of the frequency inequality, but this is only a limitation of the method. We expect that one may improve the result by constructing more delicate subsets enclosing the global attractor.

%% file: MackeyGlassEqs.tex
\subsection{Mackey--Glass equations}
\label{SEC: MackeyGlassCompoundStab}
In this section, we study the following class of nonlinear scalar delay equations suggested by Mackey and Glass in \cite{MackeyGlass1977} as a model for certain physiological processes:
\begin{equation}
	\label{EQ: MackeyGlassExample}
	\dot{x}(t) =  - \gamma x(t) + \beta \frac{ x(t-\tau) }{ 1 + |x(t-\tau)|^{\kappa}  },
\end{equation}
where $\tau, \beta,\gamma > 0$, and $\kappa > 1$ are real parameters. From the physiological perspective, one is interested in the dynamics of \eqref{EQ: MackeyGlassExample} in the cone of positive functions. However, for global analysis, it is convenient to consider the system in the entire space.

Standard arguments show that \eqref{EQ: MackeyGlassExample} generates a dissipative semiflow $\varphi$ in the space $\mathbb{E} = C([-\tau,0];\mathbb{R})$ given by $\varphi^{t}(\phi_{0}) = x_{t}$ for all $t \geq 0$ and $\phi_{0} \in \mathbb{E}$, where $x \colon [-\tau,\infty) \to \mathbb{R}$ is the classical solution to \eqref{EQ: MackeyGlassExample} with $x_{0} = \phi_{0}$. Recall that $x_{t}(\theta) \coloneq x(t+\theta)$ for $\theta \in [-\tau,0]$ denotes the $\tau$-history segment at $t \geq 0$. Consequently, there exists a global attractor $\mathcal{A}$.

In \cite{Anikushin2023LyapExp}, it is shown that for $\beta \leq \gamma$, the global attractor $\mathcal{A}$ of $\varphi$ is given by the zero equilibrium $\phi^{0}(\cdot) \equiv 0$. For $\beta > \gamma$, the global attractor $\mathcal{A}$ lies in the ball of radius $\beta \gamma^{-1} \kappa^{-1} (\kappa-1)^{(\kappa-1)/\kappa}$ centered at $0$, and any ball with a radius not smaller than that is positively invariant with respect to $\varphi$.

It is not hard to see that for $\beta < \gamma$ there is a unique equilibrium $\phi^{0}(\cdot) \equiv 0$ with a negative leading real eigenvalue. For $\beta=\gamma$, the leading eigenvalue becomes zero, and the system undergoes a pitchfork bifurcation at $\phi^{0}$ with a birth of the pair of symmetric equilibria $\phi^{+}$ and $\phi^{-}$ given by $\phi^{\pm}(\cdot) \equiv \pm (\beta \gamma^{-1} - 1)^{1/\kappa}$ for $\beta > \gamma$.

Numerical experiments conducted in \cite{MackeyGlass1977} indicate that the model \eqref{EQ: MackeyGlassExample} may possess chaotic behavior and, consequently, the attractor $\mathcal{A}$ may have rich structure. In particular, chaos is observed for $\gamma = 0.1$, $\beta = 0.2$, $\kappa=10$, and $\tau \geq 10$.

In \cite{Anikushin2023LyapExp}, the Lyapunov dimension of $\mathcal{A}$ is estimated from above by $C(\gamma,\beta,\kappa) \tau + 1$ with some $C(\gamma,\beta,\kappa)>0$. In particular, the estimate implies that the attractor $\mathcal{A}$ does not contain closed invariant contours provided that $\tau < C(\gamma,\beta,\kappa)^{-1}$. For $\beta = 0.2$, $\gamma = 0.1$, and $\kappa=10$, we have $C(\gamma,\beta,\kappa) \approx 0.9957$, and the inequality is close to $\tau \leq 1$. However, it can be verified that for $\tau = \tau^{*}$, where
\begin{equation}
	\label{EQ: MackeyGlassTau0HopfBif}
	 \tau^{*} = \frac{8}{3}\arccos\left(-\frac{1}{4}\right) \approx 4.8626, 
\end{equation}
the leading pair of complex-conjugate characteristic roots at $\phi^{\pm}$ crosses the imaginary axis, and the system undergoes a supercritical Andronov--Hopf bifurcation (in contrast to the Suarez--Schopf model, where the direction is subcritical). We expect the system to be globally stable for $\tau < \tau^{*}$ and conjecture that the same holds for any parameters as follows.
\begin{conjecture}
	\label{CONJ: ConjectureMackeyGlassStability}
	For $\beta > \gamma > 0$, $\tau > 0$ and $\kappa > 1$, the semiflow $\varphi$ generated by \eqref{EQ: MackeyGlassExample} is globally stable provided that the equilibria $\phi^{\pm} = \pm (\beta \gamma^{-1} - 1)^{1/\kappa}$ are linearly stable, i.e., all their characteristic roots have negative real parts.
\end{conjecture}

In other words, the conjecture states that the boundary of global stability in \eqref{EQ: MackeyGlassExample} is determined from the local stability of the symmetric equilibria $\phi^{\pm}$, i.e., it is trivial in the terminology of \cite{Kuzetal2020Lorenz}. This contrasts with the Suarez--Schopf oscillator \eqref{EQ: SuarezSchopfCompoundExp}, where the boundary is hidden. Now we are going to support the conjecture by means of the developed machinery.

First, it is convenient to normalize the delay in \eqref{EQ: MackeyGlassExample} by scaling the time variable $t \rightarrow \tau t$. Then \eqref{EQ: MackeyGlassExample} transforms into the following equation:
\begin{equation}
	\label{EQ: MackeyGlassExampleNormalized}
	\dot{x}(t) =  -\tau \gamma x(t) + \tau \beta \frac{ x(t-1) }{ 1 + |x(t-1)|^{\kappa}  }.
\end{equation}
Linearization of \eqref{EQ: MackeyGlassExampleNormalized} along a given solution $y_{0} \colon [-1,\infty) \to \mathbb{R}$ gives
\begin{equation}
	\label{EQ: MackeyGlassLinearization}
	\dot{x}(t) = -\tau \gamma x(t) + \tau \beta f'( y_{0}(t-1) )x(t-1),
\end{equation}
where $f(y) = y/(1+|y|^{\kappa})$ for $y \in \mathbb{R}$, and $f'(y)$ is the derivative of $f$ at $y$. Straightforward calculations show that $-\frac{(\kappa-1)^{2}}{4\kappa} \leq f'(y) \leq 1 \text{ for any } y \in \mathbb{R}$. From this, we rewrite \eqref{EQ: MackeyGlassLinearization} as follows: (here $y_{0,t}$ is the $1$-history segment of $y_{0}(\cdot)$ at $t$)
\begin{equation}
	\label{EQ: MackeyGlassNumericalSchemeExample}
	\dot{x}(t) = -\tau \gamma x(t) + (\tau \beta - \Lambda) x(t-1) + F'(y_{0,t})x(t-1),
\end{equation}
where $F'(\phi) = \tau \beta (f'(\phi(-1)) - 1) + \Lambda$ for $\phi \in C([-1,0];\mathbb{R})$, and $\Lambda$ is given by
\begin{equation}
	\label{EQ: MackeyGlassStabilityLambdaDef}
	\Lambda = \frac{1}{2} \tau \beta \left( \frac{(\kappa-1)^{2}}{4\kappa} + 1 \right).
\end{equation}
It is clear that $|F'(\phi)| \leq \Lambda$.

We consider \eqref{EQ: MackeyGlassNumericalSchemeExample} in the context of \eqref{EQ: DelayRnLinearized} with $r_{1}=r_{2}=n=1$, $\widetilde{A}\phi \coloneq -\tau \gamma\phi(0) + (\tau \beta - \Lambda) \phi(-1)$, $C\phi \coloneq \phi(-1)$ for $\phi \in C([-1,0];\mathbb{R})$, $\widetilde{B} \coloneq 1$, $F'(\wp)$ defined above for $\wp \in \mathcal{P} \coloneq C([-\tau,0];\mathbb{R})$, and $\pi \coloneq \varphi$.

Eigenvalues of the operator $A$ corresponding via \eqref{EQ: OperatorAScalarDelayEquations} to the operator $\widetilde{A}$ as above are given by the roots $p \in \mathbb{C}$ to
\begin{equation}
	-\tau\gamma + (\tau \beta - \Lambda) e^{-p} - p = 0.
\end{equation}
Let $\lambda_{1}(A), \lambda_{2}(A), \ldots$ be the eigenvalues arranged by nonincreasing their real parts and according to their multiplicities. Then the spectral bound of $A^{[\wedge m]}$ is given by $\sum_{j=1}^{m}\operatorname{Re}\lambda_{j}(A)$. 

We have the following analog of Proposition \ref{PROP: DelayCompoundGlobStabSuarezSchopf}, which gives a criterion for the absence of closed invariant contours on $\mathcal{A}$.
\begin{proposition}
	\label{PROP: MackeyGlassGlobStability}
	Let $\varphi$ be the semiflow generated by \eqref{EQ: MackeyGlassExample}. Consider \eqref{EQ: MackeyGlassNumericalSchemeExample} in the context of \eqref{EQ: DelayRnLinearized} as it is stated below the former. Suppose that there exists $\nu_{0} \geq 0$ such that $\operatorname{Re}\lambda_{1}(A) + \operatorname{Re}\lambda_{2}(A) < - \nu_{0}$, and the frequency inequality\footnote{See Section \ref{SUBSEC: ExampleResolventEquations} for an explicit interpretation of such inequalities.} \eqref{EQ: DelayCompoundSmithIneqGeneral} is satisfied with $m=2$ and $\Lambda$ given by \eqref{EQ: MackeyGlassStabilityLambdaDef}. Then the global attractor $\mathcal{A}$ of $\varphi$ does not contain closed invariant contours\footnote{Recall that a ``closed invariant contour'' should be understood as a simple $\delta$-linked $1$-boundary in the terminology of \cite{LiMuldowney1995LowBounds}.} on which $\varphi^{t}$ is bijective for some $t > 0$.
\end{proposition}
\begin{proof}
	Similarly to the proof of Theorem \ref{PROP: DelayCompoundGlobStabSuarezSchopf}, we get that the Hausdorff dimension of $\mathcal{A}$ is strictly less than $2$.
	
	Now let $\mathcal{B}_{R_{0}}(0)$ be the ball of radius $R_{0}$ centered at zero. As discussed above, for any $R_{0} \geq \beta \gamma^{-1} \kappa^{-1} (\kappa-1)^{(\kappa-1)/\kappa}$, the attractor $\mathcal{A}$ lies in $\mathcal{B}_{R_{0}}(0)$, and the ball is positively invariant. Then the conclusion follows from \cite[Corollary 2]{LiMuldowney1995LowBounds} by modulo that the statement therein requires $\varphi^{t}$ to be bijective on $\mathcal{A}$, but in the proof it is used only that $\varphi^{t}$ is bijective on the closed invariant contour as in \cite[Corollary 1]{LiMuldowney1995LowBounds}.
\end{proof}
\begin{remark}
	\label{REM: MackeyGlassGlobStabViaRobust}
	Under the conditions of Proposition \ref{PROP: MackeyGlassGlobStability}, we in fact have the robust condition
	\begin{equation}
		\lambda_{1}(\Xi) + \lambda_{2}(\Xi) \leq -\nu_{0} < 0
	\end{equation}
    for the first and the second uniform Lyapunov exponents $\lambda_{1}(\Xi)$ and $\lambda_{2}(\Xi)$ of the derivative cocycle $\Xi$ over $(C([-\tau,0];\mathbb{R}),\varphi)$, see Remark \ref{REM: RobustnessLargLyapExp}. Thus, as in finite dimensions, such a condition is expected to guarantee global stability.
\end{remark}

Let us illustrate the method for the classical parameters $\gamma = 0.1$, $\beta = 0.2$, $\kappa = 10$, and $\tau = 4.5$. Here the leading pair of eigenvalues satisfies $\lambda_{1,2}(A) \approx -0.99 \pm i1.12$. We consider the approximation scheme \nameref{DESC: AS1DelayCompound}--\nameref{DESC: AS4DelayCompound} for \eqref{EQ: MackeyGlassNumericalSchemeExample} with the given $\gamma$, $\beta$, $\tau$, $\kappa$, and $\Lambda$ given by \eqref{EQ: MackeyGlassStabilityLambdaDef}. Parameters of the scheme are taken as $m=2$, $\Lambda$ as above, $\nu_{0} = 0.01$, $\Omega = 37.5$, $T \in \{15, 25\}$, and $N \in \{2, 5, 10, 20, 30\}$. We conduct numerical experiments using a realization of the scheme on Python, see Remark \ref{REM: DelayCompoundNumerical}.

\begin{figure}[t]
	\begin{minipage}{.5\textwidth}
		\includegraphics[width=\textwidth,angle=0]{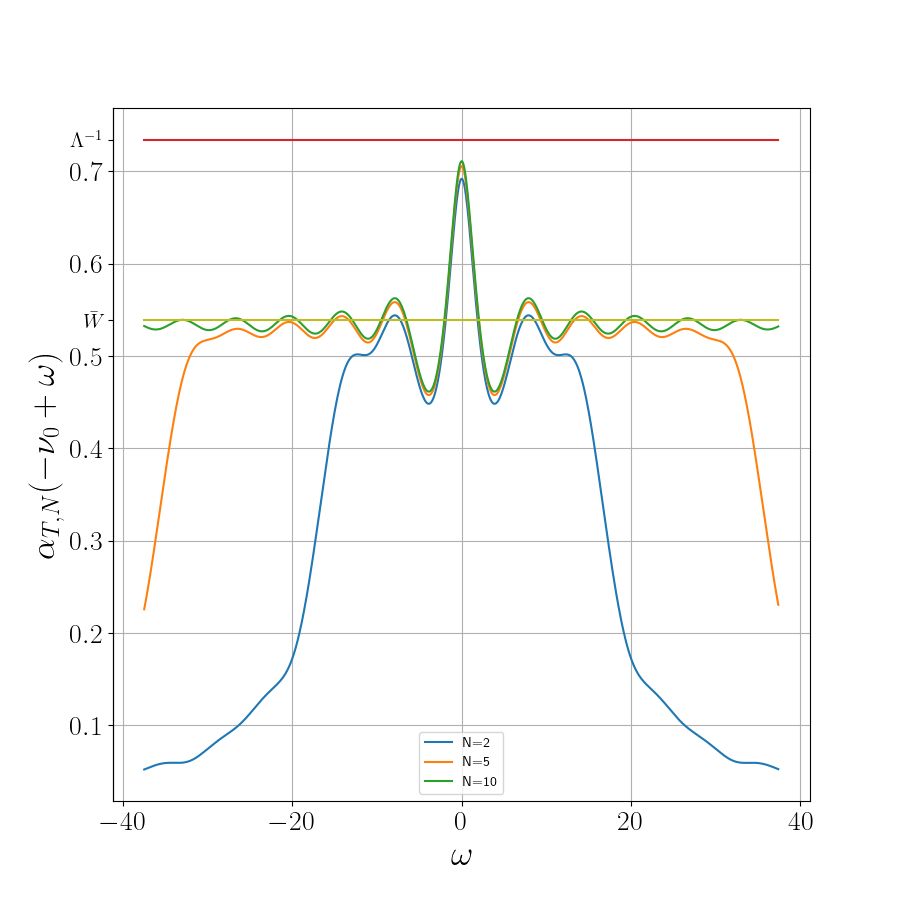}
	\end{minipage}%
	\begin{minipage}{.5\textwidth}
		\includegraphics[width=\textwidth,angle=0]{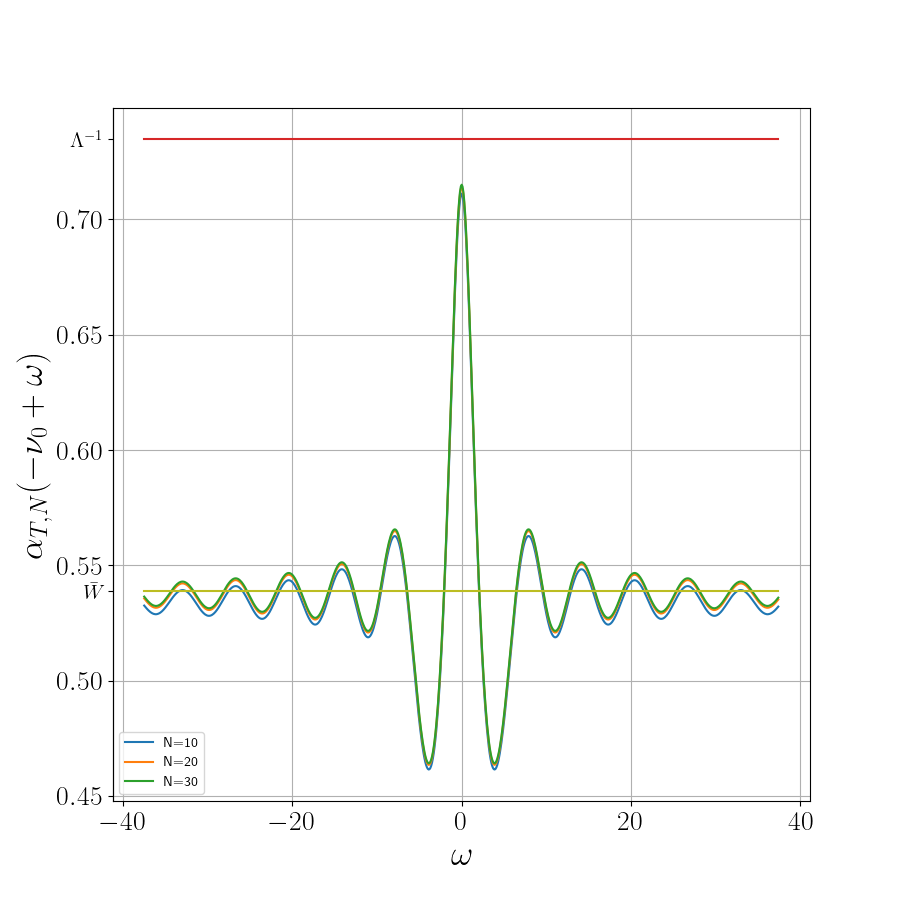}
	\end{minipage}%
	\caption{Graphs of the largest singular values $\alpha_{T,N}(-\nu_{0} + i \omega)$ versus $\omega$ computed via the numerical implementation of the approximation scheme \nameref{DESC: AS1DelayCompound}--\nameref{DESC: AS4DelayCompound} applied to \eqref{EQ: MackeyGlassNumericalSchemeExample} with $\gamma = 0.1$, $\beta = 0.2$, $\kappa = 10$, $\tau = 4.5$, and $\Lambda$ given by \eqref{EQ: MackeyGlassStabilityLambdaDef}. Parameters of the scheme are taken as $m=2$, $\Lambda$ as above, $\nu_{0} = 0.01$, $T=15$, $\Omega = 37.5$, and (Left): $N=2$ (blue), $N=5$ (orange), $N=10$ (green), and (Right): $N=10$ (blue), $N=20$ (orange), $N=30$ (green). The horizontal lines pass through the threshold value $\Lambda^{-1}$ (red) and the norm $\bar{W}$ of the asymptotic integral operator (olive) on the vertical axis. See the repository for implementation details.}
	\label{FIG: CompoundMGT=15}
\end{figure}

For $T=15$, Figure \ref{FIG: CompoundMGT=15} shows graphs of the largest singular values $\alpha_{T,N}(-\nu_{0} + i \omega)$ versus $\omega$ for $N \in \{2,5,10\}$ (left) and $N \in \{10,20,30\}$ (right). For $T=25$, the conducted experiments give indistinguishable figures. This indicates convergence of the numerical scheme. For $|\omega| > 10$, the curves tend to exhibit an oscillating pattern decaying to a positive value $\bar{W}$ in accordance with Corollary \ref{COR: ExplicitComputationConjecture}. Note also that the results are consistent with Fig. \ref{FIG: CompoundMGTKernelApproximations} (left), where truncations are obtained via the explicit representation of the transfer operator.

In fact, the numerical scheme indicates that the frequency inequality is valid even for $\tau = 4.6$, but the graphs come too close to the threshold line in the experiments\footnote{Using the optimization of Schur test functions as in \cite{AnikushinRomanov2025Schur}, one can verify the frequency inequality up to $\tau = 4.55$.}. Analogous experiments justify the validity of the frequency inequality for $\tau \in (0, 4.5]$. This indicates the absence of closed invariant contours in the system for such parameters and, as is expected, the global stability, see Remark \ref{REM: MackeyGlassGlobStabViaRobust}. Moreover, we are quite surprised, since the method is in a sense rough, that the achieved result turned out to be very close to the desirable one determined by the bifurcation parameter $\tau^{*} \approx 4.8626$ from \eqref{EQ: MackeyGlassTau0HopfBif}. We consider this as another indicator that Conjecture \ref{CONJ: ConjectureMackeyGlassStability} should be valid.

Our method can be compared with the more delicate result of Liz, Tkachenko, and Trofimchuk \cite{Lizetal2003} generalizing the well-known Myshkis stability criterion to nonlinear scalar delay equations with a single equilibrium. It often provides boundaries of global stability that are close to the boundary of linear stability (the Nicholson blowflies model considered in \cite{Lizetal2003} is a nice illustration), and it is also based on a comparison with a linear system. It can be applied to \eqref{EQ: MackeyGlassExample} in the invariant cone of positive functions. Omitting (possibly nontrivial) justifications of the applicability\footnote{In fact, this is justified by \cite[Theorem 3.2]{LizetAl2005}, but the reader should be careful, since condition (iii) of the theorem contains an error related to an incorrect calculation of the derivative.} of \cite[Corollary 2.3]{Lizetal2003} in our situation (for $\gamma = 0.1$, $\beta = 0.2$, and $\kappa = 10$), we obtain that such a criterion would guarantee the global stability for $\tau < -10 \left[\ln 4 + \ln\ln(20/17)\right] \approx 4.3066$ that is smaller than our bound. Thus, the frequency criterion can also complement even such results, which significantly rely on some specificity of scalar equations. In fact, the criterion from \cite{Lizetal2003} is optimal in the class of time-dependent delays, and our method is more specific to constant delays.

For the Nicholson blowflies model, the frequency criterion cannot compete with the result of \cite{Lizetal2003}, but it complements previously known results based on more rough methods, see \cite{AnikushinRomanov2024EffEst} for a comparison.